\UseRawInputEncoding

\documentclass[12pt,a4paper]{amsart}
\usepackage{lmodern}
\usepackage[utf8]{inputenc}
\usepackage{bbm}
\usepackage{tikz}
\usepackage{utfsym}
\usepackage{mathrsfs}
\usepackage{fix-cm}
\usepackage{cite}
\usetikzlibrary{matrix}
\usepackage{tikz-cd}
\usepackage{fontawesome5}

\newcommand{\rk}{\mathrm{rk}}
\usepackage[a4paper,left=1.8cm,right=1.8cm,top=2.5cm,bottom=2.5cm]{geometry}

\usepackage{dsfont}         
                          
\usepackage{float}
\usepackage{xurl}
\usepackage{ragged2e}
\usepackage{xcolor}
 
\usepackage{enumitem}
\usepackage{float}

\setlist{nosep}
\setlist[enumerate]{label=(\roman*)}
 
\usepackage{etoolbox}
\AtBeginEnvironment{thebibliography}{%
  \setlength{\itemsep}{0pt}%
  \setlength{\parskip}{0pt}%
  \setlength{\topsep}{0pt}%
  \setlength{\partopsep}{0pt}%
}
 
\newlist{defenum}{enumerate}{3}
\setlist[defenum]{label=(\alph*),leftmargin=*,align=left}
\newlist{defsubenum}{enumerate}{1}
\setlist[defsubenum]{label=(\roman*),leftmargin=*,align=left}

\usepackage{mathtools}
\usepackage{amssymb,amsthm}
\usepackage{graphicx}
\usepackage[cal=rsfso,scaled=1.05]{mathalfa}

\usepackage[english]{babel}
\usepackage[T1]{fontenc}
\usepackage{tikz-cd}
\usepackage{bm}

\usepackage[most]{tcolorbox}

\usepackage{array,booktabs,tabularx,multirow}

\usepackage[colorlinks=true,pagebackref,hypertexnames=false,linkcolor=red,citecolor=blue,filecolor=blue,urlcolor=blue]{hyperref}

\usepackage[nameinlink,noabbrev]{cleveref}
\usepackage{autonum}
\usepackage[alphabetic,backrefs]{amsrefs}

\makeatletter
\renewcommand\normalsize{%
    \@setfontsize\normalsize{11.7}{14pt plus .3pt minus .3pt}%
    \abovedisplayskip 10\p@ \@plus4\p@ \@minus4\p@
    \abovedisplayshortskip 6\p@ \@plus2\p@
    \belowdisplayshortskip 6\p@ \@plus2\p@
    \belowdisplayskip \abovedisplayskip}
\renewcommand\small{%
    \@setfontsize\small{9.5}{12\p@ plus .2\p@ minus .2\p@}%
    \abovedisplayskip 8.5\p@ \@plus4\p@ \@minus1\p@
    \belowdisplayskip \abovedisplayskip
    \abovedisplayshortskip \abovedisplayskip
    \belowdisplayshortskip \abovedisplayskip}
\renewcommand\footnotesize{%
    \@setfontsize\footnotesize{8.5}{9.25\p@ plus .1pt minus .1pt}%%
    \abovedisplayskip 6\p@ \@plus4\p@ \@minus1\p@
    \belowdisplayskip \abovedisplayskip
    \abovedisplayshortskip \abovedisplayskip
    \belowdisplayshortskip \abovedisplayskip}

\ifdefined\pdfpagewidth
\else
\fi
\calclayout
\makeatother
\newenvironment{standing}{\begin{quote}\itshape}{\end{quote}}

\newtheorem{theorem}{Theorem}[section]
\newtheorem{Mtheorem}{Theorem}

\newtheorem{corollary}[theorem]{Corollary}

\newtheorem{lemma}[theorem]{Lemma}

\newtheorem{proposition}[theorem]{Proposition}
\theoremstyle{definition}
\newtheorem{definition}[theorem]{Definition}
\newtheorem{example}[theorem]{Example}   % already upright: keep as is
\theoremstyle{remark}
\newtheorem{remark}[theorem]{Remark}

\crefname{theorem}{Theorem}{Theorems}
\crefname{Mtheorem}{Main Theorem}{Main Theorems}
\crefname{lemma}{Lemma}{Lemmata}
\crefname{corollary}{Corollary}{Corollaries}
\crefname{proposition}{Proposition}{Propositions}
\crefname{definition}{Definition}{Definitions}
\crefname{example}{Example}{Examples}    

\AtEndEnvironment{example}{\exend}

\newcommand{\Q}{\mathbb{Q}}
\newcommand{\R}{\mathbb{R}}
\newcommand{\Z}{\mathbb{Z}}
\newcommand{\C}{\mathbb{C}}
\newcommand{\N}{\mathbb{N}}
\newcommand{\PP}{\mathbb{P}}

\newcommand{\End}{\mathrm{End}}

\newcommand{\tr}{\mathrm{tr}}
\newcommand{\rank}{\mathrm{rk}}

\newcommand{\SL}{\mathrm{SL}}

\newcommand{\rH}{\mathrm{H}}
\newcommand{\Gr}{\mathrm{Gr}}

\newcommand{\cE}{\mathcal{E}}

\newcommand{\cO}{\mathcal{O}}

\newcommand{\cQ}{\mathcal{Q}}

\newcommand{\cU}{\mathcal{U}}

\newcommand{\Id}{\mathbf{1}}

\newcommand{\Aut}{\mathrm{Aut}}

\newcommand{\ev}{\mathrm{ev}}

\newcolumntype{L}{>{$}l<{$}}
\newcolumntype{C}{>{$}c<{$}}
\newcolumntype{R}{>{$}r<{$}}

\lstdefinestyle{gwflags}{
  basicstyle=\ttfamily\small,
  columns=fullflexible,
  keepspaces=true,
  showstringspaces=false,
  morecomment=[l]{\#},
  commentstyle=\itshape\color{black!55},
  aboveskip=0pt, belowskip=0pt,
  frame=none
}

\NewTCBListing{inputbox}{ o }{
  listing only,
  listing style=gwflags,
  colback=black!3, colframe=black!80,
  boxrule=0.8pt, arc=0pt, outer arc=0pt,
  left=8pt, right=8pt, top=6pt, bottom=6pt,
  fonttitle=\bfseries, coltitle=black,
  attach boxed title to top left={yshift=-2mm, xshift=4mm},
  boxed title style={colback=white, colframe=white},
  IfValueT={#1}{title={#1}}
}

\NewTColorBox{outputbox}{ o }{
  breakable,
  colback=white, colframe=black!80,
  boxrule=0.8pt, arc=0pt, outer arc=0pt,
  left=8pt, right=8pt, top=8pt, bottom=8pt,
  fonttitle=\bfseries, coltitle=black,
  attach boxed title to top left={yshift=-2mm, xshift=4mm},
  boxed title style={colback=white, colframe=white},
  IfValueT={#1}{title={#1}}
}
 
\newcommand{\exend}{\hfill$\diamond$}
 
\begin{document}
\title[On some atomic decompositions in flag varieties]{On the atomic decomposition of complete intersection in flag varieties}

\author[Alexeev]{Boris Alexeev}
%\address{Institute of Mathematics and Informatics, Bulgarian Academy of Sciences}
\email{boris.alexeev@gmail.com}

\author[Cavenaghi]{Leonardo F. Cavenaghi}
\address{Institute of Mathematics and Informatics, Bulgarian Academy of Sciences}
\email{leonardofcavenaghi@gmail.com}

\author[Galindo]{Giovane Galindo}
\address{Institute of Mathematics and Informatics, Bulgarian Academy of Sciences}
\email{giovanepgneto@gmail.com}

\author[Georgiev]{Bogdan Georgiev}
% \address{Google DeepMind}
\email{bogdan.m.georgiev@gmail.com}

\author[Katzarkov]{Ludmil Katzarkov}
\address{Simons Center for Geometry and Physics, Stony Brook University, Stony Brook, NY 11794, USA. \& International Center for Mathematical Sciences (ICMS), Sofia, Bulgaria. \& Institute of Mathematics and Informatics, Bulgarian Academy of Sciences, Sofia, Bulgaria.} 
\email{lkatzarkov@gmail.com}

% \author[Kontsevich]{Maxim Kontsevich}
% \address{Institut des Hautes Études Scientifiques, 35 route de Chartres, F - 91440 Bures-sur-Yvette}
% \email{maxim@ihes.fr}

\author[Martins]{Pedro Antonio Muniz Martins}
\address{Instituto de Matemática, Estatística e Computação Científica (IMECC) da Universidade Estadual de Campinas (Unicamp), Cidade Universitária, Campinas - SP, 13083-856, Brazil.} 
\email{pedroa.muniz9@gmail.com}

\begin{abstract}
We develop a root-theoretic localization formalism for genus-zero
Gromov--Witten invariants of flag varieties and of smooth zero loci of
globally generated homogeneous vector bundles.  The resulting invariants
are expressed as finite sums over decorated trees whose contributions are
determined by the root system of the ambient flag variety and by the torus
weights of the defining bundle. We connect these computations with the Theory of Hodge Atoms of Katzarkov--Kontsevich--Pantev--Yu
via the matrix of small quantum multiplication by first Chern class.

Let $X$ be a
Fano fourfold with $h^{3,1}(X)=1$ arising as a hyperplane section of a
suitable Fano fivefold.  Its cohomology decomposes into a monodromy-fixed
part and the middle vanishing cohomology. This decomposition is preserved by quantum multiplication by the first Chern class. Moreover, it acts by a scalar on the vanishing summand.  We show, for the monodromy-fixed block, that (a) if every eigenvalue has algebraic multiplicity at most two and $X$ is Hodge general, then $X$
is irrational; (b)  if every eigenvalue has Jordan defect at most one and $X$ is rational, then every weak factorization contains a smooth surface center whose minimal model is a projective K3 surface. Many applications are presented.

Using the recently introduced quantum connection approach to irrationality, we obtain from the computations here presented the irrationality of the very general smooth complete intersection of a hyperplane and a quadric in $\operatorname{LG}(3,6)$; the very general smooth
quadric section in $G_2/P_2$; and the very general
smooth complete intersection of five hyperplanes and a quadric in the spinor 10-fold. 
\end{abstract}

\maketitle
\tableofcontents

\section{Introduction}
\label{sec:intro}

\subsection*{Prelude}

Flag varieties possess a distinguished combination of characteristics
placing them in a convenient position at the intersection of algebraic
geometry, representation theory, and enumerative geometry: their geometry
is controlled by root systems and Weyl groups; their cohomology admits
distinguished Schubert bases; and the action of a maximal torus has finitely
many fixed points and finitely many one-dimensional orbit closures.
 In particular, with their standard maximal-torus
actions, flag varieties are examples of
\emph{GKM\footnote{Goresky--Kottwitz--MacPherson} spaces} \cite{Guillemin2006}.  Together, these
properties make flag varieties, and subvarieties obtained from them by
homogeneous vector bundles, particularly well suited to equivariant
localization in Gromov--Witten theory, beginning with Kontsevich's
fixed-point calculations \cite{Kontsevich1994} and the virtual-localization
formalism developed in
\cites{Graber1999,Liu_Chiu,holmes2025computationsequivariantgromovwittentheory}.
The purpose of this paper is to connect this explicit combinatorial
geometry with the Theory of Hodge Atoms of
Katzarkov--Kontsevich--Pantev--Yu~\cite{KKPY}, offering applications in birational geometry.

\medskip

First, we develop a root-theoretic localization formalism for genus-zero
primary Gromov--Witten invariants of flag varieties and of smooth zero loci
of globally generated homogeneous bundles.  The formulas are sufficiently
explicit to be implemented algorithmically. Using the fact that the required inputs can be determined from the root-theoretic and
representation-theoretic data of the flag variety and of the homogeneous
bundle, this paper is accompanied by software which computes genus-zero
Gromov--Witten invariants on flag varieties and, through the genus-zero
quantum Lefschetz/functoriality theorem of Kim--Kresch--Pantev
\cite{Kim2003}, the corresponding invariants of smooth homogeneous zero
loci with insertions restricted from the ambient flag variety.

We note that Holmes and Muratore
\cite{holmes2025computationsequivariantgromovwittentheory} have provided a
computational framework for analogous calculations in general GKM spaces,
implemented in
\href{https://mgemath.github.io/GKMtools.jl/stable/}{GKMtools.jl}.
Because our setting is restricted to flag varieties and homogeneous zero
loci therein, the necessary geometric input admits a direct description in
terms of roots, Weyl groups, Schubert classes, and bundle weights, as
explained in Appendix \ref{app:software}.  This specialization considerably
simplifies the input and the corresponding implementation.  To connect
these calculations with the Theory of Hodge Atoms, our software also
evaluates the flag-ambient matrix of small quantum multiplication by the
first Chern class and its characteristic polynomial. We point out that in GKMtools.jl (specifically, in \href{https://mgemath.github.io/GKMtools.jl/stable/GW/QH/#Quantum-product-with-c_1T(TX)}{GKMtools.jl Integration}), the functionality to produce the matrix of the quantum product was available prior to this work. We notice two distinctions\footnote {Authors thank D. Holmes and G. Muratore for clarifying this.}: the implementation in GKMtools.jl returns the matrix in the context of equivariant cohomology. Thus, after a change of basis, it coincides with the matrix produced by our software. Second, the computation in equivariant cohomology tends to be slower\footnote{Authors have not tested our process with that of GKMtools.jl, so the reader should take this sentence as a heuristic communicated by D. Holmes and G. Muratore}. 

\medskip

Using a coarse version of Hodge atoms, as developed in
\cite{CKK-lectures}, and inspired in particular by the results of Gu\'er\'e
\cite{JeremyCubic} and of
Benedetti--Fay--Gu\'er\'e--Manivel--Perrin
\cite{Benedetti_Fay_Guere_Manivel_Perrin}, we obtain theoretical
consequences concerning the rationality of Fano fourfolds.  Our examples
recover, by the criteria developed here, the irrationality of
Hodge-general, and hence very general, cubic fourfolds
(Example~\ref{ex:cubic4fold}) and ordinary Gushel--Mukai fourfolds
(Example~\ref{ex:GM4fold}).  We also obtain a uniform
K3-center constraint for rational fourfolds of K3 type \cite{Bernardara2026}, inspired by the stronger cubic-specific theorem of Gu\'er\'e.

\medskip

We now describe the main constructions and results more precisely.

\medskip

\subsection*{Localization formulae}

Let $F=G/P$ be a flag variety, where $G$ is a connected simply connected semisimple
complex Lie group and $P$ is a parabolic subgroup.  Fix a maximal torus $\mathbb T\subset B\subset P$ and a positive root system $R^+$ with corresponding simple roots $\Pi$.
If $\Pi_P\subset\Pi$ determines $P$, we set
\[
\Pi_F:=\Pi\setminus\Pi_P,
\qquad
R_F^+:=R^+\setminus R_P^+.
\]
The root-theoretic data determine the basic geometry needed in the
localization calculation.  For instance, Proposition~\ref{prop:dimension} says that $\dim F=|R_F^+|$, while the Schubert classes form a
homogeneous basis (Proposition~\ref{prop:schubertbasis}) for the rational Betti cohomology of $F$:
\[
\rH^*(F,\Q)
=
\bigoplus_{w\in W_F}\Q\,\sigma_w.
\]
 Here, $W$ is the Weyl group, $W_P\subset W$ is the subgroup generated by the simple reflections
$s_\alpha$ with $\alpha\in\Pi_P$, and $W_F$ denotes the set of
minimal-length representatives of $W/W_P$.  We also have
\[
\rH_2(F,\Z)\cong Q_F^\vee,
\]
where $Q_F^\vee$ is the reduced coroot lattice.  Other relevant quantities are also
determined by the root system; see, e.g., Proposition~\ref{prop:picard},
Proposition~\ref{prop:h2},
Corollary~\ref{cor:effective}, and
Proposition~\ref{prop:chern}.

The $\mathbb T$-action on $F$ induces a $\mathbb T$-action on the moduli space of stable maps $\overline{\mathcal M}_{0,m}(F,\beta)$ such that its fixed-point set's connected components are indexed by decorated trees: vertices are labelled by
$\mathbb T$-fixed points of $F$, equivalently, by elements of $W_F$, while
edges correspond to $\mathbb T$-invariant rational curves and carry the
corresponding root-theoretic data together with a positive covering
degree.  We compute the equivariant Euler class of the virtual normal
bundle directly from these data and obtain an explicit localization
formula for genus-zero Gromov--Witten invariants of $F$ (Theorem \ref{thm:localization}).

\medskip

We then pass to a smooth zero locus $Z=\mathscr Z(F,\cE)\subset F$ of a regular section of a globally generated homogeneous vector bundle
$\cE$. If $-\mu_1,\dots,-\mu_s$ are the $\mathbb T$-weights of the fiber of $\cE$ over the base fixed
point, the quantum Lefschetz contribution along a decorated tree is
completely determined by these weights and their pairings with the roots
labelling the edges.  In particular, under the hypotheses considered
here, the localization formula depends on the homogeneous bundle through
its torus-weight multiset. The bundles used in this paper, together with the corresponding input data
for localization, are recorded in Appendix~\ref{app:bundles}.

\medskip

The main localization statement can be summarized as follows.

\begin{Mtheorem}[=Theorem~\ref{thm:localizationCI}]
\label{thm:A}
Let $\cE$ be a globally generated homogeneous bundle on the flag variety
$F=G/P$, and let $Z=\mathscr Z(F,\cE)$ be the smooth zero locus of a regular section.  Let $\gamma_k=\iota^*\sigma_{u_k},
~
u_k\in W_F$ and $\beta\in\mathsf{NE}(F)$. Assume that
$(m,\beta)$ is stable and that
\[
\sum_{k=1}^{m}\ell(u_k)
=
\dim Z+\int_\beta\left(c_1(TF)-c_1(\cE)\right)+m-3.
\]
Then,
\[
\sum_{\substack{\beta'\in\mathsf{NE}(Z)\\
                 \iota_*\beta'=\beta}}
\big\langle
\gamma_1,\dots,\gamma_m
\big\rangle^Z_{0,\beta'}
\]
is given by a finite sum over decorated trees of type
$(m,\beta)$ whose contributions are determined by the root-theoretic data
of $F$, the combinatorics and degrees of the tree, and the weights of
$\cE$.
\end{Mtheorem}

Set $\rH_F^*(Z)
:=
\operatorname{im}
\left(
\rH^*(F,\Q)\longrightarrow\rH^*(Z,\Q)
\right)$. The localization formula developed here allows for the computation of the projection of small
quantum multiplication by $c_1(TZ)$ to this \emph{flag-ambient}
cohomology.  More precisely, when the restricted Poincar\'e pairing is
nondegenerate, Theorem \ref{thm:A} allows us to compute
\[
A(q)
=
\operatorname{pr}_F
\circ
\big(c_1(TZ)\star_{\mathrm{small}}(-)\big)
\Big|_{\rH_F^*(Z)}.
\]
It is through this very matrix, specialized to $q=1$ in all applications, that we obtain irrationality criteria. Of relevance for comparison with our results, we also point out \cite{bohning2026naiveatomsblowupsexamples}.

\subsection*{Coarse Hodge atoms}

The Theory of Atoms produces a new approach to questions in birational geometry
which stem from the combination of Hodge theory and Gromov--Witten theory. The foundational
manuscript of Katzarkov--Kontsevich-Pantev-Yu~\cite{KKPY} introduced \emph{Hodge atoms} and established (among many other results) that a very general cubic fourfold is not rational. Within a year, there was a research
literature\footnote {It is completely unfair to list only such a small amount of works. The theory of atoms has brought with itself different contributions in many directions through many interesting papers; we have just selected those which are directly related to this work.}: Gu\'er\'e~\cite{JeremyCubic} proved that the primitive cohomology of
\emph{any} rational smooth cubic fourfold is isomorphic, as a Hodge structure (up to Tate twist),
to the middle cohomology of a projective K3 surface; Benedetti--Manivel--Perrin~\cite{Benedetti_Vladimiro_Manivel} established the
irrationality of very general Gushel--Mukai fourfolds and similarly obtained a result \emph{a la} Gu\'er\'e;
Benedetti--Gu\'er\'e--Manivel--Perrin~\cite{BGMP} settled birational questions
for Verra fourfolds; Fay~\cite{Fay} obtained
\emph{equivariant} irrationality for Verra fourfolds; an atomic irrationality
criterion bypassing quantum computations now reaches K\"uchle
fourfolds~\cite{Benedetti_Fay_Guere_Manivel_Perrin}. In the equivariant direction,
Cavenaghi--Katzarkov--Kontsevich~\cite{AtomsMeetSymbols} extended the theory to varieties
with finite group actions and to orbifolds. B\"ohning--von Bothmer--Su'a \cite{bohning2026naiveatomsblowupsexamples} defined naive (and small) atomic decompositions of smooth projective varieties satisfying a version of Iritani's blowup formula \cite{iritani2025quantumcohomologyblowups}. Here, we continue this development. Bridging with the localization methods developed in the first half of this paper, we offer ``matrix'' and ``numerical'' criteria for determining non-rationality, as we summarize below. First, let us gather the basic setup from atom theory needed therein.

\medskip

The original Theory of Hodge Atoms is formulated in terms of the A-model $F$-bundle.  For this paper, we use a
coarser structure, retaining the base of the A-model as a parameter space, the quantum product, and the Euler vector field $\mathsf{Eu}=c_1(TZ)+\ldots$ \cite{CKK-lectures}. After extending coefficients from $\Q$ to the non-Archimedean field
$\Bbbk$ used in \cite{KKPY}, we build an analytic domain $B_Z$ (in the sense of Berkovich) on which
the genus-zero Gromov--Witten potential is an element of its structure sheaf.  For every $\Bbbk$-rational point
$b\in B_Z(\Bbbk)$, there is an endomorphism $\boldsymbol\kappa_b
:=
\mathsf{Eu}_b\star_b(-)
\colon
\rH^*(Z,\Bbbk)\longrightarrow\rH^*(Z,\Bbbk).$ Regard $V:=\rH^*(Z,\Q)$ as its natural graded polarizable rational Hodge structure.  Let
$\mathsf{MT}(V)$ denote its Mumford--Tate group with Tate character $f\colon\mathsf{MT}(V)\longrightarrow\mathbb G_m$. The \emph{Hodge group} used in this paper is
$\mathsf{Hod}
:=
\ker f$. Its invariant vectors are the rational Hodge classes in $\rH^*(Z,\Q)$.  The group $\mathsf{Hod}$ defines an action on $B_Z$ with fixed locus $B_Z^{\mathsf{Hod}}$.  At every point $b\in B_Z^{\mathsf{Hod}}$ the
quantum product and the Euler vector field are
$\mathsf{Hod}_{\Bbbk}$-equivariant.  Consequently,
$\boldsymbol\kappa_b$ commutes with the Hodge group, and each of its
generalized eigenspaces is naturally an
$\mathsf{Hod}_{\Bbbk}$-representation. Let $D(b)
:=
\#\operatorname{spec}(\boldsymbol\kappa_b)$ be the number of distinct eigenvalues.  The locus $\mathcal V
:=
\left\{
b\in B_Z^{\mathsf{Hod}}(\Bbbk)
\ \middle|\
D(b)=
\max_{a\in B_Z^{\mathsf{Hod}}(\Bbbk)}D(a)
\right\}$ is the set of $\Bbbk$-points of a nonempty analytic open subset.  The
\emph{coarse Hodge atoms} are the isomorphism classes, as
$\mathsf{Hod}_{\overline{\Q}}$-representations, of the generalized
eigenspaces $E_{b,\lambda},
~
b\in\mathcal V$.

The points at which our localization calculation becomes directly
relevant are the \emph{small points}.  By definition, these are points
$b\in B_Z(\Bbbk)$ for which all non-Novikov coordinates vanish.  Every
small point belongs to the Hodge locus and $\mathsf{Eu}_b=c_1(TZ)$. Therefore, $\boldsymbol\kappa_b
=
c_1(TZ)\star_b(-)$, where $\star_b$ is the small quantum product evaluated at the chosen Novikov parameters.

\subsection*{Obstructions to rationality}
Let $F$ be a flag variety with $\mathrm{Pic}(F)\cong\Z^d$ and $\cE$ be a homogeneous vector bundle on $F$. Set
\[
L
:=
\cO_F(\underbrace{1,\dots,1}_{d\text{ times}}),
\]
\[
Y:=\mathscr Z(F,\cE),
\qquad
X:=\mathscr Z(F,\cE\oplus L)=Y\cap H,
\qquad
H\in|L|.
\]
Assume that $Y$ and $X$ are smooth of dimensions $n+1$ and $n$,
respectively, with $n\ge3$ and that
\begin{standing}\label{stand:intro}
\begin{defenum}
\item[\textup{(H0)}]
the Picard lattices are identified
\[
\mathrm{Pic}(F)
\xrightarrow{\sim}
\mathrm{Pic}(Y)
\xrightarrow{\sim}
\mathrm{Pic}(X),
\]
together with the dual numerical curve lattices used for the Novikov
variables
\item[\textup{(H1)}]
for every $k<n$, $\rH^k(F,\Q)\xrightarrow{\sim}\rH^k(Y,\Q)$
\item[\textup{(H2)}]
when $n=4$, the group $\rH^4(Y,\Q)$ is entirely of Hodge type $(2,2)$
\item[\textup{(H3)}]
$Y$ is Fano.

\end{defenum}
\end{standing}
Write $j\colon X\hookrightarrow Y,~
\iota_Y\colon Y\hookrightarrow F,~
\iota=\iota_Y\circ j$ and set:
\[
\rH_F^*(X)
:=
\iota^*\rH^*(F,\Q),
\qquad
\rH_Y^*(X)
:=
j^*\rH^*(Y,\Q).
\]
We term the first graded ring above the flag-ambient part. The second is the part fixed by the monodromy of a Lefschetz pencil of hyperplane sections of $Y$.  Let $B(q)$ denote the matrix of small quantum
multiplication by $c_1(TX)$ restricted to the monodromy-fixed part
$\rH_Y^*(X)$.  Since $X$ is Fano, only finitely many curve classes can
occur in each entry, so every entry of $B(q)$ is a polynomial in the
Novikov monomials. We denote by $B^{\mathrm{alg}}(1)$ the algebraic specialization $q=1$. Likewise, $A^{\mathrm{alg}}(1)$ denotes the corresponding algebraic specialization of the flag-ambient matrix computed by localization. 

\begin{Mtheorem}[=Theorem~\ref{thm:irrat-from-A}]
\label{thm:B}
Assume (H0)--(H3). Suppose that $n=4$, that
$X$ is Fano with $h^{3,1}(X)=1$, and that $X$ is Hodge general\footnote{i.e., $\rH^4(X,\Q)^{\mathrm{Hdg}}
=
j^*\rH^4(Y,\Q).$}.  If every eigenvalue of $B^{\mathrm{alg}}(1)$ has algebraic multiplicity at most two, then $X$ is irrational.

If, in addition, \[
\rH^4(F,\Q)
\xrightarrow{\sim}
\rH^4(Y,\Q),
\] it is enough to verify this
condition on the computable flag-ambient matrix $A^{\mathrm{alg}}(1)$.
\end{Mtheorem}

\begin{Mtheorem}[=Theorem~\ref{theo:rational_Fano_K3}]
\label{thm:C}
Assume (H0)--(H3). Suppose that $n=4$ and
that $X$ is Fano with $h^{3,1}(X)=1$. Assume that every eigenvalue of $B^{\mathrm{alg}}(1)$ has Jordan defect $m_a-m_g\le1$, where $m_a$ and $m_g$ denote algebraic and geometric multiplicity.  If
$X$ is rational, then every weak factorization
\[
X\dashrightarrow\PP^4
\]
contains a smooth surface center whose minimal model is a projective K3
surface. Under the assumption that \[
\rH^4(F,\Q)
\xrightarrow{\sim}
\rH^4(Y,\Q),
\] the Jordan-defect hypothesis may be checked on
$A^{\mathrm{alg}}(1)$.
\end{Mtheorem}
A concrete application of Theorem \ref{theo:rational_Fano_K3} is given in Example \ref{ex:FanoK3-29}.

\medskip

We also obtain a numerical criterion from the theorem of Benedetti--Fay--Gu\'er\'e--Manivel--Perrin \cite{Benedetti_Fay_Guere_Manivel_Perrin}.  This criterion does not require
a quantum computation.

\begin{Mtheorem}[=Corollaries~\ref{cor:numerical} and
\ref{cor:numerical-combinatorial}]
\label{thm:D}
Assume (H0)--(H3). Suppose that $n=4$, that
$X$ is Fano and Hodge general, and that $h^{3,1}(X)=1$. If $b_4(Y)\le h^{2,2}(X)-20$, then $X$ is irrational. Equivalently if
\[
\rk\,\rH^4(X,\Q)^{\mathrm{Hdg}}
\le
h^{2,2}(X)-20.
\]

If, in addition, \[
\rH^4(F,\Q)
\xrightarrow{\sim}
\rH^4(Y,\Q),
\] the criterion is given by
\[
b_4(F)\le h^{2,2}(X)-20.
\]
\end{Mtheorem}

\medskip

Next, we discuss the examples presented in this paper which showcase both the effectiveness and limitations of these results. After this, we will present other applications related to quantum differential equations \cite{cai2026irrationalfamiliesfanoweighted,carazolli2026quantum,cai2026cubicthreefoldsymplecticallyirrational} that can also be accessed through the computational framework developed here.

\medskip

For cubic fourfolds, the flag-ambient and monodromy-fixed subspaces
coincide, and the localization algorithm recovers the known matrix of
small quantum multiplication by $c_1(TX)$.  The zero eigenvalue of this
matrix has algebraic multiplicity two and geometric multiplicity one.
Consequently, Theorem~\ref{thm:B} recovers the irrationality of
Hodge-general cubic fourfolds, and in particular of very general ones.
Furthermore, Theorem~\ref{thm:C} applies to every smooth cubic fourfold. 

For ordinary Gushel--Mukai fourfolds, the corresponding small quantum operator has four simple nonzero
eigenvalues and a two-dimensional semisimple kernel.  Consequently, both
the multiplicity bound of Theorem~\ref{thm:B} and the Jordan-defect bound
of Theorem~\ref{thm:C} are satisfied.  These conditions recover the
irrationality of Hodge-general ordinary Gushel--Mukai fourfolds and yield
the corresponding K3-center constraint for hypothetical rational members
of the ordinary Gushel--Mukai family.  This is compatible with the
stronger case-specific results of Benedetti--Manivel--Perrin \cite{Benedetti_Vladimiro_Manivel}.

The K\"uchle fourfold of type $(c5)$ exhibits qualitatively different
behavior, showing the necessity of distinguishing between the
flag-ambient and monodromy-fixed subspaces.  For
\[
F=\Gr(3,7)
\]
one has
\[
b_4(F)=2,
\qquad
b_4(Y)=4.
\]
Consequently, the $6\times6$ matrix produced by localization computes only
the projected flag-ambient operator and does not determine the action on
the two additional fixed $(2,2)$-classes inherited from the K\"uchle
fivefold.  Hence Theorems~\ref{thm:B} and~\ref{thm:C} cannot be applied
to K\"uchle $(c5)$ from this matrix alone.  The numerical criterion,
however, is satisfied exactly at equality:
\[
b_4(Y)
=
4
=
h^{2,2}(X)-20,
\]
recovering the irrationality of Hodge-general K\"uchle fourfolds of type
$(c5)$ established in
\cite{Benedetti_Fay_Guere_Manivel_Perrin}.

\medskip

\subsection*{New irrationality results in dimension $4$}

In a recent series of works \cite{cai2026irrationalfamiliesfanoweighted,carazolli2026quantum,cai2026cubicthreefoldsymplecticallyirrational}, a new method to obtain irrationality has been presented, using a tangential version of the theory of Hodge atoms. Roughly speaking, this follows from understanding the flat sections of the connection in the disk direction for the A-model F-bundle \cite[\S3.5]{KKPY}. Our described A-model retains only partial information (there is no ``disk direction''), yet, without entering into technicalities, out of the content of this paper we have enough to compute the numerical quantities used, for example, in \cite{carazolli2026quantum}, to obtain the following:

\begin{Mtheorem}[=Theorem~\ref{thm:irrational_E_V5}]
   Let $\operatorname{LG}(3,6)$ be the Lagrangian Grassmannian parameterizing three-dimensional Lagrangian subspaces of a six-dimensional symplectic vector space. A very general smooth complete intersection of a hyperplane and a quadric in $\operatorname{LG}(3,6)$ is irrational.
\end{Mtheorem}

\begin{Mtheorem}[=Theorem~\ref{thm:irrational_E_V9}]
    Let $G_2/P_2\subset\mathbb P^{13}$ be the five-dimensional adjoint $G_2$ variety, embedded as the closed
$G_2$-orbit in $\mathbb P(\mathfrak g_2)$. Then a very general smooth quadric section of it is irrational.
\end{Mtheorem}

\begin{Mtheorem}[=Theorem~\ref{thm:irrational_E_V7}]
    Let $\operatorname{OGr}^{+}(5,10)\subset\mathbb P^{15}$ be the spinor tenfold, i.e., the identity connected component of the orthogonal Grassmannian of maximal isotropic five-dimensional subspaces of a ten-dimensional quadratic vector space. A very general smooth complete intersection of five hyperplanes
and a quadric in it is irrational.
\end{Mtheorem}

\medskip

\subsection*{Software and computational assistance}

The first implementation of the localization algorithm developed in this
paper was written in Wolfram Mathematica, motivated in part by the
availability of the package
``LieART 2.0 -- A Mathematica Application for Lie Algebras and
Representation Theory'' \cite{Feger_2020}. Algorithms and coding were completely done independently by the authors, without referring to other sources.
The implementation was subsequently ported to SageMath with the help of publicly available large language models (LLMs).  The mathematical formulas,
the choice of algorithms, and the verification of the outputs reported in
this paper remain the responsibility of the authors. No part of this text was generated by any LLMs. However, Grammar, orthography, and spelling checks were done via free versions of Grammarly and Overleaf AI model features.

The SageMath implementation that accompanies this paper, together
with instructions for reproducing the computations appearing in the
examples, is presented in \href{https://leonardofcavenaghi.github.io/atomicdecompositions/}{https://leonardofcavenaghi.github.io/atomicdecompositions/}. This website, its source code catalog generation scripts, and its documentation have been fully generated using Google Antigravity (AGY), an advanced autonomous AI coding assistant, and Codex.

Lastly, we remark that an updated version of the software has also been
developed in collaboration with Milena Brockhof, together with a broader
computational catalog of small quantum multiplication matrices.  This
extended implementation now also covers toric varieties and complete
intersections in toric varieties.

\medskip

\section*{Acknowledgments}

   L.~F.~Cavenaghi, L.~Katzarkov, and G.~Galindo are supported by the Simons Foundation, grant SFI-MPS-T-Institutes-00007697, and the Ministry of Education and Science of the Republic of Bulgaria, grant DO1-239/10.12.2024. 

   L.~Katzarkov is supported by the Simons Investigators Award (no.003136), NSF FRG grant DMS-2245099, and NSF FRG grant DMS-2255039.
   
   P.~A.~M.~Martins is supported by FAPESP grants no. 2024/07684-7, 2025/22312-1. 

   The authors are in serious debt to M.~Kontsevich and~J.~Gu\'er\'e for many fruitful discussions and important comments towards the development of this paper.

\bigskip

\section{Flag varieties}\label{sec:flags}

Flag varieties form a special class of homogeneous spaces whose geometry can
be described combinatorially.  In this section, we detail how this
combinatorial structure characterizes the geometric quantities used
throughout this paper.  While (part of) the material is standard in Lie theory, we
give a comprehensive review to fix the notation and to lay the
groundwork for the computations of the later sections, which rely on it
heavily. On the other hand, the authors believe many of the explicit geometric
constructions appearing here are not given elsewhere, so we are confident
others will benefit.

\begin{definition}
A \emph{Borel subgroup} of a complex Lie group $G$ is a maximal connected
solvable closed algebraic subgroup. A \emph{parabolic subgroup} is a closed subgroup
containing a Borel subgroup.  A \emph{flag variety} is a complex homogeneous
space $G/P$, where $G$ is a connected simply connected semisimple complex
Lie group and $P\subset G$ is parabolic.
\end{definition}

Classic examples of flag varieties include the projective spaces
$\PP^{n}=\mathrm{SL}_{n+1}(\C)/P$, where
\[
P=
\left\{
\begin{bmatrix}
a & * & \cdots & * \\
0 &   &        &   \\
\vdots & & A & \\
0 &   &        &
\end{bmatrix}
\in \mathrm{SL}_{n+1}(\C)
\right\},
\]
and the Grassmannians $\Gr(k,n)=\mathrm{SL}_{n}(\C)/P_{k}$, where
\[
P_k=
\left\{
\begin{bmatrix}
A & B \\
0 & D
\end{bmatrix}
\in \mathrm{SL}_n(\C)
\;\middle|\;
A\in \mathrm{GL}_k(\C),\; D\in \mathrm{GL}_{n-k}(\C)
\right\}.
\]

\medskip

\subsection{Root space decomposition}

In this subsection, we review the root space decomposition of Lie algebras and apply it to classify flag varieties.  For a comprehensive treatment, we refer the reader to \cites{alekseevsky1997flag,humphreys2012introduction, arvanitogeorgos2003introduction}.

\medskip

Let $G$ be a connected simply connected semisimple complex Lie group.
Then $G$ contains a maximal torus $\mathbb T\cong(\C^{*})^{r}$, where $r$ is
the rank of $G$.  Denote by $\mathfrak g$ and $\mathfrak t$ the Lie algebras
of $G$ and $\mathbb T$.  The roots of $\mathfrak g$ are the linear
functionals giving the eigenvalues of the adjoint action of the Cartan
subalgebra $\mathfrak t$.  Precisely:

\begin{definition}
A nonzero linear functional $\alpha\in\mathfrak t^{*}$ is a \emph{root} of
$\mathfrak g$ if the corresponding \emph{root space}
\[
    \mathfrak{g}_\alpha := \{X\in \mathfrak{g} \mid \mathrm{ad}_H(X)=\alpha(H)X
      \text{ for all } H\in \mathfrak{t}\}
\]
is nonzero.  We write $R=R(\mathfrak g,\mathfrak t)$ for the set of all
roots and call it the \emph{root system}.  It gives the \emph{root space
decomposition}
\[
    \mathfrak{g} = \mathfrak{t} \oplus \bigoplus_{\alpha \in R} \mathfrak{g}_\alpha .
\]
\end{definition}

\begin{definition}
Given $R$, a choice of \emph{positive roots} is a decomposition
$R=R^{+}\sqcup R^{-}$ such that $R^{-}=-R^{+}$ and such that
$\alpha+\beta\in R^{+}$ whenever $\alpha,\beta\in R^{+}$ and
$\alpha+\beta\in R$.  The elements of $R^{+}$ are the \emph{positive roots},
those of $R^{-}$ the \emph{negative roots}.
\end{definition}

\begin{definition}\label{def:simpleroots}
Let $R^{+}$ be a choice of positive roots.  A subset
$\Pi=\{\alpha_{1},\dots,\alpha_{r}\}\subset R^{+}$ is a set of \emph{simple
roots} if
\begin{defenum}
\item $\Pi$ is a basis of $\mathfrak t^{*}$, and
\item every $\alpha\in R$ is uniquely of the form
      $\alpha=\sum_{i=1}^{r}n_{i}\alpha_{i}$ with $n_{i}\in\Z$, all
      $n_{i}\ge0$ for $\alpha\in R^{+}$ and all $n_{i}\le0$ for
      $\alpha\in R^{-}$.
\end{defenum}
Such a set exists and is unique: it consists of the positive roots that are
not sums of two positive roots.
\end{definition}

\begin{example}\label{ExRoots}
Let $\mathfrak g=\mathfrak{sl}_{n}(\C)$.  The Cartan subalgebra is the space
of diagonal traceless matrices,
\[
    \mathfrak{t} = \Big\{ \operatorname{diag}(a_1,\dots,a_n) \;\Big|\;
      a_i \in \C, \; \textstyle\sum_{i=1}^n a_i = 0 \Big\},
\]
and $\mathfrak t^{*}$ is spanned by the forms $\varepsilon_{i}$ defined by
$\varepsilon_{i}(\operatorname{diag}(a_{1},\dots,a_{n}))=a_{i}$.  The
tracelessness imposes $\sum_{i}\varepsilon_{i}=0$, reflecting
$\dim_{\C}\mathfrak t^{*}=n-1$, so here $r=n-1$.

The root system is
\[
    R=\{\alpha_{ij}=\varepsilon_{i}-\varepsilon_{j}\mid 1\le i\neq j\le n\},
\]
with $\mathfrak g_{\alpha_{ij}}=\C E_{ij}$, where $E_{ij}$ has a single
nonzero entry, equal to $1$, in position $(i,j)$. Since $i\neq j$, we have
$\tr(E_{ij})=0$, so indeed $E_{ij}\in\mathfrak{sl}_{n}(\C)$.

Taking $i<j$ defines a positive system
$R^{+}=\{\alpha_{ij}\mid 1\le i<j\le n\}$.  Setting
$\alpha_{i}:=\alpha_{i,i+1}=\varepsilon_{i}-\varepsilon_{i+1}$, every
positive root is a telescoping sum
\[
    \alpha_{ij}=(\varepsilon_i-\varepsilon_{i+1})+\dots+(\varepsilon_{j-1}-\varepsilon_j)
    =\sum_{k=i}^{j-1}\alpha_{k},
\]
so every positive root is a nonnegative integral combination of
$\alpha_{1},\dots,\alpha_{n-1}$. These are $n-1$ linearly independent roots,
hence a basis of $\mathfrak t^{*}$, and the simple roots are
\[
    \Pi=\{\alpha_{1},\dots,\alpha_{n-1}\}.
\]
\end{example}

We now construct the Borel and parabolic subgroups and use them to classify flag varieties.

\medskip

Fix a Cartan subalgebra $\mathfrak t\subset\mathfrak g$, let
$R=R(\mathfrak g,\mathfrak t)$ and choose a positive system
$R^{+}\subset R$, which determines $\Pi\subset R^{+}$ as in
Definition \ref{def:simpleroots}.  From these data we construct the
\emph{standard Borel subalgebra} associated with $R^{+}$,
\begin{equation}\label{eq:canonicalborel}
    \mathfrak{b} = \mathfrak{t} \oplus \bigoplus_{\alpha \in R^+}\mathfrak{g}_\alpha .
\end{equation}
Since a connected Lie subgroup is determined by its Lie subalgebra, there is a unique connected subgroup $B\subset G$ with Lie algebra $\mathfrak b$. It is a Borel subgroup, called the \emph{standard Borel subgroup} associated with $R^{+}$.  We use this correspondence between connected subgroups and subalgebras repeatedly and without further comment.

\begin{theorem}[\cite{humphreys2012introduction}]\label{thm:borel-conjugate}
Any two Borel subalgebras of $\mathfrak g$ are conjugate under the adjoint
action of $G$. Equivalently, by the correspondence above, any two Borel
subgroups of $G$ are conjugate.
\end{theorem}

By Theorem \ref{thm:borel-conjugate}, every parabolic subgroup contains a conjugate of $B$, and replacing $P$ by $gPg^{-1}$ does not change the isomorphism class of $G/P$. Hence, any flag variety is isomorphic to a quotient $G/P$ with $P\supset B$. Such a $P$ is called a \emph{standard parabolic subgroup}.  We now classify these, and with them the flag varieties, by the combinatorics of simple roots \cite{alekseevsky1997flag}.

\medskip

Given $\Pi_{P}\subset\Pi$, let
\[
    R_{P}=R\cap\operatorname{span}_{\Z}(\Pi_{P})
\]
be the root subsystem generated by $\Pi_{P}$, and let
$R_{P}^{-}=R_{P}\cap R^{-}$.  Adjoining the corresponding negative root
spaces to $\mathfrak b$ gives the subspace
\begin{equation}\label{parabolic}
    \mathfrak{p}
    = \mathfrak{b} \oplus \bigoplus_{\alpha \in R_P^-} \mathfrak{g}_{\alpha}
    = \mathfrak{t} \oplus \bigoplus_{\alpha \in R^+}\mathfrak{g}_\alpha
      \oplus \bigoplus_{\alpha \in R_P^-} \mathfrak{g}_{\alpha}\;\supset\;\mathfrak b .
\end{equation}

To see that $\mathfrak p$ is a subalgebra, it suffices to check that it is closed under the bracket. First,
$[\mathfrak t,\mathfrak g_{\beta}]=\mathfrak g_{\beta}$ for
$\beta\in R^{-}_{P}$.  Second, for $\beta,\beta'\in R^{-}_{P}$ the sum
$\beta+\beta'$, if a root, lies again in $R^{-}_{P}$ because $R_{P}$ is a
root subsystem.  Finally let $\alpha\in R^{+}$ and $\beta\in R^{-}_{P}$, write $\alpha =\sum_{\gamma_i \in \Pi}a_i \gamma_i$ and $\beta =\sum_{\gamma_i \in \Pi_P} -b_i \gamma_i$ , for some $a_i, b_i \geq 0$, then 
\begin{equation}
    \alpha+\beta = \sum_{\gamma_i \in \Pi_P} (a_i-b_i )\gamma_i{+}\sum_{\gamma_i \in \Pi\setminus\Pi_P} a_i\gamma_i .
\end{equation}
For $\alpha+\beta$ to be a root, either it is positive,  $\alpha + \beta \in R^+ \subset \mathfrak{p}$, or it is negative, in which case we would have $(a_i -b_i )\leq 0$ for every $\gamma_i \in \Pi_P$ and $a_i =0$ for $\gamma_i \in \Pi/\Pi_P$; therefore $\alpha+\beta \in R_{P}^-$.

The connected subgroup $P\subset G$ with Lie algebra $\mathfrak p$ is
therefore a standard parabolic subgroup, and $G/P$ is a flag variety.
Conversely, let $\mathfrak q\supset\mathfrak b$ be a subalgebra.  Since
$\mathfrak q\supset\mathfrak t$, it is stable under the adjoint action of
$\mathfrak t$ and is therefore the sum of $\mathfrak t$ with the root spaces it
contains.  It contains every $\mathfrak g_{\alpha}$ with $\alpha\in R^{+}$,
so it is determined by the set of $\alpha\in R^{+}$ with
$\mathfrak g_{-\alpha}\subset\mathfrak q$. This set is closed with the
addition of roots; therefore, it is equal to $R^{+}_{P}$ for
$\Pi_{P}=\{\alpha\in\Pi\mid\mathfrak g_{-\alpha}\subset\mathfrak q\}$, and
$\mathfrak q=\mathfrak p$. Different subsets clearly give distinct
subalgebras. 

\begin{example}[Flag varieties of $\mathrm{SL}_{n+1}(\C)$]\label{ExSln}
By Example \ref{ExRoots}, $\mathfrak{sl}_{n+1}(\C)$ has simple roots
$\Pi=\{\alpha_{1},\dots,\alpha_{n}\}$.  Every subset $\Pi_{P}\subset\Pi$
defines a standard parabolic subgroup and a flag variety, although distinct
subsets may yield varieties isomorphic as abstract algebraic varieties, as
for dual Grassmannians.

The choice $\Pi_{P}=\emptyset$ gives $P=B$ and the full flag variety
$\mathrm{SL}_{n+1}(\C)/B=\mathrm{Fl}(1,2,\dots,n+1)$, parameterizing the
nested subspaces $V_{1}\subset V_{2}\subset\dots\subset V_{n}\subset\C^{n+1}$
with $\dim_{\C}V_{i}=i$.

More generally, if $\Pi\setminus\Pi_{P}=\{\alpha_{k_{1}},\dots,\alpha_{k_{m}}\}$
then $G/P$ is the variety of partial flags
$\mathrm{Fl}(k_{1},\dots,k_{m},n+1)$, that is of the nested subspaces
$V_{k_{1}}\subset\dots\subset V_{k_{m}}\subset\C^{n+1}$ with
$\dim_{\C}V_{k_{i}}=k_{i}$.  In particular $\PP^{n}$ corresponds to
$\Pi\setminus\Pi_{P}=\{\alpha_{1}\}$ and $\Gr(k,n+1)$ to
$\Pi\setminus\Pi_{P}=\{\alpha_{k}\}$. This identifies the subgroups $P$ and
$P_{k}$ displayed above as parabolic.
\end{example}

\medskip

For a flag variety $Y=G/P$ defined by $\Pi_{P}\subset\Pi$ we write
$\Pi_{Y}=\Pi\setminus\Pi_{P}$, and we extend the subscript to the positive
roots whose \emph{opposite} root spaces do not lie in $\mathfrak p$,
\[
    R^{+}_{Y}=R^{+}\setminus R^{+}_{P},
\]
these are the roots that index the tangent directions of $Y$ at the base
point.  As an immediate application of \eqref{parabolic}, we compute the dimension of $Y$.

\begin{proposition}\label{prop:dimension}
Let $Y=G/P$ be the flag variety corresponding to $\Pi_{P}\subset\Pi$.  Then
\[
    \dim_{\C}Y = |R^{+}_{Y}|.
\]
\end{proposition}

\begin{proof}
Since $Y=G/P$ is homogeneous, its dimension equals that of its tangent
space at the base point $eP$, which is canonically
$\mathfrak g/\mathfrak p$.  Comparing the root space decomposition
\[
    \mathfrak{g} = \mathfrak{t} \oplus \bigoplus_{\alpha \in R^+} \mathfrak{g}_\alpha
      \oplus \bigoplus_{\alpha \in R^-} \mathfrak{g}_\alpha
\]
with the description \eqref{parabolic} of $\mathfrak p$, the quotient kills
$\mathfrak t$, all positive root spaces, and the negative root spaces
indexed by $R^{-}_{P}$, leaving
\[
    \mathfrak{g} / \mathfrak{p} \cong \bigoplus_{\alpha \in R^- \setminus R_P^-} \mathfrak{g}_\alpha .
\]
Each root space is one-dimensional, so
$\dim_{\C}(\mathfrak g/\mathfrak p)=\big|R^{-}\setminus R^{-}_{P}\big|$.
Multiplication by $-1$ is a bijection $R^{-}\to R^{+}$ carrying $R^{-}_{P}$
onto $R^{+}_{P}$, hence $R^{-}\setminus R^{-}_{P}$ onto
$R^{+}\setminus R^{+}_{P}=R^{+}_{Y}$.  Therefore
\[
    \dim_{\C}Y=\big|R^{-}\setminus R^{-}_{P}\big|=|R^{+}_{Y}|. 
\]
\end{proof}

\medskip

\subsection{The Weyl group and cohomology}

As argued, the root system associated with $\mathfrak g$ encodes a large
part of the geometry of a flag variety.  Different geometric questions,
however, call for different repackagings of the same data.  We now describe
the Weyl group and explain how it encodes the cohomology of a flag
variety. We then describe the second homology group in terms of coroots, which we 
describe next.

\begin{definition}\label{def:weylgroup}
Let $\mathfrak g$ be a complex semisimple Lie algebra, $\mathfrak t$ a Cartan
subalgebra, $R$ the root system and $\Pi$ a fixed choice of simple roots.
We write
\[
  \langle\,\cdot\,,\,\cdot\,\rangle\colon \mathfrak t^{*}\times\mathfrak t\to\C,
  \qquad \langle\lambda,H\rangle:=\lambda(H),
\]
for the evaluation pairing.  The Killing form $\mathsf B$ restricts to a
nondegenerate form on $\mathfrak t$, hence induces an isomorphism
$\nu\colon\mathfrak t\to\mathfrak t^{*}$, $\nu(H)=\mathsf B(H,\cdot)$, and we
transport $\mathsf B$ to a symmetric bilinear form on $\mathfrak t^{*}$ by
\[
  (\lambda,\mu):=\mathsf B\big(\nu^{-1}\lambda,\nu^{-1}\mu\big).
\]
For $\alpha\in R$ the \emph{coroot} is
\[
  \alpha^{\vee}:=\nu^{-1}\!\left(\frac{2\alpha}{(\alpha,\alpha)}\right)\in\mathfrak t,
  \qquad\text{so that}\qquad
  \langle\beta,\alpha^{\vee}\rangle=\frac{2(\alpha,\beta)}{(\alpha,\alpha)}
  \quad\text{for all }\beta\in\mathfrak t^{*}.
\]
The \emph{root reflection} $s_{\alpha}\in\mathrm{GL}(\mathfrak t)$ is
\[
  s_{\alpha}(H)=H-\langle\alpha,H\rangle\,\alpha^{\vee},\qquad H\in\mathfrak t,
\]
the reflection fixing $\ker\alpha$ pointwise and sending
$\alpha^{\vee}\mapsto-\alpha^{\vee}$.  The \emph{Weyl group} is
\[
    W=\langle\, s_{\alpha}\mid\alpha\in\Pi \,\rangle\subset\mathrm{GL}(\mathfrak t),
\]
acting on $\mathfrak t^{*}$ by
$\langle w\lambda,H\rangle=\langle\lambda,w^{-1}H\rangle$. Explicitly:
\[
    s_{\alpha}(\lambda)=\lambda-\langle\lambda,\alpha^{\vee}\rangle\,\alpha,
    \qquad\lambda\in\mathfrak t^{*}.
\]
This action preserves $R$.  For $w\in W$ the \emph{length} $\ell(w)$ is
the least $k\ge0$ with $w=s_{\alpha_{i_{1}}}\cdots s_{\alpha_{i_{k}}}$,
$\alpha_{i_{j}}\in\Pi$. An expression of minimal length is called
\emph{reduced}.
\end{definition}

\begin{example}\label{ex:weylSL3}
Consider $\mathfrak{sl}_3(\C)$, with simple roots $\alpha_1,\alpha_2$ as in
Example \ref{ExRoots}.  Under $\alpha_i=\varepsilon_i-\varepsilon_{i+1}$ the
simple reflection $s_i$ transposes $\varepsilon_i$ and $\varepsilon_{i+1}$,
so $W=\langle s_1,s_2\rangle$ maps onto the symmetric group $\mathcal S_3$.
The map is injective because $W$ acts faithfully on $\mathfrak t^{*}$,
whence $W\cong\mathcal S_3$ has six elements,
\[
    W = \{ \mathrm{Id},\ s_1,\ s_2,\ s_1 s_2,\ s_2 s_1,\ s_1 s_2 s_1 \},
\]
and the defining relations $s_1^2=s_2^2=\mathrm{Id}$,
$(s_1s_2)^3=\mathrm{Id}$ hold.  The lengths are $\ell(\mathrm{Id})=0$,
$\ell(s_1)=\ell(s_2)=1$, $\ell(s_1s_2)=\ell(s_2s_1)=2$ and
$\ell(s_1s_2s_1)=3$.  Note that the longest element admits two reduced
expressions, $s_1s_2s_1=s_2s_1s_2$.
\end{example}

Let $Y=G/P$ be the flag variety determined by a subset $\Pi_P\subset\Pi$.
Just as we restricted the root system to obtain $R_P$, we now single out a
subgroup of the Weyl group and a set of representatives for its cosets.

\begin{definition}
The \emph{parabolic subgroup of $W$} determined by $\Pi_P$ is
\[
    W_P=\langle\, s_{\alpha}\mid\alpha\in\Pi_P \,\rangle\subset W ,
\]
and we set
\[
    W_Y=\{\, w\in W \mid \ell(wv)=\ell(w)+\ell(v)\ \text{for all}\ v\in W_P \,\}.
\]
\end{definition}

\begin{proposition}[\cite{bjorner2005combinatorics}]\label{prop:minreps}
Every coset in $W/W_P$ contains a unique element of minimal length, and
$W_Y$ coincides with the set of these elements.  Consequently, the map
$W_Y\to W/W_P$ is a bijection, and every $w\in W$ factors uniquely as
$w=\overline{w}\,v$ with $\overline{w}\in W_Y$ and $v\in W_P$, with
$\ell(w)=\ell(\overline{w})+\ell(v)$.
\end{proposition}

\begin{remark}
$W_Y$ is only a set of representatives: it is not a subgroup of $W$, since
it need not be closed under multiplication.  For instance, for
$\PP^{2}=\SL_{3}(\C)/P$ with $\Pi_{P}=\{\alpha_{2}\}$ one has
$W_{P}=\{\mathrm{Id},s_{2}\}$ and
$W_{Y}=\{\mathrm{Id},s_{1},s_{2}s_{1}\}$, while
$s_{1}\cdot s_{2}s_{1}=s_{1}s_{2}s_{1}\notin W_{Y}$.  We write
$\overline{w}\in W_Y$ for the representative of the coset $wW_P$.
\end{remark}

\begin{lemma}[\cite{Humphreys1975}]\label{lem:weyl-equal} Let $N_{G}(\mathbb T)$ be the normalizer of $\mathbb T$ on $G$.
The map $N_{G}(\mathbb T)/\mathbb T\to \mathrm{GL}(\mathfrak t)$ induced by
the adjoint action is injective with image $W$.  We henceforth identify
$W=N_{G}(\mathbb T)/\mathbb T$.  In particular, every $w\in W$ lifts to an
element $n\in N_{G}(\mathbb T)$, and any two lifts differ by an element of
$\mathbb T$.
\end{lemma}

By Lemma \ref{lem:weyl-equal}, the coset $nP/P\in Y$ does not depend on the
chosen lift, since two lifts differ by an element of
$\mathbb T\subset B\subset P$.  We may therefore write $wP/P$
unambiguously, and the action of $w$ carries the base point $eP/P$ to
$wP/P$.  As $w$ runs over $W_{Y}$ these points are pairwise distinct, by
Proposition \ref{prop:minreps}.

\begin{definition}\label{def:schubert}
Let $Y=G/P$ be a flag variety and $B\subset P$ the standard Borel subgroup.
Let $w_{0}\in W$ be the unique longest element. Since $w_{0}(R^{+})=R^{-}$,
the element $w_{0}^{2}$ preserves $R^{+}$ and is therefore the identity.
Let $\dot w_{0}\in N_{G}(\mathbb T)$ be a lift.  The \emph{opposite Borel
subgroup} $B^{-}=\dot w_{0}B\dot w_{0}^{-1}$ does not depend on the choice
of lift, since $\mathbb T$ normalizes $B$.  For each $w\in W_{Y}$ we define:
\begin{defenum}
\item the \emph{opposite Schubert cell} $C_{w}=B^{-}wP/P\subset Y$, the orbit
      of $wP/P$ under $B^{-}$, which is isomorphic to the affine space
      $\C^{\dim Y-\ell(w)}$;
\item the \emph{Schubert variety} $\Omega_{w}=\overline{C_{w}}$, the Zariski
      closure of $C_{w}$ in $Y$, of codimension $\ell(w)$;
\item the \emph{Schubert class}
      $\sigma_{w}=[\Omega_{w}]\in\rH^{2\ell(w)}(Y,\Q)$, the cycle class of
      $\Omega_{w}$.
\end{defenum}
\end{definition}

\begin{remark}
Working with $B^{-}$ rather than $B$ is what makes $\ell(w)$ the
\emph{codimension} of $\Omega_{w}$, so that $\sigma_{w}$ sits in degree
$2\ell(w)$. The $B$-orbits give the same varieties indexed in the opposite
order.  This convention also fixes the direction of the containments among
Schubert varieties: one has $\Omega_{u}\supseteq\Omega_{v}$ precisely when
$u\le v$ in the Bruhat order, and in particular $wP/P\in\Omega_{u}$ if and
only if $u\le w$.  We have collected in Appendix~\ref{app:bruhat} the
definition of this order and the few properties we use.
\end{remark}

\begin{proposition}\label{prop:schubertbasis}
In the setting of Definition \ref{def:schubert}, the Schubert classes
$\{\sigma_{w}\}_{w\in W_{Y}}$ form a basis of $\rH^{*}(Y,\Q)$ as a graded
vector space,
\[
    \rH^{*}(Y,\Q)=\bigoplus_{w\in W_{Y}}\Q\,\sigma_{w},
    \qquad \sigma_{w}\in \rH^{2\ell(w)}(Y,\Q),
\]
and each $\sigma_{w}$ is purely of Hodge type $(\ell(w),\ell(w))$.  In
particular $\rH^{\mathrm{odd}}(Y,\Q)=0$, every class in $\rH^{*}(Y,\Q)$ is
algebraic, and
\[
    \chi(Y)=|W_{Y}|=\left|W/W_{P}\right| .
\]
\end{proposition}

\begin{proof}
By the Bruhat decomposition \cite{Borel1991}, $Y$ is the disjoint union of the opposite
Schubert cells,
\[
    Y=\bigsqcup_{w\in W_{Y}}C_{w},\qquad C_{w}=B^{-}wP/P,
\]
the union being independent of the representatives chosen, since
$B^{-}wP=B^{-}wvP$ for $v\in W_{P}$. Moreover, $C_{w}\cong\C^{\dim Y-\ell(w)}$.
A variety stratified by affine spaces admits a CW structure whose cells are
the $C_{w}$, all of even real dimension, so every boundary map in the
cellular chain complex vanishes.  Hence $\rH_{\mathrm{odd}}(Y,\Z)=0$ and
$\rH_{\mathrm{even}}(Y,\Z)$ is free with one generator for each cell, given
by the Borel--Moore fundamental class of its closure $\Omega_{w}$.  Since
$\Omega_{w}$ has complex codimension $\ell(w)$, its class lies in
$\rH_{2\dim Y-2\ell(w)}(Y,\Z)$, and Poincaré duality on the smooth variety
$Y$ carries it to $\sigma_{w}\in \rH^{2\ell(w)}(Y,\Z)$.  The $\sigma_{w}$
therefore form a basis of $\rH^{*}(Y,\Q)$ as a graded vector space.
Finally, each $\sigma_{w}$ is the cycle class of a closed subvariety of
codimension $\ell(w)$, hence lies in $\rH^{\ell(w),\ell(w)}(Y,\Q)$ by the
standard properties of cycle classes and all of $\rH^{*}(Y,\Q)$ is
algebraic.  Counting cells gives $\chi(Y)=|W_{Y}|$.
\end{proof}

\begin{example}
By Proposition \ref{prop:schubertbasis}, the Hodge diamond of a flag variety
is determined by the length distribution on $W_{Y}$: one has $h^{p,q}=0$ for
$p\neq q$ and $h^{p,p}=\left|\{w\in W_{Y}\mid \ell(w)=p\}\right|$.

For the full flag variety $\mathrm{Fl}(1,2,3)=\SL_{3}(\C)/B$ we have
$W_{Y}=W\cong\mathcal S_{3}$, with lengths $0,1,1,2,2,3$ by Example
\ref{ex:weylSL3}, so the nonzero Hodge numbers are $h^{0,0}=h^{3,3}=1$ and
$h^{1,1}=h^{2,2}=2$, and $\chi=6$:
    \[
    \begin{array}{ccccccc}
          &&&1&&&  \\
          &&0&&0&&\\
          &0&&2&&0&\\
          0&&0&&0&&0\\
           &0&&2&&0&\\
             &&0&&0&&\\
               &&&1&&&
    \end{array}
    \]

For $\PP^{2}=\SL_{3}(\C)/P$ with $\Pi_{P}=\{\alpha_{2}\}$ we have
$W_{P}=\{\mathrm{Id},s_{2}\}$ and
$W_{Y}=\{\mathrm{Id},\,s_{1},\,s_{2}s_{1}\}$, of lengths $0,1,2$, whence
$h^{0,0}=h^{1,1}=h^{2,2}=1$ and $\chi=3$:
  \[
    \begin{array}{ccccc}
         &&1&&\\
         &0&&0&\\
         0&&1&&0\\
         &0&&0&\\
         &&1&&
    \end{array}
    \]
\end{example}

Lastly, we give a combinatorial description of the second homology group
$\rH_{2}(Y,\Z)$ in terms of coroots.  Recall from Definition
\ref{def:weylgroup} the isomorphism $\nu\colon\mathfrak t\to\mathfrak t^{*}$
induced by the Killing form and the coroot
$\alpha^{\vee}=\nu^{-1}\big(2\alpha/(\alpha,\alpha)\big)\in\mathfrak t$
attached to $\alpha\in R$.  The set
$R^{\vee}=\{\alpha^{\vee}\mid\alpha\in R\}$ is again a root system, with
simple system $\Pi^{\vee}=\{\alpha_{i}^{\vee}\mid\alpha_{i}\in\Pi\}$.

\begin{definition}
The \emph{coroot lattice} is $Q^{\vee}=\operatorname{span}_{\Z}(R^{\vee})$,
free with basis $\Pi^{\vee}$.  For a flag variety $Y=G/P$ determined by
$\Pi_{P}\subset\Pi$, the \emph{parabolic coroot sublattice} is
$Q^{\vee}_{P}=\operatorname{span}_{\Z}\{\alpha_{j}^{\vee}\mid\alpha_{j}\in\Pi_{P}\}$
and the \emph{reduced coroot lattice} is the quotient
$Q^{\vee}_{Y}=Q^{\vee}/Q^{\vee}_{P}$, free with basis the images
$\overline{\alpha_{i}^{\vee}}$ for $\alpha_{i}\in\Pi_Y$.
\end{definition}

\begin{proposition}\label{prop:h2}
Let $\Pi_Y = \{\alpha_1 ,\ldots,\alpha_k\}$ and denote the Schubert varieties of the simple reflections $s_{\alpha_i}$ by $C_i=\overline{Bs_{\alpha_i}P/P}$. There is a canonical isomorphism of abelian groups
\begin{align}
        &\rH_{2}(Y,\Z)\;\cong\;Q^{\vee}_{Y}\\
        &[C_i] \mapsto \overline{\alpha_{i}^{\vee}}.
\end{align}
\end{proposition}

\begin{proof}
Consider the fibration $P\to G\to G/P=Y$ and its long exact sequence
\[
\ldots\to  \pi_{2}(G)\to\pi_{2}(Y)\to\pi_{1}(P)\to\pi_{1}(G)\to\pi_{1}(Y)\to\pi_{0}(P).
\]
Here $\pi_{1}(G)=0$ by hypothesis, $\pi_{2}(G)=0$ by a theorem of
É.~Cartan \cite{Bott-Cartan}, and $P$ is connected, so $\pi_{0}(P)=0$.  Hence $Y$ is
simply connected and $\pi_{2}(Y)\cong\pi_{1}(P)$. The Hurewicz theorem
gives $\rH_{2}(Y,\Z)\cong\pi_{2}(Y)\cong\pi_{1}(P)$.

Write $P=L\ltimes U$ for the Levi decomposition, $U$ the unipotent radical
and $L$ reductive containing $\mathbb T$.  As $U$ is contractible,
$\pi_{1}(P)\cong\pi_{1}(L)$.  Since $G$ is simply connected the cocharacter
lattice of $\mathbb T$ is the coroot lattice,
$\pi_{1}(\mathbb T)\cong Q^{\vee}$, and the inclusion
$\mathbb T\hookrightarrow L$ induces a surjection
$\pi_{1}(\mathbb T)\twoheadrightarrow\pi_{1}(L)$ whose kernel is generated
by the classes of the loops coming from the rank one subgroups attached to
the roots of $L$. These classes are the coroots $\alpha^{\vee}$ with
$\alpha\in R_{P}$, so the kernel is $Q^{\vee}_{P}$.  Therefore
\[
  \rH_{2}(Y,\Z)\cong\pi_{1}(L)\cong Q^{\vee}/Q^{\vee}_{P}=Q^{\vee}_{Y}.
\]
Under this identification, the Schubert curve $C_i$,
which is the closure of a one-dimensional cell and hence a copy of
$\PP^{1}$, has class $\overline{\alpha_{i}^{\vee}}$.
\end{proof}

By Proposition \ref{prop:schubertbasis}, the group $\rH^{2}(Y,\Q)$ has basis
the Schubert classes of length one, $\sigma_{s_{\alpha_{i}}}$ with
$\alpha_{i}\in\Pi_Y$, and by Proposition \ref{prop:h2} the group
$\rH_{2}(Y,\Z)$ has basis the classes $\overline{\alpha_{i}^{\vee}}$.  These
two bases are dual for the Kronecker pairing, so for
$\beta=\sum_{j}b_{j}\overline{\alpha^{\vee}_{i_{j}}}$ and
$\gamma=\sum_{j}a_{j}\sigma_{s_{\alpha_{i_{j}}}}$,
\begin{equation}\label{Intbeta}
    \langle\gamma,\beta\rangle=\int_{\beta}\gamma=a_{1}b_{1}+\dots+a_{k}b_{k}.
\end{equation}

\medskip

\subsection{The Picard group and first Chern class}\label{sec:picard}
 
In this subsection, we describe how to compute the Picard group and the first
Chern class of a flag variety from the representation-theoretic data of
fundamental weights.
 
\begin{definition}\label{def:fundweights}
Let $\Pi=\{\alpha_{1},\dots,\alpha_{r}\}$.  The \emph{fundamental weights}
$\omega_{1},\dots,\omega_{r}\in\mathfrak t^{*}$ are the basis dual to the
simple coroots for the evaluation pairing of Definition \ref{def:weylgroup},
\[
    \langle\omega_{i},\alpha_{j}^{\vee}\rangle=\delta_{ij}.
\]
For a flag variety $Y=G/P$ determined by $\Pi_{P}\subset\Pi$ we set
\[
    \Lambda_{Y}:=\bigoplus_{\alpha_{i}\in\Pi_Y}\Z\,\omega_{i}
    =\Big\{\lambda\in\bigoplus_{i}\Z\omega_{i}\ \Big|\
      \langle\lambda,\alpha_{j}^{\vee}\rangle=0\ \text{for all}\ \alpha_{j}\in\Pi_{P}\Big\}.
\]
\end{definition}

\begin{example}[Fundamental weights for $\SL_{4}(\C)/B$, $\PP^{3}$ and $\Gr(2,4)$]
\label{ex:fundweightsSL4}
The group $\SL_{4}(\C)$ has simple roots $\alpha_{1},\alpha_{2},\alpha_{3}$
in the notation $\alpha_{i}=\alpha_{i,i+1}$ of Example \ref{ExRoots}.
The fundamental weights are expressed as
\[
  \omega_{1}=\tfrac{3}{4}\alpha_{1}+\tfrac{1}{2}\alpha_{2}+\tfrac{1}{4}\alpha_{3},
  \qquad
  \omega_{2}=\tfrac{1}{2}\alpha_{1}+\alpha_{2}+\tfrac{1}{2}\alpha_{3},
  \qquad
  \omega_{3}=\tfrac{1}{4}\alpha_{1}+\tfrac{1}{2}\alpha_{2}+\tfrac{3}{4}\alpha_{3}.
\]
For the full flag variety, $\Pi_{P}=\emptyset$ and
$\Lambda_{Y}=\Z\omega_{1}\oplus\Z\omega_{2}\oplus\Z\omega_{3}$.  For
$\PP^{3}$, where $\Pi_{P}=\{\alpha_{2},\alpha_{3}\}$ and
$\Pi_Y=\{\alpha_{1}\}$, we get $\Lambda_{Y}=\Z\omega_{1}$; and for
$\Gr(2,4)$, where $\Pi_{P}=\{\alpha_{1},\alpha_{3}\}$ and
$\Pi_Y=\{\alpha_{2}\}$, we get $\Lambda_{Y}=\Z\omega_{2}$.
\end{example}

\medskip
 
\subsubsection{Homogeneous line bundles and the Picard group}
 
Given a representation $f\colon P\to\mathrm{GL}(\C^{n})$ we form the
associated vector bundle of rank $n$ over $G/P$,
\[
    E = G \times_P \C^{n},
\]
the quotient of $G\times\C^{n}$ identifying $(g,v)$ with
$(gp,f(p)^{-1}v)$ for $p\in P$.
 
Let $\omega_{i}$ be a fundamental weight with $\alpha_{i}\in\Pi_Y$.  By
Definition \ref{def:fundweights} we have
$\langle\omega_{i},\alpha_{j}^{\vee}\rangle=0$ for all $\alpha_{j}\in\Pi_{P}$,
so $\omega_{i}$ vanishes on the semisimple part of the Levi subalgebra of
$\mathfrak p$ and extends by zero across the nilpotent radical to a Lie
algebra character $\mathfrak p\to\C$.  This character integrates to a group
homomorphism $\theta_{i}\colon P\to\C^{*}$, and we write $\cO_{Y}(\omega_{i})$
for the line bundle associated with the one-dimensional representation
$\theta_{i}^{-1}$.  More generally, for
$\lambda=\sum_{j}a_{j}\omega_{i_{j}}\in\Lambda_{Y}$ we set
\[
  \cO_{Y}(\lambda)=G\times_{P}\C_{-\lambda}
  =\cO_{Y}(\omega_{i_{1}})^{\otimes a_{1}}\otimes\dots\otimes
   \cO_{Y}(\omega_{i_{k}})^{\otimes a_{k}},
\]
abbreviated $\cO_{Y}(a_{1},\dots,a_{k})$.
 
\begin{proposition}\label{prop:picard}
The assignment $\lambda\mapsto\cO_{Y}(\lambda)$ is a group isomorphism
\[
    \Lambda_{Y}\ \xrightarrow{\ \sim\ }\ \mathrm{Pic}(Y),
\]
and the first Chern class of $\cO_{Y}(\lambda)$ expands in the Schubert
basis as
\[
    c_{1}\big(\cO_{Y}(\lambda)\big)
    =a_{1}\sigma_{s_{\alpha_{i_{1}}}}+\dots+a_{k}\sigma_{s_{\alpha_{i_{k}}}} .
\]
\end{proposition}
 
\begin{proof}
By Proposition \ref{prop:schubertbasis} the Hodge numbers $h^{0,i}(Y)$
vanish for $i>0$, so $\rH^{1}(Y,\cO_{Y})=\rH^{2}(Y,\cO_{Y})=0$ and the
exponential sequence makes
$c_{1}\colon\mathrm{Pic}(Y)\to \rH^{2}(Y,\Z)$ an isomorphism.  Since
$\rH_{1}(Y,\Z)=0$, the universal coefficient theorem identifies
$\rH^{2}(Y,\Z)$ with $\operatorname{Hom}(\rH_{2}(Y,\Z),\Z)$, and
$\rH_{2}(Y,\Z)\cong Q^{\vee}_{Y}$ is free with basis the
$\overline{\alpha^{\vee}_{i_{j}}}$ by Proposition \ref{prop:h2}.  The dual
basis is $\{\omega_{i_{j}}\}$ by Definition \ref{def:fundweights}, so
$\mathrm{Pic}(Y)$ is free of rank $|\Pi_Y|$ and $\lambda\mapsto\cO_{Y}(\lambda)$
is an isomorphism.  Finally, by \eqref{Intbeta} both sides of the displayed
formula pair with $\overline{\alpha^{\vee}_{i_{j}}}$ to give
$\langle\lambda,\alpha^{\vee}_{i_{j}}\rangle=a_{j}$, and $\rH^{2}(Y,\Q)$ has
the classes $\sigma_{s_{\alpha_{i}}}$ as a basis by
Proposition \ref{prop:schubertbasis}.
\end{proof}

\begin{corollary}\label{cor:effective}
A class $\beta=\sum_{j}b_{j}\overline{\alpha^{\vee}_{i_{j}}}$ is represented
by an effective algebraic $1$-cycle if and only if $b_{j}\ge0$ for all $j$.
\end{corollary}
\begin{proof}
If all $b_{j}\ge0$ then $\beta$ is the class of the effective cycle
$\sum_{j}b_{j}\,C_j$. Conversely, let $C$ be an
irreducible curve.  The line bundle $\cO_{Y}(\omega_{i})$ associated with
the dominant weight $\omega_{i}$ is globally generated by the Borel--Weil
theorem \cite{Bott1957}, so $\int_{C}\sigma_{s_{\alpha_{i}}}\ge0$ for each $i$.
By \eqref{Intbeta}, these integrals are the coordinates of $[C]$, which are
therefore nonnegative, and the same holds for nonnegative combinations of
irreducible curves.
\end{proof}
 
\subsubsection{The first Chern class}
 
To compute the first Chern class of $Y$, we consider the weight
\[
    \rho_{Y}=\sum_{\alpha\in R^{+}_{Y}}\alpha ,
\]
the sum of the positive roots $\alpha$ for which $\mathfrak g_{-\alpha}$
does not lie in $\mathfrak p$, that is, of those indexing the tangent
directions of $Y$ at the base point.
 
\begin{lemma}\label{lem:rhoX}
With $\rho_{Y}$ as above:
\begin{defenum}
\item for $\alpha_{j}\in\Pi_{P}$ the reflection $s_{\alpha_{j}}$ permutes
      $R^{+}_{Y}$, and $\langle\rho_{Y},\alpha_{j}^{\vee}\rangle=0$;
\item for $\alpha_{i}\in\Pi_Y$, we have
      $\langle\rho_{Y},\alpha_{i}^{\vee}\rangle\ge2$.
\end{defenum}
In particular $\rho_{Y}\in\Lambda_{Y}$, and writing
$\Pi_Y=\{\alpha_{i_{1}},\dots,\alpha_{i_{k}}\}$,
\[
  \rho_{Y}=c_{1}\omega_{i_{1}}+\dots+c_{k}\omega_{i_{k}},
  \qquad c_{j}=\langle\rho_{Y},\alpha_{i_{j}}^{\vee}\rangle\ \ge\ 2 .
\]
\end{lemma}
 
\begin{proof}
Recall that a simple reflection $s_{\alpha}$ permutes
$R^{+}\smallsetminus\{\alpha\}$ and sends $\alpha$ to $-\alpha$.
 
(a)  Let $\alpha_{j}\in\Pi_{P}$.  Then $s_{\alpha_{j}}\in W_{P}$ preserves
$R_{P}$, hence preserves $R^{+}_{P}$ up to the single sign change
$\alpha_{j}\mapsto-\alpha_{j}$, and therefore permutes the complement
$R^{+}_{Y}=R^{+}\smallsetminus R^{+}_{P}$.  Applying $s_{\alpha_{j}}$ to
$\rho_{Y}$ thus fixes it, while the reflection formula gives
$s_{\alpha_{j}}(\rho_{Y})=\rho_{Y}-\langle\rho_{Y},\alpha_{j}^{\vee}\rangle\alpha_{j}$,
whence $\langle\rho_{Y},\alpha_{j}^{\vee}\rangle=0$.
 
(b) Define
    \[
        \rho=\sum_{\alpha\in R^+}\alpha,
        \qquad
        \rho_P=\rho-\rho_Y
        =\sum_{\alpha\in R_P^+}\alpha.
    \]
    Since the simple reflection $s_{\alpha_i}$ permutes
    $R^+\setminus\{\alpha_i\}$ and sends $\alpha_i$ to $-\alpha_i$, we have
    \[
        s_{\alpha_i}(\rho)
        =\rho-2\alpha_i.
    \]
    On the other hand, 
    \[
        s_{\alpha_i}(\rho)
        =\rho-\langle \rho,\alpha_i^\vee\rangle\alpha_i.
    \]
    Comparing the two expressions yields
    \[
        \langle \rho,\alpha_i^\vee\rangle=2
    \]
    for every $\alpha_i\in\Pi_Y$.

    Therefore,
    \[
        \langle \rho_Y,\alpha_i^\vee\rangle
        =
        \langle \rho,\alpha_i^\vee\rangle
        -\langle \rho_P,\alpha_i^\vee\rangle
        =
        2-\langle \rho_P,\alpha_i^\vee\rangle.
    \]
    Finally, $\rho_P$ is a positive linear combination of the simple roots
    in $\Pi_P$. Since
    \[
        \langle\alpha_j,\alpha_i^\vee\rangle\leq 0
        \qquad\text{for }j\neq i,
    \]
    and $\alpha_i\notin\Pi_P$, it follows that
    \[
        \langle \rho_P,\alpha_i^\vee\rangle\leq 0.
    \]
    Hence
    \[
      \langle \rho_Y,\alpha_i^\vee\rangle\geq 2.
    \]

By (a) the weight $\rho_{Y}$ pairs to zero with every coroot of $\Pi_{P}$,
so $\rho_{Y}\in\Lambda_{Y}$ by Definition \ref{def:fundweights}, and its
coordinates in the basis $\{\omega_{i}\}_{\alpha_{i}\in\Pi_Y}$ are the
pairings $\langle\rho_{Y},\alpha_{i_{j}}^{\vee}\rangle$.
\end{proof}

\begin{proposition}\label{prop:chern}
Let $Y=G/P$ be a flag variety with
$\Pi_{Y}=\{\alpha_{i_{1}},\dots,\alpha_{i_{k}}\}$.  Then
$K_{Y}^{-1}=\cO_{Y}(\rho_{Y})$ and
\[
  c_{1}(TY)=c_{1}\,\sigma_{s_{\alpha_{i_{1}}}}+\dots+c_{k}\,\sigma_{s_{\alpha_{i_{k}}}},
  \qquad c_{j}=\langle\rho_{Y},\alpha_{i_{j}}^{\vee}\rangle\ \ge\ 2 .
\]
\end{proposition}
 
\begin{proof}
The tangent bundle of the homogeneous space $Y=G/P$ is the associated
bundle $TY=G\times_{P}(\mathfrak g/\mathfrak p)$, so
$\det TY=K_{Y}^{-1}$ is the line bundle associated with the
one-dimensional $P$-module $\Lambda^{\mathrm{top}}(\mathfrak g/\mathfrak p)$.
By the proof of Proposition \ref{prop:dimension} the $\mathbb T$-weights of
$\mathfrak g/\mathfrak p$ are the elements of $-R^{+}_{Y}$, so that module
has weight $-\rho_{Y}$.  With the convention
$\cO_{Y}(\lambda)=G\times_{P}\C_{-\lambda}$ this says
$K_{Y}^{-1}=\cO_{Y}(\rho_{Y})$.  Now apply Lemma \ref{lem:rhoX} and
Proposition \ref{prop:picard}.
\end{proof}
 
\begin{corollary}\label{cor:fano}
Every flag variety is Fano, and its Fano index (the largest integer $q$
such that $K_{Y}^{-1}=L^{\otimes q}$ for some $L\in\mathrm{Pic}(Y)$)
equals $\gcd(c_{1},\dots,c_{k})$.
\end{corollary}
 
\begin{proof}
By Corollary \ref{cor:effective} a nonzero effective class $\beta$ has
nonnegative coordinates $b_{j}$, not all zero, so
$\int_{\beta}c_{1}(TY)=\sum_{j}c_{j}b_{j}>0$ by \eqref{Intbeta} and
Lemma \ref{lem:rhoX}.  The statement about the index follows from
$\mathrm{Pic}(Y)\cong\Lambda_{Y}$ and $K^{-1}_{Y}=\cO_{Y}(\rho_{Y})$.
\end{proof}
 
\begin{example}\label{ex:chernSL4}
For $\SL_{4}(\C)$ the positive roots are
\[
  R^{+}=\{\alpha_{1},\ \alpha_{2},\ \alpha_{3},\
    \alpha_{1}+\alpha_{2},\ \alpha_{2}+\alpha_{3},\
    \alpha_{1}+\alpha_{2}+\alpha_{3}\}.
\]
 
For the full flag variety $\SL_{4}(\C)/B$ we have $\Pi_{P}=\emptyset$ and
$R^{+}_{Y}=R^{+}$, so
\[
  \rho_{Y}=3\alpha_{1}+4\alpha_{2}+3\alpha_{3}
  =2\omega_{1}+2\omega_{2}+2\omega_{3},
\]
giving
$c_{1}(TY)=2\sigma_{s_{\alpha_{1}}}+2\sigma_{s_{\alpha_{2}}}+2\sigma_{s_{\alpha_{3}}}$
and Fano index $2$.
 
For $\PP^{3}=\SL_{4}(\C)/P$ with $\Pi_{P}=\{\alpha_{2},\alpha_{3}\}$ the
subsystem $R_{P}$ contributes $\alpha_{2},\alpha_{3},\alpha_{2}+\alpha_{3}$,
so $R^{+}_{Y}=\{\alpha_{1},\ \alpha_{1}+\alpha_{2},\
\alpha_{1}+\alpha_{2}+\alpha_{3}\}$ and
\[
  \rho_{Y}=3\alpha_{1}+2\alpha_{2}+\alpha_{3}=4\omega_{1},
\]
so $c_{1}(TY)=4\sigma_{s_{\alpha_{1}}}$, of Fano index $4$.
 
For $\Gr(2,4)=\SL_{4}(\C)/P$ with $\Pi_{P}=\{\alpha_{1},\alpha_{3}\}$ we get
$R^{+}_{Y}=\{\alpha_{2},\ \alpha_{1}+\alpha_{2},\ \alpha_{2}+\alpha_{3},\
\alpha_{1}+\alpha_{2}+\alpha_{3}\}$ and
\[
  \rho_{Y}=2\alpha_{1}+4\alpha_{2}+2\alpha_{3}=4\omega_{2},
\]
so $c_{1}(TY)=4\sigma_{s_{\alpha_{2}}}$, of Fano index $4$.  In each case
$|R^{+}_{Y}|=\dim Y$, namely $6$, $3$ and $4$, in accordance with
Proposition \ref{prop:dimension}.
\end{example}
 
\subsubsection{Positivity}
 
\begin{proposition}\label{prop:ampleness}
Let $L=\cO_{Y}(\lambda)$ with
$\lambda=a_{1}\omega_{i_{1}}+\dots+a_{k}\omega_{i_{k}}\in\Lambda_{Y}$.  Then
\begin{defenum}
\item $L$ is ample if and only if $a_{j}>0$ for all $j$;
\item $L$ is generated by global sections if and only if $a_{j}\ge0$ for
      all $j$.
\end{defenum}
\end{proposition}
 
\begin{proof}
(a)  By Corollary \ref{cor:effective} the cone of effective $1$-cycles is
generated by the classes $\overline{\alpha^{\vee}_{i_{j}}}$ of the Schubert
curves $C_{i_j}$.  Being finitely generated, it is
closed, so the Kleiman criterion applies.  By \eqref{Intbeta},
\[
    \int_{C_{i_j}}c_{1}(L)
    =\langle\lambda,\alpha_{i_{j}}^{\vee}\rangle=a_{j},
\]
and $c_{1}(L)$ is positive on every nonzero effective class if and only if
all $a_{j}>0$.
 
(b)  A homogeneous line bundle $\cO_{Y}(\lambda)$ is globally generated if
and only if $\lambda$ is dominant, that is
$\langle\lambda,\alpha^{\vee}\rangle\ge0$ for every $\alpha\in\Pi$: by the
Borel--Weil theorem \cite{Bott1957}, $\rH^{0}(Y,\cO_{Y}(\lambda))$ is then
the irreducible $G$-module of highest weight $\lambda$, whose evaluation map
is surjective by homogeneity, and it vanishes otherwise.  Since
$\lambda\in\Lambda_{Y}$ already pairs to zero with the coroots of
$\Pi_{P}$, dominance reduces to $a_{j}\ge0$ for all $j$.
\end{proof}
 
\medskip
 
\subsection{Complete intersections and Sommese's theorem}\label{sec:ci}
 
By cutting a flag variety with sections of direct sums of homogeneous line
bundles, the combinatorial data of the ambient root system determines the
basic geometry of the subvariety.
 
\begin{theorem}\label{thm:ci}
Let $Y=G/P$ be a flag variety of dimension $N=|R^{+}_{Y}|$, with
$\Pi_{Y}=\{\alpha_{i_{1}},\dots,\alpha_{i_{k}}\}$ and
$\rho_{Y}=\sum_{m=1}^{k}c_{m}\omega_{i_{m}}$ as in Lemma \ref{lem:rhoX}.
Let
\[
  \cE=\bigoplus_{j=1}^{e}\cO_{Y}(\lambda_{j}),
  \qquad \lambda_{j}=\sum_{m=1}^{k}a_{j,m}\,\omega_{i_{m}}\in\Lambda_{Y},
\]
and let $X\subset Y$ be the zero locus of a global section of $\cE$,
assumed smooth of codimension $e$.  Then
\begin{defenum}
\item if every $\lambda_{j}$ is ample, that is $a_{j,m}>0$ for all $j,m$,
      and $N-e\ge3$, then the restriction is an isomorphism
      \[
        \mathrm{Pic}(Y)\ \xrightarrow{\ \sim\ }\ \mathrm{Pic}(X),
        \qquad\text{so that}\qquad \mathrm{Pic}(X)\cong\Lambda_{Y};
      \]
\item the canonical bundle of $X$ is
      \[
        K_{X}=\cO_{X}\Big(\sum_{j=1}^{e}\lambda_{j}-\rho_{Y}\Big)=i^* \cO_{Y}\Big(\sum_{j=1}^{e}\lambda_{j}-\rho_{Y}\Big),
      \]
      where $i:X \rightarrow Y$ is the inclusion map;
\item if $\sum_{j=1}^{e}a_{j,m}<c_{m}$ for every $m$, then $X$ is Fano;
\item if $\sum_{j=1}^{e}a_{j,m}=c_{m}$ for every $m$, then $K_{X}$ is
      trivial.
\end{defenum}
\end{theorem}
 
\begin{proof}
\emph{(a)}  Since each $\lambda_{j}$ is ample, $\cE$ is an ample vector
bundle by Proposition \ref{prop:ampleness}, and $X$ is the zero locus of a
transverse section.  Sommese's generalization of the Lefschetz hyperplane
theorem \cite{Lazarsfeld2004} gives $\pi_{q}(Y,X)=0$ for $q\le N-e$, hence
by the relative Hurewicz theorem $\rH_{q}(Y,X;\Z)=0$ in the same range.  The
long exact sequence in homology then makes $\rH_{q}(X,\Z)\to \rH_{q}(Y,\Z)$
an isomorphism for $q<N-e$, hence for $q=1,2$ since $N-e\ge3$.  As $Y$ is
simply connected, $\rH_{1}(X,\Z)=\rH_{1}(Y,\Z)=0$, and dualizing the
isomorphism in degree $2$ gives
$\rH^{2}(Y,\Z)\xrightarrow{\sim}\rH^{2}(X,\Z)$.
 
It remains to see that $c_{1}\colon\mathrm{Pic}(X)\to \rH^{2}(X,\Z)$ is an
isomorphism, for which it suffices that
$\rH^{1}(X,\cO_{X})=\rH^{2}(X,\cO_{X})=0$.  The section of $\cE$ resolves
$\cO_{X}$ by the Koszul complex
\[
  0\to\Lambda^{e}\cE^{*}\to\Lambda^{e-1}\cE^{*}\to\dots\to\cE^{*}\to\cO_{Y}\to\cO_{X}\to0,
\]
which we split into short exact sequences
$0\to K_{1}\to\cO_{Y}\to\cO_{X}\to0$ and
$0\to K_{m+1}\to\Lambda^{m}\cE^{*}\to K_{m}\to0$ for $1\le m\le e-1$, with
$K_{e}=\Lambda^{e}\cE^{*}$.  For $m\ge1$ the sheaf $\Lambda^{m}\cE^{*}$ is a
direct sum of line bundles
$\big(\cO_{Y}(\lambda_{j_{1}}+\dots+\lambda_{j_{m}})\big)^{-1}$, each the
inverse of an ample line bundle, so Kodaira vanishing in the form
$\rH^{q}(Y,A^{-1})=0$ for $A$ ample and $q<N$ gives
\[
  \rH^{q}\big(Y,\Lambda^{m}\cE^{*}\big)=0
  \qquad (1\le m\le e,\ q<N).
\]
Moreover $\rH^{q}(Y,\cO_{Y})=0$ for $q>0$, since $h^{0,q}(Y)=0$ by
Proposition \ref{prop:schubertbasis}.  The first short exact sequence
therefore gives $\rH^{q}(\cO_{X})\cong \rH^{q+1}(K_{1})$ for $q\ge1$, and
the remaining ones give $\rH^{q}(K_{m})\cong \rH^{q+1}(K_{m+1})$ for
$q<N-1$.  Iterating,
\[
  \rH^{1}(\cO_{X})\cong \rH^{e+1}\big(\Lambda^{e}\cE^{*}\big),
  \qquad
  \rH^{2}(\cO_{X})\cong \rH^{e+2}\big(\Lambda^{e}\cE^{*}\big),
\]
and both vanish as soon as $e+2<N$, that is $N-e\ge3$.  Hence
\[
  \mathrm{Pic}(X)\;\cong\;\rH^{2}(X,\Z)\;\cong\;\rH^{2}(Y,\Z)\;\cong\;\mathrm{Pic}(Y),
\]
the outer isomorphisms by the exponential sequence and the middle one by the
first paragraph; and $\mathrm{Pic}(Y)\cong\Lambda_{Y}$ by
Proposition \ref{prop:picard}.  All the maps involved are induced by
restriction, so the composite is the restriction map.
 
\emph{(b)}  Since $X$ is the zero locus of a transverse section, its normal
bundle is $N_{X/Y}\cong\cE|_{X}$, and adjunction gives
$K_{X}=\big(K_{Y}\otimes\det\cE\big)\big|_{X}$.  By
Proposition \ref{prop:chern} we have $K_{Y}=\cO_{Y}(-\rho_{Y})$, while
$\det\cE=\cO_{Y}\big(\sum_{j}\lambda_{j}\big)$, whence the formula.
 
\emph{(c)}  By (b),
\[
  K_{X}^{-1}=\cO_{X}\Big(\rho_{Y}-\sum_{j}\lambda_{j}\Big),
  \qquad
  \rho_{Y}-\sum_{j}\lambda_{j}
  =\sum_{m=1}^{k}\Big(c_{m}-\sum_{j=1}^{e}a_{j,m}\Big)\omega_{i_{m}} .
\]
If every coordinate is strictly positive, then
$\cO_{Y}\big(\rho_{Y}-\sum_{j}\lambda_{j}\big)$ is ample on $Y$ by
Proposition \ref{prop:ampleness}, and the restriction of an ample line
bundle to a closed subvariety is ample; hence $K_{X}^{-1}$ is ample.
 
\emph{(d)}  If $\sum_{j}a_{j,m}=c_{m}$ for every $m$ then the weight
$\sum_{j}\lambda_{j}-\rho_{Y}$ vanishes, so $K_{X}=\cO_{X}$ by (b).
\end{proof}
 
\begin{remark}\label{rmk:converses}
The converses require an extra input. For \textup{(d)}, triviality of $\cO_{X}(\nu)$ forces $\nu=0$ only
if restriction $\mathrm{Pic}(Y)\to\mathrm{Pic}(X)$ is injective, which holds
under the hypotheses of \textup{(a)}.  For \textup{(c)}, ampleness of
$K^{-1}_{X}$ is positivity on the cone of effective curves \emph{of $X$},
which under the identification of \textup{(a)} sits inside that of $Y$ but
may a priori be smaller; positivity on a smaller cone is a weaker condition,
so the converse does not follow formally.  It does follow as soon as each
class $\overline{\alpha^{\vee}_{i_{m}}}$ is represented by a curve contained
in $X$, since then $\int c_{1}(TX)=c_{m}-\sum_{j}a_{j,m}$ must be
positive.  
\end{remark}
 
\begin{example}[A Calabi--Yau threefold in a Grassmannian]
Let $Y=\Gr(2,5)$, of dimension $N=2(5-2)=6$, with $\Pi_{Y}=\{\alpha_{2}\}$
and $\mathrm{Pic}(Y)=\Z\,\omega_{2}$ of rank one.  As in
Example \ref{ex:chernSL4} one computes $\rho_{Y}=5\omega_{2}$, so
$K_{Y}=\cO_{Y}(-5\omega_{2})$ and the Fano index of $\Gr(2,5)$ is $5$.
 
Take
$\cE=\cO_{Y}(\omega_{2})\oplus\cO_{Y}(2\omega_{2})\oplus\cO_{Y}(2\omega_{2})$,
that is hypersurfaces of degrees $1,2,2$ in the Plücker embedding, so
$e=3$.  Since $\mathrm{Pic}(Y)$ has rank one, every nonzero effective weight
is ample, and the hypothesis of \textup{(a)} holds; moreover $N-e=3$, so
$\mathrm{Pic}(X)=\Z\,\omega_{2}$ by Theorem \ref{thm:ci}\textup{(a)}.  By
\textup{(d)},
\[
    K_{X}=\cO_{X}\big((1+2+2)\omega_{2}-5\omega_{2}\big)=\cO_{X},
\]
and $X$ is simply connected by the first paragraph of the proof of
\textup{(a)}, while $\rH^{1}(X,\cO_{X})=\rH^{2}(X,\cO_{X})=0$ by its second
paragraph.  Hence, $X$ is a projective Calabi--Yau threefold.
\end{example}

 \bigskip

\section{Gromov--Witten invariants via localization}\label{sec:Loc_methods}
 
Let $X$ be a smooth complex projective variety.  Computing the
Gromov--Witten invariants of $X$ is in general difficult.  In
\cite{Kontsevich1995}, Kontsevich proposed a \emph{localization method}
that reduces this task, for a certain class of varieties carrying torus
actions, to a combinatorial problem; these are called in the literature GKM spaces \cites{Liu_Chiu,holmes2025computationsequivariantgromovwittentheory}.  Building on this method, many works have appeared in the literature providing closed formulae for
Gromov--Witten invariants \cites{Graber1999, Maeno2001, Liu_Chiu, holmes2025computationsequivariantgromovwittentheory}.
In this section, we revise the theory and apply it to our context of flag
varieties and complete intersections in flag varieties. It is worth noticing that, as shown in \cite{Guillemin2006}, compact homogeneous spaces with non-zero Euler characteristic are examples of a GKM space. Our situation is a particular instance. Check, for example, our propositions \ref{prop:fixedpoints},~\ref{prop:tangentweights}, against the given reference.
 
\medskip
 
\subsection{Warming up: localization in flag varieties}
 
For a smooth proper variety $Z$ of complex dimension $n=\dim_{\C}Z$,
equipped with an action of the algebraic torus $\mathbb T=(\C^{*})^{r}$, we
write $\rH^{*}_{\mathbb T}(Z,\Q)=\rH^{*}\big((Z\times E\mathbb T)/\mathbb T,\Q\big)$
for its equivariant cohomology.  We recall that this is an algebra on
$\rH^{*}_{\mathbb T}(\mathrm{pt},\Q)=\Q[x_{1},\dots,x_{r}]$, where each
generator $x_{i}$ has degree $2$.  The symbol $\int_{Z}$ will throughout
denote the equivariant pushforward to a point.  It lowers the cohomological
degree by the real dimension $2n$, taking values in $\Q[x_{1},\dots,x_{r}]$
rather than in $\Q$.  For a class $\phi\in \rH^{2n}_{\mathbb T}(Z,\Q)$ of
pure degree $2n$, the equivariant pushforward $\int_{Z}\phi$ belongs to
$\rH^{0}_{\mathbb T}(\mathrm{pt},\Q)\cong\Q$ and coincides with the ordinary
integral $\int_{Z}\iota^{*}\phi$, where
$\iota^{*}\colon \rH^{*}_{\mathbb T}(Z,\Q)\to \rH^{*}(Z,\Q)$ is the natural
map forgetting equivariance.
 
The already classical Atiyah--Bott or Berline--Vergne theorems can be
stated as \cites{atiyah-boot,Berline1983}:
 
\begin{theorem}[Atiyah--Bott--Berline--Vergne]\label{thm:ABBV}
Let $Z$ be a smooth proper variety, or more generally a smooth proper
Deligne--Mumford stack, with an action of $\mathbb T$.  Let
$Z^{\mathbb T}=\bigsqcup_{F}F$ be the decomposition of its fixed locus into
connected components, each with normal bundle $N_{F}$.  Then for every
$\phi\in \rH^{*}_{\mathbb T}(Z,\Q)$
\[
  \int_{Z}\phi\;=\;\sum_{F}\int_{F}\frac{\phi|_{F}}{e^{\mathbb T}(N_{F})}
\]
where the right-hand side should be understood in the fraction field of
$\Q[x_{1},\dots,x_{r}]$.  Here, $e^{\mathbb T}(N_{F})$ is the equivariant
Euler class of the vector bundle $N_{F}$.
\end{theorem}
 
\medskip
 
Let $Y=G/P$ be a flag variety defined by the following data: $G$ is a
connected simply connected semisimple complex Lie group with Lie algebra
$\mathfrak g$, and $\mathbb T\subset B\subset P$ is a sequence of subgroups
consisting of a maximal torus $\mathbb T$, a Borel subgroup $B$ containing
$\mathbb T$, and a parabolic subgroup $P$ containing $B$.  The torus
$\mathbb T$ acts on $Y$ by left translation,
\[
  \mathbb T\times Y\longrightarrow Y,
  \qquad (t,gP/P)\longmapsto t\cdot gP/P:=tgP/P.
\]
 
\begin{proposition}\label{prop:fixedpoints}
The fixed locus of the $\mathbb T$-action on $Y$ is the finite set
\[
  Y^{\mathbb T}\;=\;\{\,p_{w}:=wP/P \ \mid\ w\in W_{Y}\,\},
\]
and $w\mapsto p_{w}$ is a bijection.  In particular
$|Y^{\mathbb T}|=|W_{Y}|=\chi(Y)$.
\end{proposition}
 
\begin{proof}
Let $w\in W_{Y}$ and lift it to $n\in N_{G}(\mathbb T)$.  For $t\in\mathbb T$
we have $tn=n(n^{-1}tn)$ with $n^{-1}tn\in\mathbb T\subset P$, whence
$t\cdot p_{w}=n(n^{-1}tn)P/P=p_{w}$.  Thus, every $p_{w}$ is fixed.
 
Conversely, let $gP/P$ be fixed.  Then $t g P/P=gP/P$ for all $t$, i.e.
$g^{-1}\mathbb Tg\subset P$, so $g^{-1}\mathbb Tg$ is a maximal torus of
$P$.  All maximal tori of $P$ are conjugate in $P$, so there is $q\in P$
with $q\,g^{-1}\mathbb T g\,q^{-1}=\mathbb T$, that is
$gq^{-1}\in N_{G}(\mathbb T)$.  Since $q\in P$ we get
$gP/P=(gq^{-1})P/P=p_{w}$, where $w\in W_{Y}$ is the minimal length
representative of the class of $gq^{-1}$ in $W/W_{P}$.  Injectivity holds
because $wP/P=w'P/P$ with $w,w'\in W_{Y}$ forces $w^{-1}w'\in W_{P}$ and
therefore $w=w'$.  The last assertion was proven in
Proposition \ref{prop:schubertbasis}.
\end{proof}
 
\begin{proposition}\label{prop:tangentweights}
For $w\in W_{Y}$ the tangent space at $p_{w}$ decomposes as a
$\mathbb T$-module into one-dimensional weight spaces,
\[
  T_{p_{w}}Y\;\cong\;\bigoplus_{\beta\in R^{+}_{Y}}\mathfrak g_{-w\beta},
\]
the summand indexed by $\beta$ carrying the character $-w\beta$.  All these
characters are nonzero and pairwise non-proportional, and
\[
  e^{\mathbb T}\big(T_{p_{w}}Y\big)\;=\;\prod_{\beta\in R^{+}_{Y}}\big(-w\beta\big)
  \ \in\ \rH^{2\dim Y}_{\mathbb T}(\mathrm{pt}).
\]
\end{proposition}
 
\begin{proof}
At the base point, $T_{eP/P}Y=\mathfrak g/\mathfrak p\cong
\bigoplus_{\beta\in R^{+}_{Y}}\mathfrak g_{-\beta}$ as in the proof of
Proposition \ref{prop:dimension}, and $\mathbb T$ acts on
$\mathfrak g_{-\beta}$ by $-\beta$.  Left translation by $w$ carries
$T_{eP/P}Y$ to $T_{p_{w}}Y$ and twists the action by conjugation: for
$X\in\mathfrak g_{-\beta}$ and $t\in\mathbb T$,
\[
  t\,w\exp(X)P/P=w\,\big(w^{-1}tw\big)\exp(X)\big(w^{-1}tw\big)^{-1}P/P
  =w\exp\!\big(\mathrm{Ad}(w^{-1}tw)X\big)P/P,
\]
and $\mathrm{Ad}(w^{-1}tw)X=(-\beta)(w^{-1}tw)\,X=(-w\beta)(t)\,X$.  Since
$\beta$ is a root, $-w\beta\neq0$; and two distinct elements of $R^{+}_{Y}$
are distinct positive roots, therefore never proportional in a reduced root
system.  The Euler class of a sum of characters is their product.
\end{proof}
 
Since the fixed points of $Y$ are isolated by
Proposition \ref{prop:fixedpoints} and their tangent weights are given by
Proposition \ref{prop:tangentweights}, the formula of
Theorem \ref{thm:ABBV} takes on $Y$ the following shape:
 
\begin{corollary}\label{cor:ABBVflag}
For every $\phi\in \rH^{*}_{\mathbb T}(Y,\Q)$,
\[
  \int_Y\phi\;=\;\sum_{w\in W_{Y}}
  \frac{\phi|_{p_{w}}}{\prod_{\beta\in R^{+}_{Y}}\big(-w\beta\big)} .
\]
\end{corollary}

\medskip
 
\subsection{Gromov--Witten invariants: a quick recall}\label{Sub.Sec:GW}
 
Throughout this subsection, $Y$ denotes a smooth complex projective variety
which is \emph{convex}, that is, such that
\[
  \rH^{1}\big(\PP^{1},f^{*}TY\big)=0
  \qquad\text{for every morphism }f\colon\PP^{1}\longrightarrow Y ;
\]
by the normalization sequence, applied inductively over the components of
the domain, this implies $\rH^{1}(C,f^{*}TY)=0$ for every genus-zero
stable map $f\colon C\to Y$.  Convexity holds for homogeneous spaces, in
particular for the flag varieties $Y=G/P$ that concern us, because $TY$
is then globally generated.  We carry this assumption herein since it
simplifies the exposition below by making all the moduli spaces smooth.
 
\begin{definition}
A \emph{genus-zero stable map} to $Y$ is a morphism
$f\colon(C;x_{1},\dots,x_{n})\to Y$ where
\begin{defenum}
  \item $C$ is a projective, connected, reduced, nodal curve of arithmetic
        genus $\mathsf g=0$;
  \item $x_{1},\dots,x_{n}\in C$ are distinct smooth points;
  \item $\Aut(f)$ is finite, where $\Aut(f)$ is the group of automorphisms
        $\varphi$ of $C$ with $\varphi(x_{i})=x_{i}$ for all $i$ and
        $f\circ\varphi=f$.
\end{defenum}
\end{definition}
 
Condition (c) is equivalent to the geometric statement: every irreducible
component of $C$ contracted by $f$ carries at least three \emph{special}
points, that is, nodes or marked points.
 
\medskip
 
Let $\beta\in \rH_{2}(Y,\Z)$ be an effective curve class.  We denote by
$\overline{\mathcal M}_{0,n}(Y,\beta)$ the moduli stack of genus-zero stable
maps $f\colon(C;x_{1},\dots,x_{n})\to Y$ with $f_{*}[C]=\beta$.  Since $Y$
is convex, $\overline{\mathcal M}_{0,n}(Y,\beta)$ is a smooth proper
Deligne--Mumford stack of dimension
\[
  \boldsymbol\delta(n,\beta)\;=\;\dim Y-3+n+\int_{\beta}c_{1}(TY),
\]
with fundamental class
\[
  \big[\overline{\mathcal M}_{0,n}(Y,\beta)\big]
  \ \in\ \mathsf{CH}_{\boldsymbol\delta(n,\beta)}
     \big(\overline{\mathcal M}_{0,n}(Y,\beta),\Q\big),
\]
where $\mathsf{CH}_{k}(-,\Q)$ denotes the $k$-th Chow group with rational
coefficients.  For $\beta=0$ one has
$\overline{\mathcal M}_{0,n}(Y,0)\cong Y\times\overline{\mathcal M}_{0,n}$,
which requires $n\ge3$.  We write $\overline M_{0,n}(Y,\beta)$ for the
coarse moduli space, a normal projective variety with finite quotient
singularities.  Let $\mathcal M_{0,n}(Y,\beta)$ denote the open substack of
maps with smooth irreducible domain
\cite{fulton1997notesstablemapsquantum}.  When $Y$ is homogeneous,
$\overline{\mathcal M}_{0,n}(Y,\beta)$ is irreducible and
$\mathcal M_{0,n}(Y,\beta)$ is dense in it \cite{Thomsen}.
 
Let $\ev_{i}\colon\overline{\mathcal M}_{0,n}(Y,\beta)\to Y$,
$(C;x_{1},\dots,x_{n};f)\mapsto f(x_{i})$, be the \emph{evaluation
morphisms}.
 
\begin{definition}
For $\gamma_{1},\dots,\gamma_{n}\in \rH^{*}(Y,\Q)$ the \emph{genus-zero
Gromov--Witten invariant} is
\[
  \big\langle\gamma_{1},\dots,\gamma_{n}\big\rangle^{Y}_{0,\beta}
  \;=\;\int_{[\overline{\mathcal M}_{0,n}(Y,\beta)]}
   \ \prod_{i=1}^{n}\ev_{i}^{*}\gamma_{i}\ \in\Q .
\]
It vanishes unless $\sum_{i}\deg\gamma_{i}=2\boldsymbol\delta(n,\beta)$, a
condition known in the literature as the \emph{dimension axiom}.
\end{definition}
 
\medskip
 
Assume now that $\cE$ is a globally generated vector bundle on $Y$.  Then
$\cE$ is \emph{convex}, that is
\[
  \rH^{1}(C,f^{*}\cE)=0
  \qquad\text{for every genus-zero stable map }f\colon C\to Y :
\]
indeed $f^{*}\cE$ is globally generated on $C$ and on every subcurve, hence
has vanishing $\rH^{1}$ on each component, and the normalization sequence
propagates the vanishing along the tree of components.  We set
$X=\mathscr Z(Y,\cE)\subset Y$, the zero locus of a general section, smooth
of codimension $\rank\cE$ by Bertini.
 
The first remark is that $X$ is generally \emph{not} convex: already a
smooth cubic surface $X\subset\PP^{3}$ fails the condition.  Indeed a line
$L\subset X$ is a $(-1)$-curve, so $N_{L/X}=\cO(-1)$ and the sequence
$0\to TL\to TX|_{L}\to N_{L/X}\to0$ splits, giving
$TX|_{L}=\cO(2)\oplus\cO(-1)$; a degree-$d$ cover
$f\colon\PP^{1}\to L$ then has $f^{*}TX=\cO(2d)\oplus\cO(-d)$ and hence
$\rH^{1}(\PP^{1},f^{*}TX)\neq0$ for $d\ge2$.  A priori, this causes
complications we wish to avoid: the moduli spaces
$\overline{\mathcal M}_{0,n}(X,\beta')$ may fail to have the expected
dimension, and one must replace the fundamental class by a surrogate
$\big[\overline{\mathcal M}_{0,n}(X,\beta')\big]^{\mathrm{vir}}$
\cites{Li1998, Behrend1997}.  The unfamiliar reader may keep in
mind that this surrogate behaves, for the purposes of intersection theory,
exactly as a fundamental class does; and, as we now explain, in our setup
it only ever appears on the left-hand side of an identity whose right-hand
side is an honest integral over the smooth stack
$\overline{\mathcal M}_{0,n}(Y,\beta)$.
 
Following Kim--Kresch--Pantev \cite{Kim2003}, we write $\cE_{0,n,\beta}$ for
the sheaf on $\overline{\mathcal M}_{0,n}(Y,\beta)$ whose fiber at a stable
map equivalence class $[(C;x_{1},\dots,x_{n};f)]$ is
$\rH^{0}(C,f^{*}\cE)$.  The requirement that $\cE$ is convex is exactly what
makes it a vector bundle of rank
$\rank\cE+\int_{\beta}c_{1}(\cE)$ by Riemann--Roch on a genus-zero curve.
Fix the following notation: $\iota\colon X\hookrightarrow Y$ stands for the
canonical inclusion, $e\big(\cE_{0,n,\beta}\big)$ for the Euler class of
$\cE_{0,n,\beta}\to\overline{\mathcal M}_{0,n}(Y,\beta)$, and, for each
$\beta'\in \rH_{2}(X,\Z)$ with $\iota_{*}\beta'=\beta$,
\[
  j_{\beta'}\colon\overline{\mathcal M}_{0,n}(X,\beta')
  \longrightarrow\overline{\mathcal M}_{0,n}(Y,\beta)
\]
for the morphism induced by $\iota$.  One has
\[
  \sum_{\beta'\,:\,\iota_{*}\beta'=\beta}
   \big(j_{\beta'}\big)_{*}
   \big[\overline{\mathcal M}_{0,n}(X,\beta')\big]^{\mathrm{vir}}
  \;=\;e\big(\cE_{0,n,\beta}\big)\cap
   \big[\overline{\mathcal M}_{0,n}(Y,\beta)\big],
\]
so that, for $\gamma_{i}=\iota^{*}\widetilde\gamma_{i}$ with
$\widetilde\gamma_{i}\in \rH^{*}(Y,\Q)$,
\[
  \sum_{\beta'\,:\,\iota_{*}\beta'=\beta}
  \big\langle\gamma_{1},\dots,\gamma_{n}\big\rangle^{X}_{0,\beta'}
  \;=\;\int_{[\overline{\mathcal M}_{0,n}(Y,\beta)]}
    e\big(\cE_{0,n,\beta}\big)\prod_{i}\ev_{i}^{*}\widetilde\gamma_{i}.
\]
 
In what follows, we shall explain how to use localization to compute the
above Gromov--Witten invariants for $Y$ and $X$.

\medskip

\subsection{Gromov--Witten invariants via localization}
 
In this section, we explain how Theorem \ref{thm:ABBV} can be applied to the
moduli of stable maps (as envisioned by Kontsevich \cite{Kontsevich1995}),
to express integrals over $\overline{\mathcal M}_{0,n}(Y,\beta)$
combinatorially. Here, because $Y$ is a flag variety, we will recognize
these quantities in terms of quantities related to roots. We then use them to
compute Gromov--Witten invariants for $Y$ (and for some complete
intersections $X$ on $Y$).
 
\medskip
 
We will need the following technical results.
 
\begin{proposition}\label{prop:invcurves}
For $w\in W_{Y}$ and $\alpha\in R^{+}_{Y}$, let $G_{\alpha}\subset G$ be the
connected subgroup with Lie algebra
$\mathfrak g_{\alpha}\oplus[\mathfrak g_{\alpha},\mathfrak g_{-\alpha}]\oplus\mathfrak g_{-\alpha}\cong\mathfrak{sl}_{2}$,
and set
\[
  C_{w,\alpha}\;:=\;w\,G_{\alpha}\,P/P\;\subset\;Y .
\]
Then $C_{w,\alpha}\cong\PP^{1}$ is the closure of a one-dimensional
$\mathbb T$-orbit, it joins $p_{w}$ to $p_{\overline{ws_{\alpha}}}$, its
tangent weight at $p_{w}$ is $-w\alpha$ and at the other endpoint
$+w\alpha$, and its class in $\rH_{2}(Y,\Z)=Q^{\vee}_{Y}$ is the image
$\overline{\alpha^{\vee}}$ of the coroot.  Conversely, every closure of a
one-dimensional $\mathbb T$-orbit is of this form.
\end{proposition}
 
\begin{proof}
Since $\alpha\in R^{+}$ we have $\mathfrak g_{\alpha}\subset\mathfrak p$,
while $\alpha\notin R_{P}$ gives
$\mathfrak g_{-\alpha}\not\subset\mathfrak p$. Hence, $G_{\alpha}\cap P$ is a
Borel subgroup of $G_{\alpha}$ and
$G_{\alpha}P/P\cong G_{\alpha}/(G_{\alpha}\cap P)\cong\PP^{1}$, which is
closed in $Y$.  It is $\mathbb T$-invariant because $\mathbb T$ normalizes
$G_{\alpha}$: for $t\in\mathbb T$ one has
$t\,wG_{\alpha}P/P=w\big(w^{-1}tw\big)G_{\alpha}P/P=wG_{\alpha}P/P$, since
$w^{-1}tw\in\mathbb T$.  Moreover $s_{\alpha}\in G_{\alpha}$ shows that
$C_{w,\alpha}$ contains $ws_{\alpha}P/P=p_{\overline{ws_{\alpha}}}$.  Its
tangent space at $p_{w}$ is the line $\mathfrak g_{-w\alpha}$ of
Proposition \ref{prop:tangentweights}, of weight $-w\alpha$. The two tangent
weights of a $\mathbb T$-equivariant $\PP^{1}$ are opposite, giving
$+w\alpha$ at the other end.  For the class, pair with $\cO_{Y}(\lambda)$:
by Proposition \ref{prop:linebundleweights}(b) below, whose proof uses only
the endpoints and the tangent weight established above,
$\int_{C_{w,\alpha}}c_{1}(\cO_{Y}(\lambda))=\langle\lambda,\alpha^{\vee}\rangle$
for every $\lambda\in\Lambda_{Y}$, which by \eqref{Intbeta} characterizes
$\overline{\alpha^{\vee}}$.
 
Conversely, let $C$ be the closure of a one-dimensional orbit.  Being a
complete $\mathbb T$-curve it contains a fixed point, say $p_{w}$; after
translating by $w^{-1}$, which is $\mathbb T$-equivariant up to the
automorphism $t\mapsto w^{-1}tw$ of $\mathbb T$, we may assume $w=e$.  The
big cell $\Omega=\prod_{\beta\in R^{+}_{Y}}U_{-\beta}\,P/P$ is a
$\mathbb T$-invariant affine neighbourhood of $eP$, equivariantly isomorphic
to the linear representation
$\bigoplus_{\beta\in R^{+}_{Y}}\mathfrak g_{-\beta}$.  If a point of
$\Omega$ had two nonzero components, with characters $-\beta$ and
$-\beta'$, its orbit would have dimension equal to the rank of the subgroup
of the character lattice generated by $\beta$ and $\beta'$, which is $2$
since $\beta$ and $\beta'$ are non-proportional by
Proposition \ref{prop:tangentweights}.  Hence $C\cap\Omega$ is one of the
coordinate lines $\mathfrak g_{-\alpha}$, and
$C=\overline{U_{-\alpha}\,eP/P}=G_{\alpha}P/P$.
\end{proof}
 
\begin{lemma}\label{lem:simplegraph}
Fix $w\in W_{Y}$.  The map
$R^{+}_{Y}\to W_{Y}$, $\alpha\mapsto\overline{ws_{\alpha}}$, is injective.
Equivalently, an invariant curve is determined by its pair of endpoints, and
the graph has vertices $W_{Y}$ and edges are the invariant curves.
\end{lemma}
 
\begin{proof}
Cancelling $w$ on the left, the condition
$\overline{ws_{\alpha}}=\overline{ws_{\alpha'}}$ reads
$s_{\alpha}W_{P}=s_{\alpha'}W_{P}$ and does not involve $w$.  Put
$\lambda_{Y}:=\sum_{\alpha_{i}\in\Pi_{Y}}\omega_{i}$.  Then
$\mathrm{Stab}_{W}(\lambda_{Y})=W_{P}$, and for $\gamma\in R^{+}$ one has
$\langle\lambda_{Y},\gamma^{\vee}\rangle>0$ exactly when
$\gamma\notin R_{P}$.  Therefore
\[
  s_{\alpha}W_{P}=s_{\alpha'}W_{P}
  \iff s_{\alpha}\lambda_{Y}=s_{\alpha'}\lambda_{Y}
  \iff \langle\lambda_{Y},\alpha^{\vee}\rangle\,\alpha
      =\langle\lambda_{Y},\alpha'^{\vee}\rangle\,\alpha' ,
\]
and both coefficients are strictly positive because
$\alpha,\alpha'\in R^{+}_{Y}$.  Hence $\alpha$ and $\alpha'$ are
proportional roots, so $\alpha=\alpha'$.
\end{proof}
 
\begin{remark}\label{rmk:basepoint}
By Lemma \ref{lem:simplegraph}, an invariant curve is determined by its two
endpoints, but the root that labels it is \emph{not} intrinsic: it depends
on which endpoint one starts from.  If $Y=G/B$, the two labels always agree,
since no projection is involved and $ws_{\alpha}\cdot s_{\alpha}=w$.  For a
proper parabolic, they may differ.  This is the reason why, in the
combinatorial description of the fixed loci below, an edge need carry no
root at all: by Lemma \ref{lem:simplegraph} the root is already determined
by the labels of the two vertices it joins.
\end{remark}
 
Next, let us explain how the localization shall be performed.  We now carry
this out in our setting, following the theory developed by Kontsevich in
\cite{Kontsevich1995}.

\medskip

\subsubsection{The torus action on the moduli of stable maps}

The torus $\mathbb T$ acts on $\overline{\mathcal M}_{0,n}(Y,\beta)$ by
post-composition,
\[
  t\cdot\big[(C;x_{1},\dots,x_{n};f)\big]
  :=\big[(C;x_{1},\dots,x_{n};t\circ f)\big],
\]
where $t\circ f$ denotes $f$ followed by the automorphism $y\mapsto ty$ of
$Y$.  A point is fixed when for every $t$ there is an automorphism
$\varphi_{t}$ of $(C;x_{1},\dots,x_{n})$ with $f\circ\varphi_{t}=t\circ f$.

\begin{proposition}\label{prop:fixedmaps}
A stable map $(C;x_{1},\dots,x_{n};f)$ is $\mathbb T$-fixed if and only if
\begin{defenum}
\item every irreducible component of $C$ is either contracted by $f$ to a
      point of $Y^{\mathbb T}$, or maps onto an invariant curve
      $C_{w,\alpha}$ of Proposition \ref{prop:invcurves};
\item in the second case the induced map $\PP^{1}\to C_{w,\alpha}\cong\PP^{1}$
      is, in suitable coordinates, $z\mapsto z^{d}$ for some $d\ge1$, so it
      is totally ramified over the two fixed points;
\item every node and every marked point is mapped into $Y^{\mathbb T}$.
\end{defenum}
\end{proposition}

\begin{proof}
Suppose the map is fixed.  The image of each irreducible component of $C$ is
an irreducible $\mathbb T$-invariant subvariety of $Y$ of dimension $0$ or
$1$, hence either a fixed point or the closure of a one-dimensional orbit;
Proposition \ref{prop:invcurves} identifies the latter with the curves
$C_{w,\alpha}$.  On a non-contracted component, $f$ is a finite
$\mathbb T$-equivariant map $\PP^{1}\to\PP^{1}$ up to reparameterization.
Such a map preserves the two fixed points and is $z\mapsto z^{d}$ in
coordinates in which the action is $z\mapsto\chi(t)z$.  Moreover,
$t\mapsto\varphi_{t}$ defines an action of $\mathbb T$ on $C$ permuting the
finite set of nodes and marked points. Since $\mathbb T$ is connected, each
of them is fixed by every $\varphi_{t}$, so its image is a
$\mathbb T$-fixed point of $Y$.  The converse is immediate, taking for
$\varphi_{t}$ the reparameterization of each non-contracted component
determined by $t$.
\end{proof}

\medskip

We now define the decorated trees, which we shall
recognize in Proposition \ref{prop:Fgamma} as parameterizing the fixed point
set for the $\mathbb T$-action on $\overline{\mathcal M}_{0,n}(Y,\beta)$.

\begin{definition}\label{def:decoratedtree}
Let $n\ge0$ and let $\beta\in \rH_{2}(Y,\Z)$ be effective.  A
\emph{decorated tree of type $(n,\beta)$} is a tuple
$\Gamma=(\tau,w,d,\mathrm{mk})$ where
\begin{defenum}
\item $\tau$ is a tree with vertex set $V$ and edge set $E$;
\item $w\colon V\to W_{Y}$ is a labelling such that the endpoints of every
      edge are joined by an invariant curve, that is, for every edge
      $e=\{v,v'\}$ there is a root $\alpha_{e,v}\in R^{+}_{Y}$
      (unique by Lemma \ref{lem:simplegraph}) with
      $w_{v'}=\overline{w_{v}s_{\alpha_{e,v}}}$;
\item $d\colon E\to\Z_{>0}$ assigns a degree to each edge;
\item $\mathrm{mk}\colon\{1,\dots,n\}\to V$ distributes the marked points,
\end{defenum}
subject to the degree constraint
\[
  \sum_{e\in E}d_{e}\,\overline{\alpha^{\vee}_{e,v}}\;=\;\beta
  \qquad\text{in }\ Q^{\vee}_{Y}\cong \rH_{2}(Y,\Z).
\]
An isomorphism of decorated trees is an isomorphism of trees preserving
$w$, $d$ and $\mathrm{mk}$. We write $\Aut\Gamma$ for the automorphism
group of $\Gamma$, $\mathrm{val}(v)$ for the valence of a vertex and
$n_{v}=\mathrm{val}(v)+|\mathrm{mk}^{-1}(v)|$ for its total number of
special points.
\end{definition}

\begin{remark}
The degree constraint is well-posed although the root is not: by
Remark \ref{rmk:basepoint} the roots $\alpha_{e,v}$ and $\alpha_{e,v'}$ read
from the two ends of an edge may differ, but by
Proposition \ref{prop:invcurves} their coroots have the same class in
$Q^{\vee}_{Y}$, namely that of the invariant curve.
\end{remark}

\begin{proposition}\label{prop:Fgamma}
The $\mathbb T$-fixed locus of $\overline{\mathcal M}_{0,n}(Y,\beta)$ is the
disjoint union, over the isomorphism classes of decorated trees of type
$(n,\beta)$, of the substacks
\[
  F_{\Gamma}\;\cong\;\Big[\ \prod_{v\,:\,n_{v}\ge3}\overline{\mathcal M}_{0,n_{v}}
  \ \Big/\ \mathbb A_{\Gamma}\ \Big],
\]
where $\mathbb A_{\Gamma}$ is an extension of $\Aut\Gamma$ by
$\prod_{e\in E}\Z/d_{e}$, the second factor acting trivially.  In
particular, for every class $\psi$ on $F_{\Gamma}$,
\[
  \int_{F_{\Gamma}}\psi
  =\frac{1}{|\Aut\Gamma|\ \prod_{e\in E}d_{e}}
   \int_{\prod_{v\,:\,n_{v}\ge3}\overline{\mathcal M}_{0,n_{v}}}\psi .
\]
Moreover the number of decorated trees of type $(n,\beta)$ is finite: if
$\beta=\sum_{j}b_{j}\overline{\alpha^{\vee}_{i_{j}}}$ then
$|E|\le\sum_{j}b_{j}$ and $|V|\le\sum_{j}b_{j}+1$.
\end{proposition}

\begin{proof}
By Proposition \ref{prop:fixedmaps} a $\mathbb T$-fixed stable map
determines a decorated tree of type $(n,\beta)$, and conversely $\Gamma$
determines the map except for the moduli of its contracted components.  A
component is contracted at a vertex with $n_{v}\ge3$, since a
contracted component must carry at least three special points by
stability, and it contributes a factor $\overline{\mathcal M}_{0,n_{v}}$;
the components covering the invariant curves carry no moduli, being totally
ramified covers of a fixed $\PP^{1}$ over the two fixed points.  The deck
group $\Z/d_{e}$ of each such cover acts trivially on the remaining data and
therefore lies in the automorphism group of every point of $F_{\Gamma}$,
while $\Aut\Gamma$ permutes the factors of the product. This gives the
extension $\mathbb A_{\Gamma}$, and the displayed formula for
$\int_{F_{\Gamma}}$ is the usual one for an integral over a quotient stack.

For the finiteness, each $\alpha\in R^{+}_{Y}$ has $\alpha^{\vee}$ with
nonnegative coordinates in the basis of simple coroots, and
$\overline{\alpha^{\vee}}\neq0$ in $Q^{\vee}_{Y}$ because it is the class of
a curve.  Hence, every edge contributes at least $1$ to $\sum_{j}b_{j}$, so
$|E|\le\sum_{j}b_{j}$, and $|V|=|E|+1$ since $\tau$ is a tree.
\end{proof}

With Proposition \ref{prop:Fgamma} in hand, obtaining Gromov--Witten
invariants of flag varieties by localization requires, according to
Theorem \ref{thm:ABBV}, only that we understand the normal bundles of the
fixed components $F_{\Gamma}$ and their equivariant Euler classes.  Even
more: since we are in the flag setting, the restrictions of the Schubert
classes to the fixed points are heavily constrained, and the resulting
formulae simplify accordingly.

Indeed, the classes $\sigma_{w}$ of Section~\ref{sec:flags} admit canonical
equivariant lifts $\sigma^{w}\in \rH^{2\ell(w)}_{\mathbb T}(Y,\Q)$,
characterized by $\sigma^{w}|_{p_{v}}=0$ unless $w\le v$ in the Bruhat
order --- as one expects from Remark \ref{rmk:bruhatgeom}, since $p_{v}$
then lies outside the support of the cycle.  Their restrictions to the fixed
points are given by a closed combinatorial formula
\cite{tymoczko2013billeysformulacombinatoricsgeometry}.

\begin{theorem}[Billey's Formula]\label{thm:billey}
Fix a reduced expression $v=s_{v_{1}}\cdots s_{v_{\ell}}$ and put
$\gamma_{i}=s_{v_{1}}\cdots s_{v_{i-1}}(\alpha_{v_{i}})$, a positive root
for each $i$.  Then, for every $u\in W$,
\[
  \sigma^{u}\big|_{p_{v}}
  =\sum_{\substack{1\le i_{1}<\dots<i_{\ell(u)}\le \ell(v)\\
      s_{v_{i_{1}}}\cdots s_{v_{i_{\ell(u)}}}=u}}\ \prod_{k=1}^{\ell(u)}\gamma_{i_{k}} ,
\]
independently of the chosen reduced expression.  In particular, the sum is
empty, so that the restriction vanishes, unless $u\le v$.
\end{theorem}

\medskip

\subsubsection{Weights along the invariant curves}
 
By Theorem \ref{thm:ABBV}, everything we shall compute is assembled from
weights: the weight of the torus on each fiber at a fixed point, and the
degree of each bundle along each invariant curve.  This subsection collects
those data for the two bundles we need, the homogeneous line bundles
$\cO_{Y}(\lambda)$ and the tangent bundle $TY$.
 
We first fix the vocabulary.  A $\mathbb T$-\emph{equivariant} bundle on a
$\mathbb T$-variety $Z$ is a bundle $E\to Z$ together with a lift of the
action to the total space, linear on fibers.  If $z\in Z^{\mathbb T}$, the
fiber $E_{z}$ is then a $\mathbb T$-representation; when $E$ is a line
bundle, this representation is a character, which we call the \emph{fiber
weight} of $E$ at $z$ and regard as an element of
$\mathfrak t^{*}=\rH^{2}_{\mathbb T}(\mathrm{pt})$.  The bundles occurring
here are all associated bundles $E=G\times_{P}V$, with the convention of
Section~\ref{sec:picard}, and they are equivariant for the action
$t\cdot[g,v]=[tg,v]$ induced by left translation.
 
On a curve $C\cong\PP^{1}$ with a nontrivial action, the two fiber weights of
a line bundle determine its degree, which is the source of every degree
computation below.
 
\begin{lemma}\label{lem:degreeP1}
Let $C\cong\PP^{1}$ carry a nontrivial $\mathbb T$-action with fixed points
$q_{0},q_{\infty}$, and let $\lambda_{C}$ be the weight of $T_{q_{0}}C$.  If
$L$ is a $\mathbb T$-equivariant line bundle on $C$ with fiber weights
$\lambda_{0}$ at $q_{0}$ and $\lambda_{\infty}$ at $q_{\infty}$, then
\[
  \deg L\;=\;\frac{\lambda_{0}-\lambda_{\infty}}{\lambda_{C}} .
\]
\end{lemma}
 
\begin{proof}
Write $C=\PP(V)$ with $V=\C_{\chi_{0}}\oplus\C_{\chi_{\infty}}$ and
$q_{0}=[\C_{\chi_{0}}]$, so that
$T_{q_{0}}C=\mathrm{Hom}(\C_{\chi_{0}},\C_{\chi_{\infty}})$ has weight
$\lambda_{C}=\chi_{\infty}-\chi_{0}$.  The tautological bundle $\cO(-1)$,
with its natural linearization, has fiber weights $\chi_{0}$ and
$\chi_{\infty}$, so the formula holds for it, both sides being $-1$.  Every
$\mathbb T$-equivariant line bundle on $\PP^{1}$ is of the form
$\cO(n)\otimes\C_{\chi}$ for some $n\in\Z$ and some character $\chi$. Both
sides of the asserted identity are additive under tensor products and vanish
on the constant twist $\C_{\chi}$, so the general case follows.
\end{proof}
 
We now apply this to $Y$.  Part (a) below is the computation that fixes all
our signs: it says that the fiber weights of $\cO_{Y}(\lambda)$ are the
$W$-translates of $-\lambda$.  Parts (b) and (c) then read off the degrees
along the invariant curves of Proposition \ref{prop:invcurves}.
 
\begin{proposition}\label{prop:linebundleweights}
Let $\lambda\in\Lambda_{Y}$ and let
$\cO_{Y}(\lambda)=G\times_{P}\C_{-\lambda}$ carry the $\mathbb T$-action
$t\cdot[g,a]=[tg,a]$.  Then
\begin{defenum}
  \item the fiber weight of $\cO_{Y}(\lambda)$ at $p_{w}$ is $-w\lambda$;
  \item for every $w\in W_{Y}$ and $\alpha\in R^{+}_{Y}$,
        \[
          \deg\big(\cO_{Y}(\lambda)\big|_{C_{w,\alpha}}\big)
          =\langle\lambda,\alpha^{\vee}\rangle ,
        \]
        independently of $w$;
  \item $TY|_{C_{w,\alpha}}$ carries a $\mathbb T$-equivariant filtration
        whose graded pieces are line bundles $L_{\gamma}$, indexed by
        $\gamma\in R^{+}_{Y}$, with fiber weight $-w\gamma$ at $p_{w}$ and
        \[
          \deg L_{\gamma}=\langle\gamma,\alpha^{\vee}\rangle .
        \]
\end{defenum}
\end{proposition}
 
\begin{proof}
(a)  With the identification $(g,a)\sim(gq,\chi(q)^{-1}a)$ for
$\chi=-\lambda$, as fixed in Section~\ref{sec:picard}, write
$tw=w(w^{-1}tw)$ with $w^{-1}tw\in\mathbb T\subset P$; then
$t\cdot[w,a]=[w,\chi(w^{-1}tw)a]$ and
$\chi(w^{-1}tw)=(w\chi)(t)=(-w\lambda)(t)$.  The same computation shows,
more generally, that $G\times_{P}\C_{\chi}$ has fiber weight $w\chi$ at
$p_{w}$.
 
(b)  Apply Lemma \ref{lem:degreeP1} with $q_{0}=p_{w}$,
$q_{\infty}=p_{\overline{ws_{\alpha}}}$ and $\lambda_{C}=-w\alpha$ by
Proposition \ref{prop:invcurves}.  By part (a), the fiber weights are
$-w\lambda$ and $-\overline{ws_{\alpha}}\lambda$; the latter equals
$-ws_{\alpha}\lambda$, since $\langle\lambda,\alpha_{j}^{\vee}\rangle=0$ for
every $\alpha_{j}\in\Pi_{P}$ by Definition \ref{def:fundweights}, so that
$W_{P}$ fixes $\lambda$.  Hence
\[
  \deg=\frac{-w\lambda+ws_{\alpha}\lambda}{-w\alpha}
      =\frac{-w\big(\lambda-s_{\alpha}\lambda\big)}{-w\alpha}
      =\frac{-\langle\lambda,\alpha^{\vee}\rangle\,w\alpha}{-w\alpha}
      =\langle\lambda,\alpha^{\vee}\rangle ,
\]
using $s_{\alpha}\lambda=\lambda-\langle\lambda,\alpha^{\vee}\rangle\alpha$.
The answer does not involve $w$, as it must not: by
Proposition \ref{prop:invcurves} all the curves $C_{w,\alpha}$ with the same
$\alpha$ have the same homology class.
 
(c)  Write $B_{\alpha}=G_{\alpha}\cap P$, a Borel subgroup of
$G_{\alpha}\cong\mathrm{SL}_{2}$ by the proof of
Proposition \ref{prop:invcurves}, so that
$C_{w,\alpha}=w\,G_{\alpha}/B_{\alpha}$ and
\[
  TY\big|_{C_{w,\alpha}}
  =w\cdot\Big(G_{\alpha}\times_{B_{\alpha}}\big(\mathfrak g/\mathfrak p\big)\Big),
\]
translation by $w$ being equivariant up to the automorphism
$t\mapsto w^{-1}tw$ of $\mathbb T$ --- which is precisely what turns a weight
$\mu$ into $w\mu$.  Since $B_{\alpha}$ is solvable, the $B_{\alpha}$-module
$\mathfrak g/\mathfrak p$ admits a filtration by submodules with
one-dimensional graded pieces, whose $\mathfrak t$-weights are the
$-\gamma$ with $\gamma\in R^{+}_{Y}$; after translation these become
$-w\gamma$, which is part (a) again.  For the degrees, the line bundle
$G_{\alpha}\times_{B_{\alpha}}\C_{\mu}$ on
$G_{\alpha}/B_{\alpha}\cong\PP^{1}$ has fiber weights $\mu$ and
$s_{\alpha}\mu$ at the two fixed points by (a), and tangent weight
$-\alpha$ at the first, so Lemma \ref{lem:degreeP1} gives
\[
  \deg\big(G_{\alpha}\times_{B_{\alpha}}\C_{\mu}\big)
  =\frac{\mu-s_{\alpha}\mu}{-\alpha}
  =-\langle\mu,\alpha^{\vee}\rangle ,
\]
which for $\mu=-\gamma$ is $\langle\gamma,\alpha^{\vee}\rangle$.
\end{proof}
 
\begin{remark}\label{rmk:filtration}
In (c) one cannot in general replace ``filtration'' by ``direct sum'': the
degrees $\langle\gamma,\alpha^{\vee}\rangle$ are those of the graded pieces,
not of the summands.  On $\mathrm{Fl}(1,2,3)$ with $\alpha=\alpha_{2}$ they
are $-1,2,1$, whereas $TY|_{C}\cong\cO(2)\oplus\cO\oplus\cO$: the
projection $\mathrm{Fl}(1,2,3)\to\PP^{2}$ forgetting $V_{2}$ has $C$ as a
fiber, so that $TY|_{C}$ is an extension of the trivial bundle
$T_{[e_{1}]}\PP^{2}\otimes\cO_{C}$ by $T_{C}=\cO(2)$, and this extension
splits because $\rH^{1}(\PP^{1},\cO(2))=0$.  Both readings give total degree
$2$, as they must.  The discrepancy costs us nothing below: all we shall use
is the class of $TY|_{C_{w,\alpha}}$ in equivariant $K$-theory, and Euler
classes are multiplicative in exact sequences.
\end{remark}
 
The same argument applies verbatim to any homogeneous bundle, which is what
we shall need for the complete intersections of Section~\ref{sec:ci}.  Let
$\cE=G\times_{P}V$ be associated with a $P$-module $V$ whose
$\mathbb T$-weights are $-\mu_{1},\dots,-\mu_{s}$.  Then $\cE$ has fiber
weights $-w\mu_{1},\dots,-w\mu_{s}$ at $p_{w}$, and
$\cE|_{C_{w,\alpha}}$ carries a $\mathbb T$-equivariant filtration whose
graded pieces have degrees
$\langle\mu_{1},\alpha^{\vee}\rangle,\dots,\langle\mu_{s},\alpha^{\vee}\rangle$.
For $\cE=\bigoplus_{j}\cO_{Y}(\lambda_{j})$ as in Theorem \ref{thm:ci} the
restriction genuinely splits, with degrees
$\langle\lambda_{j},\alpha^{\vee}\rangle$, all nonnegative when the
$\lambda_{j}$ are dominant.
 
\medskip
 
We need to introduce a last notation. A component of the domain of a fixed stable map
covers an invariant curve with some degree $d_{e}$, and it is the tangent
weight \emph{upstairs} (the tangent weight downstairs divided by $d_{e}$) on that component that enters the localization
formula.
 
\begin{definition}\label{def:omega}
Let $\Gamma$ be a decorated tree and $e$ an edge of $\Gamma$ with endpoints $v,v'$.  The
\emph{normalized tangent weight} of $e$ at $v$ is
\[
  \omega_{v,e}\;:=\;-\frac{1}{d_{e}}\,w_{v}\,\alpha_{e,v}
  \;\in\;\mathfrak t^{*}\otimes\Q ,
\]
the weight of the tangent line to the corresponding component of the domain
at the point lying over $p_{w_{v}}$.  One has
$\omega_{v',e}=-\omega_{v,e}$.
\end{definition}
% Indeed, the tangent weight of $C_{w_{v},\alpha_{e,v}}$ at $p_{w_{v}}$ is
% $-w_{v}\alpha_{e,v}$ by Proposition \ref{prop:invcurves} and is opposite at
% the other endpoint, while a degree-$d_{e}$ cover totally ramified over the
% two fixed points divides both by $d_{e}$.
 
\medskip
 
\subsubsection{Integrals over $\overline{\mathcal M}_{0,m}$}

Up to this point, every quantity we have computed has been a weight, that is,
an element of $\Q[x_{1},\dots,x_{r}]$.  This is no longer enough.  By
Proposition \ref{prop:Fgamma} the fixed locus $F_{\Gamma}$ is a product of
moduli spaces $\overline{\mathcal M}_{0,n_{v}}$, one for each vertex with
$n_{v}\ge3$, and $\dim\overline{\mathcal M}_{0,m}=m-3$. Thus, as soon as some
vertex carries four or more special points, $F_{\Gamma}$ is
positive-dimensional and Theorem \ref{thm:ABBV} asks us to \emph{integrate}
over it rather than merely to evaluate at a point.  Moreover. the integrand
$1/e^{\mathbb T}(N_{\Gamma})$ genuinely varies along $F_{\Gamma}$, for the
following reason.

Among the normal directions to $F_{\Gamma}$ are the deformations smoothing
the nodes of the domain.  At the node joining a contracted component $C_{v}$
to an edge component $C_{e}$, the smoothing parameter lives in the tensor
product of the two tangent lines,
\[
  T_{\mathrm{node}}C_{v}\ \otimes\ T_{\mathrm{node}}C_{e}.
\]
The second factor is constant on $F_{\Gamma}$: it is the tangent line to a
totally ramified cover of an invariant curve, of weight $\omega_{v,e}$ by
Definition \ref{def:omega}.  The first factor is not constant: as the
contracted curve $C_{v}$ moves in $\overline{\mathcal M}_{0,n_{v}}$, its
tangent line at the node moves with it.  It is precisely the line bundle
whose first Chern class is the $\psi$ class of the moduli space.

\begin{definition}\label{def:psi}
Let $m\ge3$.  For $1\le i\le m$ let $\mathbb L_{i}$ be the line bundle on
$\overline{\mathcal M}_{0,m}$ whose fiber at a point
$[(C;x_{1},\dots,x_{m})]$ is the cotangent line $T^{*}_{x_{i}}C$, and set
\[
  \psi_{i}:=c_{1}(\mathbb L_{i})\ \in\ \rH^{2}\big(\overline{\mathcal M}_{0,m},\Q\big).
\]
\end{definition}

With this, the node contributes $\omega_{v,e}-\psi_{(v,e)}$ to
$e^{\mathbb T}(N_{\Gamma})$: the Euler class of a tensor product of line
bundles is the sum of their first Chern classes, the equivariant one
contributing the weight $\omega_{v,e}$ and the moduli one contributing
$-\psi_{(v,e)}$, the sign because $\mathbb L$ is the \emph{co}tangent line.
Here we index the $m=n_{v}$ special points of $C_{v}$ by the flags $(v,e)$
at $v$ together with the marks in $\mathrm{mk}^{-1}(v)$.

Two remarks are in order before we compute.  First, the torus acts trivially
on $F_{\Gamma}$, so that
$\rH^{*}_{\mathbb T}(F_{\Gamma},\Q)=\rH^{*}(F_{\Gamma},\Q)\otimes\Q[x_{1},\dots,x_{r}]$:
the $\psi$ classes are ordinary cohomology classes of the fixed locus, the
$\omega_{v,e}$ are pure weights, and the two enter the formula as
independent variables.  Second, nothing
in this subsection depends on $Y$.  The spaces $\overline{\mathcal M}_{0,m}$
and their $\psi$ classes are the same whatever the target, and the flag
variety enters only through the constants $\omega_{v,e}$.  This is the one
step of the computation where the combinatorics of the root system plays no
role at all, and it is why we may quote a classical evaluation
\cite{Kontsevich1995}:
\[
  \int_{\overline{\mathcal M}_{0,m}}\psi_{1}^{a_{1}}\cdots\psi_{m}^{a_{m}}
  =\binom{m-3}{a_{1},\dots,a_{m}},\qquad \sum_{i}a_{i}=m-3 .
\]

What we shall actually need is the following packaged form of it.  Inverting
the node contributions produces, at a vertex $v$, the factor
$\prod_{e\ni v}(\omega_{v,e}-\psi_{(v,e)})^{-1}$, in which the product runs
over the flags at $v$ but \emph{not} over the marks. The marks, nevertheless,
enlarge the moduli space, and hence its dimension.  The lemma is stated so
as to keep the two roles apart: $S$ is the set of flags, and
$\{1,\dots,m\}\smallsetminus S$ the set of marks.

\begin{lemma}\label{lem:M0n}
Let $m\ge3$, let $S\subseteq\{1,\dots,m\}$, and let $(\omega_{i})_{i\in S}$
be invertible elements of $\Q[x_{1},\dots,x_{r}]$ localized at the nonzero
characters.  Then
\[
  \int_{\overline{\mathcal M}_{0,m}}\ \prod_{i\in S}\frac{1}{\omega_{i}-\psi_{i}}
  =\Big(\prod_{i\in S}\omega_{i}^{-1}\Big)
   \Big(\sum_{i\in S}\omega_{i}^{-1}\Big)^{m-3} .
\]
\end{lemma}

\begin{proof}
Expand each factor as a geometric series,
$\frac{1}{\omega_{i}-\psi_{i}}=\omega_{i}^{-1}\sum_{a\ge0}(\psi_{i}/\omega_{i})^{a}$,
which terminates because $\psi_{i}$ is nilpotent,
$\dim\overline{\mathcal M}_{0,m}=m-3$.  Only the terms of total degree
$m-3$ survive the integration, and by the evaluation above they contribute
\[
  \Big(\prod_{i\in S}\omega_{i}^{-1}\Big)
  \sum_{\substack{(a_{i})_{i\in S}\\ \sum_{i}a_{i}=m-3}}
  \binom{m-3}{(a_{i})_{i\in S}}\ \prod_{i\in S}\omega_{i}^{-a_{i}},
\]
which is the multinomial expansion of
$\big(\prod_{i\in S}\omega_{i}^{-1}\big)\big(\sum_{i\in S}\omega_{i}^{-1}\big)^{m-3}$.
\end{proof}

Two special cases are worth recording, both of which occur constantly in
practice.  If $m=3$, the moduli space is a point, the exponent $m-3$ is zero, and the integral is simply $\prod_{i\in S}\omega_{i}^{-1}$: a vertex with
exactly three special points contributes no $\psi$ class.  If $S=\emptyset$
and $m>3$ the right-hand side is $0^{m-3}=0$, in agreement with
$\int_{\overline{\mathcal M}_{0,m}}1=0$ for dimension reasons.

With Corollary \ref{cor:ABBVflag}, Theorem \ref{thm:billey},
Proposition \ref{prop:Fgamma} and Lemma \ref{lem:M0n} in place, all the
ingredients of the localization formula have been assembled: it remains to
compute the Euler class of the normal bundle of $F_{\Gamma}$, which is the
content of the next subsection.

\medskip

\subsubsection{The Euler class of the normal bundle}
 
Only one ingredient of Theorem \ref{thm:ABBV} is still missing, namely
$e^{\mathbb T}(N_{\Gamma})$.  The strategy is dictated by the deformation
theory of stable maps: the tangent space to
$\overline{\mathcal M}_{0,n}(Y,\beta)$ at a fixed stable map is built from
three pieces: the deformations of the map, which are
$\rH^{0}(C,f^{*}TY)$; the infinitesimal automorphisms of the pointed
domain; and the deformations of the pointed domain, which smooth its nodes.
Each of the three is computed by normalizing $C$ and reading off weights one
component at a time, which is exactly what the previous subsections have
prepared.  The normal bundle is then the part of the result on which the
torus acts nontrivially.
 
We therefore fix the following vocabulary.  A $\mathbb T$-representation $M$
splits canonically as $M=M^{\mathrm{fix}}\oplus M^{\mathrm{mov}}$, the fixed
part being the weight-zero subspace and the \emph{moving part} the sum of
the nonzero weight spaces; only the latter has invertible equivariant Euler
class, and it is the moving part of the tangent space that constitutes the
normal bundle to the fixed locus.  We write
$\mathrm{aut}(C,x_{\bullet})=\rH^{0}\big(C,T_{C}(-\textstyle\sum_{i}x_{i})\big)$
for the infinitesimal automorphisms of the pointed domain and
$\mathrm{def}(C,x_{\bullet})$ for its first-order deformations, the latter
being the direct sum over the nodes of the tensor products of the two
tangent lines at each node.
 
Throughout this subsection, $\Gamma$ is a decorated tree of type
$(n,\beta)$, $F_{\Gamma}\subset\overline{\mathcal M}_{0,n}(Y,\beta)$ the
corresponding fixed substack of Proposition \ref{prop:Fgamma}, and
$N_{\Gamma}$ its normal bundle.  We abbreviate the Euler class of
Proposition \ref{prop:tangentweights} by
\[
  e_{w}\;:=\;e^{\mathbb T}\big(T_{p_{w}}Y\big)
  \;=\;\prod_{\gamma\in R^{+}_{Y}}\big(-w\gamma\big),
  \qquad w\in W_{Y}.
\]
Recall that a point of $F_{\Gamma}$ is a stable map whose domain $C$ has one
component $C_{e}\cong\PP^{1}$ for each edge, mapping to
$C_{w_{v},\alpha_{e,v}}$ with degree $d_{e}$ and totally ramified over the
two fixed points, and one contracted component $C_{v}$ for each vertex with
$n_{v}\ge3$.
 
We begin with the edges, where all the cohomology sits.
 
\begin{lemma}\label{lem:edgeweights}
Let $e$ be an edge, $v$ one of its endpoints, and write $w=w_{v}$,
$\alpha=\alpha_{e,v}$, $d=d_{e}$ and $\omega=\omega_{v,e}$.  Let $L$ be a
$\mathbb T$-equivariant line bundle on $C_{w,\alpha}$ with fiber weight
$\lambda_{0}$ at $p_{w}$, and set $m=d\cdot\deg L$.  Then
$\deg\big(f^{*}L|_{C_{e}}\big)=m$ and
\[
  \rH^{0}\big(C_{e},f^{*}L\big)\ \text{has weights}\
    \{\lambda_{0}-k\,\omega\}_{k=0}^{m}\quad (m\ge0),
\]
\[
  \rH^{1}\big(C_{e},f^{*}L\big)\ \text{has weights}\
    \{\lambda_{0}+j\,\omega\}_{j=1}^{-m-1}\quad (m\le-2),
\]
both being zero for $m=-1$.  Pulling back the filtration of
Proposition \ref{prop:linebundleweights}(c) along $f$ gives a
$\mathbb T$-equivariant filtration of $f^{*}TY|_{C_{e}}$ whose graded
pieces are the $f^{*}L_{\gamma}$, with $\lambda_{0}=-w\gamma$ and
\[
  m_{\gamma}\;=\;d\,\langle\gamma,\alpha^{\vee}\rangle .
\]
\end{lemma}
 
\begin{proof}
The map $C_{e}\to C_{w,\alpha}$ is $\mathbb T$-equivariant of degree $d$ and
sends the two fixed points of $C_{e}$ to those of $C_{w,\alpha}$, so the
fiber weights of $f^{*}L$ are those of $L$ while the tangent weight is
divided by $d$; this is exactly $\omega$ by Definition \ref{def:omega}.  Now
apply Lemma \ref{lem:degreeP1} to $C_{e}$, which gives the degree, and
expand $\rH^{0}(\PP^{1},\cO(m))$ and $\rH^{1}(\PP^{1},\cO(m))$ in weights as
in the proof of that lemma.  Pullback is exact on the filtration and
multiplies degrees by $d$, whence the value of $m_{\gamma}$ from
Proposition \ref{prop:linebundleweights}(c).
\end{proof}
 
The direction $\gamma=\alpha$ requires separate attention.  There
$m_{\alpha}=2d$, and the weight with $k=d$ is $-w\alpha-d\,\omega=0$: it is
the infinitesimal automorphism of $C_{e}$ fixing its two special points, so
it lies in the fixed part and is discarded when we pass to moving parts.  It
is the only weight that has to be discarded, since
\[
  -w\gamma-k\,\omega=-w\Big(\gamma-\tfrac{k}{d}\,\alpha\Big)
\]
vanishes only if $\gamma$ is proportional to $\alpha$, which for two
distinct positive roots of a reduced root system does not happen
(Proposition \ref{prop:tangentweights}).  The product of the surviving
weights in that direction is a constant multiple of a power of $\omega$:
 
\begin{lemma}\label{lem:constant}
$\displaystyle\prod_{\substack{k=0\\ k\neq d}}^{2d}
   \big(-w\alpha-k\,\omega\big)
   \;=\;(-1)^{d}\,(d!)^{2}\,\omega^{2d}.$
\end{lemma}
 
\begin{proof}
By Definition \ref{def:omega}, $-w\alpha=d\,\omega$, so the $k$-th factor is
$(d-k)\,\omega$. There are $2d$ of them, and
$\prod_{k\neq d}(d-k)=d!\cdot\prod_{j=1}^{d}(-j)=(-1)^{d}(d!)^{2}$.
\end{proof}
 
We can now name the two blocks out of which $e^{\mathbb T}(N_{\Gamma})$ will
be assembled: one for each edge, collecting the cohomology along that edge,
and one for each vertex, collecting the automorphisms and the node
smoothings there.
 
\begin{definition}\label{def:OmegaTheta}
For an edge $e$, read at one of its endpoints $v$ with the notation of
Lemma \ref{lem:edgeweights}, set
\[
  \Omega_{e}\;:=\;
  \frac{1}{(-1)^{d}(d!)^{2}\,\omega^{2d}}
  \prod_{\substack{\gamma\in R^{+}_{Y}\\ \gamma\neq\alpha}}
  \begin{cases}
    \displaystyle\prod_{j=1}^{-m_{\gamma}-1}
      \big(-w\gamma+j\,\omega\big),
      & m_{\gamma}\le-2,\\[1.2em]
    1, & m_{\gamma}=-1,\\[0.4em]
    \displaystyle\prod_{k=0}^{m_{\gamma}}
      \big(-w\gamma-k\,\omega\big)^{-1},
      & m_{\gamma}\ge0 .
  \end{cases}
\]
For a vertex $v$, writing $\psi_{(v,e)}$ for the $\psi$ class of the
contracted component $C_{v}$ at the node joining it to $C_{e}$, set
\[
  \Theta_{v}\;:=\;
  \begin{cases}
    \omega_{v,e}, & n_{v}=1,\\[0.3em]
    1, & n_{v}=2\ \text{with one edge and one mark},\\[0.3em]
    \big(\omega_{v,e_{1}}+\omega_{v,e_{2}}\big)^{-1},
      & n_{v}=2\ \text{with two edges},\\[0.5em]
    \displaystyle\prod_{e\ni v}\big(\omega_{v,e}-\psi_{(v,e)}\big)^{-1},
      & n_{v}\ge3 .
  \end{cases}
\]
\end{definition}
 
\begin{remark}\label{rmk:omega-intrinsic}
The three branches of $\Omega_{e}$ are the contributions of the graded
pieces of Lemma \ref{lem:edgeweights}: those of degree $\le-2$ contribute
their $\rH^{1}$, those of degree $\ge0$ their $\rH^{0}$, and those of degree
$-1$ nothing.  Since Euler classes are multiplicative in exact sequences,
\[
  \Omega_{e}
  =\frac{1}{e^{\mathbb T}\Big(\big(\rH^{0}(C_{e},f^{*}TY)
      -\rH^{1}(C_{e},f^{*}TY)\big)^{\mathrm{mov}}\Big)}
  =\frac{1}{e^{\mathbb T}\big(\rH^{0}(C_{e},f^{*}TY)^{\mathrm{mov}}\big)} ,
\]
the second equality because $\rH^{1}(C_{e},f^{*}TY)=0$ by convexity of
$Y$; the graded pieces of negative degree do contribute individually, but
their contributions cancel in $K$-theory against part of the $\rH^{0}$
terms.  In particular $\Omega_{e}$ is intrinsic to the edge and does not
depend on the endpoint at which it is read: reading from the other end
replaces $\omega$ by $-\omega$ and $\alpha$ by $\alpha_{e,v'}$, which
permutes the families of weights above.
\end{remark}
 
\begin{proposition}\label{prop:euler}
With the notation above,
\[
  \frac{1}{e^{\mathbb T}(N_{\Gamma})}
  \;=\;\prod_{v\in V}e_{w_{v}}^{\;\mathrm{val}(v)-1}
   \ \prod_{e\in E}\Omega_{e}
   \ \prod_{v\in V}\Theta_{v} .
\]
\end{proposition}
 
\begin{proof}
Write $\mathcal F=f^{*}TY$.  The deformation theory of stable maps gives
the exact sequence
\[
  0\to\mathrm{aut}(C,x_{\bullet})\to \rH^{0}(C,\mathcal F)\to T^{1}
   \to\mathrm{def}(C,x_{\bullet})\to \rH^{1}(C,\mathcal F)\to0,
\]
with $T^{1}$ the tangent space of the moduli stack and
$\rH^{1}(C,\mathcal F)=0$ by convexity.  Hence, in equivariant $K$-theory,
$T^{1}=\rH^{0}(C,\mathcal F)-\mathrm{aut}+\mathrm{def}$; and since
$N_{\Gamma}=(T^{1})^{\mathrm{mov}}$,
\[
  \frac{1}{e^{\mathbb T}(N_{\Gamma})}
  =\frac{e^{\mathbb T}\big(\mathrm{aut}^{\mathrm{mov}}\big)}
        {e^{\mathbb T}\big(\rH^{0}(C,\mathcal F)^{\mathrm{mov}}\big)\,
         e^{\mathbb T}\big(\mathrm{def}^{\mathrm{mov}}\big)} .
\]
 
\emph{The term $\rH^{0}(C,\mathcal F)$.}  The normalization sequence
\[
  0\to\mathcal F\to
   \bigoplus_{e}\mathcal F|_{C_{e}}\oplus
   \bigoplus_{v\,:\,n_{v}\ge3}\mathcal F|_{C_{v}}
   \to\bigoplus_{\text{nodes}}\mathcal F_{p}\to0
\]
gives, in equivariant $K$-theory,
\[
  \rH^{0}(C,\mathcal F)-\rH^{1}(C,\mathcal F)
  =\sum_{e}\Big(\rH^{0}(C_{e},\mathcal F)-\rH^{1}(C_{e},\mathcal F)\Big)
   +\sum_{v\,:\,n_{v}\ge3}T_{p_{w_{v}}}Y
   -\sum_{\text{nodes}}T_{p}Y ,
\]
a contracted component contributing its constant fiber $T_{p_{w_{v}}}Y$.  A
vertex with $n_{v}\ge3$ carries one contracted component and
$\mathrm{val}(v)$ nodes, so contributes $T_{p_{w_{v}}}Y$ with multiplicity
$1-\mathrm{val}(v)$.  A vertex with $n_{v}\le2$ carries no component and
$\mathrm{val}(v)-1$ nodes, hence the same multiplicity.  Since
$T_{p_{w}}Y$ has no fixed part, taking Euler classes gives
\[
  e^{\mathbb T}\big(\rH^{0}(C,\mathcal F)^{\mathrm{mov}}\big)
  =\prod_{v}e_{w_{v}}^{\,1-\mathrm{val}(v)}
   \prod_{e}\frac{e^{\mathbb T}\big(\rH^{0}(C_{e},\mathcal F)^{\mathrm{mov}}\big)}
                 {e^{\mathbb T}\big(\rH^{1}(C_{e},\mathcal F)\big)} ,
\]
and inverting, using Remark \ref{rmk:omega-intrinsic} to identify the edge
factor with $\Omega_{e}$, produces the first two products of the statement.
 
\emph{The term $\mathrm{aut}$.}  Infinitesimal automorphisms are supported
on the components carrying fewer than three special points. The contracted
components have $n_{v}\ge3$ by construction, so only the edge components
contribute.  An end of $C_{e}$ is a special point unless the corresponding
vertex has $n_{v}=1$.  If both ends are special, the only automorphism is the
scaling, of weight $0$, contributing nothing to the moving part; each end
with $n_{v}=1$ contributes one further automorphism, of weight
$\omega_{v,e}$.  This is the first case of $\Theta_{v}$, which appears in
the numerator.
 
\emph{The term $\mathrm{def}$.}  Deformations smooth the nodes, and a node
contributes the tensor product of the two tangent lines there.  At a vertex
with $n_{v}\ge3$ the node joining $C_{v}$ to $C_{e}$ contributes
$\omega_{v,e}-\psi_{(v,e)}$, as explained before Definition \ref{def:psi}.
At a vertex with $n_{v}=2$ meeting two edges the single node contributes
$\omega_{v,e_{1}}+\omega_{v,e_{2}}$, both tangent lines now being constant.
A vertex with one edge and one mark, or with $n_{v}=1$, carries no node.
Inverting gives the remaining cases of $\Theta_{v}$.
\end{proof}
 
\medskip
 
Assembling the pieces, we finally obtain the formula the section has been
working towards.
 
\begin{theorem}\label{thm:localization}
Let $Y=G/P$, let $\beta\in \rH_{2}(Y,\Z)$ be effective and nonzero, and let
$u_{1},\dots,u_{n}\in W_{Y}$ satisfy
$\sum_{k}\ell(u_{k})=\boldsymbol\delta(n,\beta)$.  Then
\[
  \big\langle\sigma_{u_{1}},\dots,\sigma_{u_{n}}\big\rangle^{Y}_{0,\beta}
  \;=\;\sum_{\Gamma}\frac{1}{|\Aut\Gamma|\ \prod_{e\in E}d_{e}}
  \ \prod_{v\in V}e_{w_{v}}^{\;\mathrm{val}(v)-1}
  \ \prod_{e\in E}\Omega_{e}
  \ \prod_{v\in V}I_{v}
  \ \prod_{k=1}^{n}\sigma^{u_{k}}\big|_{p_{w_{\mathrm{mk}(k)}}} ,
\]
the sum being over the finitely many isomorphism classes of decorated trees
of type $(n,\beta)$, where $I_{v}=\Theta_{v}$ for $n_{v}\le2$ and
\[
  I_{v}=\Big(\prod_{e\ni v}\omega_{v,e}^{-1}\Big)
        \Big(\sum_{e\ni v}\omega_{v,e}^{-1}\Big)^{n_{v}-3}
  \qquad\text{for }n_{v}\ge3 .
\]
\end{theorem}
 
\begin{proof}
The moduli stack $\overline{\mathcal M}_{0,n}(Y,\beta)$ is smooth and proper
because $Y$ is convex, so Theorem \ref{thm:ABBV} applies to
$\phi=\prod_{k}\ev_{k}^{*}\sigma^{u_{k}}$.  Its fixed loci are the
$F_{\Gamma}$ of Proposition \ref{prop:Fgamma}, whose presentation as a
quotient contributes the factor $\big(|\Aut\Gamma|\prod_{e}d_{e}\big)^{-1}$.
Since $\ev_{k}$ maps $F_{\Gamma}$ to the fixed point
$p_{w_{\mathrm{mk}(k)}}$, the integrand restricts to
$\prod_{k}\sigma^{u_{k}}|_{p_{w_{\mathrm{mk}(k)}}}$, a constant on
$F_{\Gamma}$.  Dividing by $e^{\mathbb T}(N_{\Gamma})$ as computed in
Proposition \ref{prop:euler} and integrating the $\psi$ classes over
$\prod_{v\,:\,n_{v}\ge3}\overline{\mathcal M}_{0,n_{v}}$ by
Lemma \ref{lem:M0n} replaces each $\Theta_{v}$ with $n_{v}\ge3$ by $I_{v}$,
and gives the stated sum.  It computes the equivariant integral
$\int_{\overline{\mathcal M}_{0,n}(Y,\beta)}\phi$, which under the hypothesis
on the degrees is a constant and therefore equals the nonequivariant
invariant.  Finiteness of the index set is Proposition \ref{prop:Fgamma}.
\end{proof}
 
\medskip

We conclude by obtaining a localization formula for complete intersections in flag varieties.

\medskip

\subsection{Gromov--Witten invariants for complete intersections in flag
varieties}\label{sec:GWci}
 
Let $Y=G/P$ be a flag
variety and let
\[
  \cE=G\times_{P}V
\]
be a \emph{homogeneous} vector bundle on $Y$, associated with a $P$-module
$V$ whose $\mathbb T$-weights we write as $-\mu_{1},\dots,-\mu_{s}$, so that
by Proposition \ref{prop:linebundleweights} the fiber of $\cE$ at $p_{w}$
has weights $-w\mu_{1},\dots,-w\mu_{s}$.  We assume $\cE$ globally
generated, and set $X=\mathscr Z(Y,\cE)$, the zero locus of a general
section, smooth of codimension $s=\rank\cE$ by Bertini.
 
Homogeneity is the new hypothesis here, and it is what makes the whole
apparatus of Section~\ref{sec:Loc_methods} applicable: $\cE$ is then
$\mathbb T$-equivariant, hence so is the bundle $\cE_{0,n,\beta}$ of
Subsection~\ref{Sub.Sec:GW}, and its Euler class admits an equivariant lift
$e^{\mathbb T}(\cE_{0,n,\beta})$ that we may restrict to the fixed loci.
Global generation, as before, gives both the convexity of $\cE$ and the
smoothness of $X$.  For $\cE=\bigoplus_{j}\cO_{Y}(\lambda_{j})$ as in
Theorem \ref{thm:ci} one has $\mu_{j}=\lambda_{j}$.
 
The strategy is now completely determined.  The functoriality identity of
Subsection~\ref{Sub.Sec:GW} expresses the invariants of $X$ as an integral
over $\overline{\mathcal M}_{0,n}(Y,\beta)$, which is smooth. Localization
on that stack has already been carried out in Theorem \ref{thm:localization}. Thus,
the only thing left to compute is the restriction of
$e^{\mathbb T}(\cE_{0,n,\beta})$ to a fixed locus $F_{\Gamma}$.  That
computation is the same normalization argument as in
Proposition \ref{prop:euler}, run on $\cE$ instead of on $TY$, and it
produces the same combinatorial bookkeeping: one factor per edge and one
factor per vertex, the latter with the exponent $1-\mathrm{val}(v)$.
 
\begin{proposition}\label{prop:twist}
Let $\Gamma$ be a decorated tree of type $(n,\beta)$.  Then
$e^{\mathbb T}\big(\cE_{0,n,\beta}\big)$ is constant on $F_{\Gamma}$, equal
to
\[
  \mathrm{Tw}(\Gamma)\;=\;
  \prod_{e\in E}\ \mathrm{Tw}_{e}
  \ \cdot\
  \prod_{v\in V}\Big(\prod_{i=1}^{s}\big(-w_{v}\mu_{i}\big)\Big)^{1-\mathrm{val}(v)} ,
\]
where, for an edge $e$ read at one of its endpoints $v$ with the notation of
Lemma \ref{lem:edgeweights} and with
$b_{i}:=d_{e}\,\langle\mu_{i},\alpha_{e,v}^{\vee}\rangle$,
\[
  \mathrm{Tw}_{e}\;=\;\prod_{i=1}^{s}
  \begin{cases}
    \displaystyle\prod_{k=0}^{b_{i}}\big(-w_{v}\mu_{i}-k\,\omega_{v,e}\big),
      & b_{i}\ge0,\\[1.2em]
    1, & b_{i}=-1,\\[0.4em]
    \displaystyle\prod_{j=1}^{-b_{i}-1}
      \big(-w_{v}\mu_{i}+j\,\omega_{v,e}\big)^{-1},
      & b_{i}\le-2 .
  \end{cases}
\]
\end{proposition}
 
\begin{proof}
On a point of $F_{\Gamma}$ the sheaf $f^{*}\cE$ is trivial of rank $s$ on
each contracted component, with fiber $\cE|_{p_{w_{v}}}$, while on each edge
component it carries the pullback of the filtration of
Proposition \ref{prop:linebundleweights}, whose graded pieces have degrees
$b_{i}$ and fiber weights $-w_{v}\mu_{i}$ at the point over $p_{w_{v}}$.
Lemma \ref{lem:edgeweights} then gives the weights of their $\rH^{0}$ and
$\rH^{1}$, and since Euler classes are multiplicative in exact sequences, the
three branches above compute
$e^{\mathbb T}\big(\rH^{0}(C_{e},f^{*}\cE)\big)$, the $b_{i}\le-2$ branch
appearing in the denominator because it is an $\rH^{1}$.
 
Normalizing $C$ gives, exactly as in the proof of
Proposition \ref{prop:euler} and using $\rH^{1}(C,f^{*}\cE)=0$,
\[
  \rH^{0}(C,f^{*}\cE)
  =\sum_{e}\Big(\rH^{0}(C_{e},f^{*}\cE)-\rH^{1}(C_{e},f^{*}\cE)\Big)
   +\sum_{v\,:\,n_{v}\ge3}\cE\big|_{p_{w_{v}}}
   -\sum_{\text{nodes}}\cE\big|_{p}
\]
in equivariant $K$-theory.  Counting contracted components against nodes as
before, each vertex contributes $\cE|_{p_{w_{v}}}$ with multiplicity
$1-\mathrm{val}(v)$, whether or not it carries a contracted component, and
taking Euler classes gives the formula.  Every weight occurring is
determined by $\Gamma$ alone, so the result is constant on $F_{\Gamma}$. In
particular, and in contrast with the normal bundle, no $\psi$ class enters
the twist.
\end{proof}
% \begin{remark}
% Two differences with $\Omega_{e}$ deserve a remark.  First, the branches are
% inverted: there we computed $e(\rH^{1})/e(\rH^{0})$, here
% $e(\rH^{0})/e(\rH^{1})$, since the twist enters the integrand and the normal
% bundle in its denominator.  Second, we do not discard zero weights.  For the
% normal bundle, a vanishing weight signals a direction along the fixed locus
% and is removed; here a factor $-w_{v}\mu_{i}-k\,\omega_{v,e}$ may genuinely
% vanish, namely when $\mu_{i}=\tfrac{k}{d_{e}}\alpha_{e,v}$, and then
% $\mathrm{Tw}(\Gamma)=0$: the graph simply does not contribute.  Finally,
% when all the $\mu_{i}$ are dominant (in particular for
% $\cE=\bigoplus_{j}\cO_{Y}(\lambda_{j})$ with the $\lambda_{j}$ globally
% generated, which is the case of Theorem \ref{thm:ci}) one has
% $b_{i}\ge0$ throughout and only the first branch occurs.
% \end{remark}
 
We can now read off the localization formula for $X$.
 
\begin{theorem}\label{thm:localizationCI}
Let $\cE$ be a globally generated homogeneous bundle on $Y$ with weights
$\mu_{1},\dots,\mu_{s}$, let $X=\mathscr Z(Y,\cE)$, and let
$\gamma_{k}=\iota^{*}\sigma_{u_{k}}$ with $u_{k}\in W_{Y}$.  Let
$\beta\in \rH_{2}(Y,\Z)$ be effective and assume
\[
  \sum_{k=1}^{n}\ell(u_{k})
  =\dim X+\int_{\beta}c_{1}(TX)+n-3 .
\]
Then,
\[
  \sum_{\beta'\,:\,\iota_{*}\beta'=\beta}
  \big\langle\gamma_{1},\dots,\gamma_{n}\big\rangle^{X}_{0,\beta'}
  \;=\;\sum_{\Gamma}\frac{1}{|\Aut\Gamma|\ \prod_{e\in E}d_{e}}
  \ \prod_{v\in V}e_{w_{v}}^{\;\mathrm{val}(v)-1}
  \ \prod_{e\in E}\Omega_{e}
  \ \prod_{v\in V}I_{v}
  \ \cdot\ \mathrm{Tw}(\Gamma)
  \ \cdot\prod_{k=1}^{n}\sigma^{u_{k}}\big|_{p_{w_{\mathrm{mk}(k)}}} ,
\]
the sum being over the isomorphism classes of decorated trees of type
$(n,\beta)$, with $\Omega_{e}$ and $I_{v}$ as in
Theorem \ref{thm:localization} and $\mathrm{Tw}(\Gamma)$ as in
Proposition \ref{prop:twist}.
\end{theorem}
 
\begin{proof}
Since $\cE$ is globally generated, it is convex, so the functoriality
identity of Subsection~\ref{Sub.Sec:GW} applies and reduces the left-hand
side to
\[
  \int_{[\overline{\mathcal M}_{0,n}(Y,\beta)]}
   e\big(\cE_{0,n,\beta}\big)\prod_{k}\ev_{k}^{*}\sigma_{u_{k}} ,
\]
an integral over a smooth proper Deligne--Mumford stack.  Apply
Theorem \ref{thm:ABBV} to the equivariant lift
$\phi=e^{\mathbb T}(\cE_{0,n,\beta})\prod_{k}\ev_{k}^{*}\sigma^{u_{k}}$.
Its fixed loci are the $F_{\Gamma}$ of Proposition \ref{prop:Fgamma},
contributing the factor $\big(|\Aut\Gamma|\prod_{e}d_{e}\big)^{-1}$. The
restriction of $\phi$ to $F_{\Gamma}$ is
$\mathrm{Tw}(\Gamma)\prod_{k}\sigma^{u_{k}}|_{p_{w_{\mathrm{mk}(k)}}}$ by
Proposition \ref{prop:twist} and Theorem \ref{thm:billey}, both factors
being constant there; and $e^{\mathbb T}(N_{\Gamma})^{-1}$ is given by
Proposition \ref{prop:euler}.  Integrating the $\psi$ classes with
Lemma \ref{lem:M0n} turns $\Theta_{v}$ into $I_{v}$, exactly as in the proof
of Theorem \ref{thm:localization}, and yields the displayed sum.
 
It remains to see that the hypothesis on the degrees makes the right-hand
side a constant, so that it computes the nonequivariant integral.  The class
$\sigma^{u_{k}}|_{p_{w}}$ is homogeneous of degree $\ell(u_{k})$, while
$e^{\mathbb T}(\cE_{0,n,\beta})\,e^{\mathbb T}(N_{\Gamma})^{-1}$ is
homogeneous of degree
$\rank\cE_{0,n,\beta}-\dim\overline{\mathcal M}_{0,n}(Y,\beta)$.  Using
$\rank\cE_{0,n,\beta}=\rank\cE+\int_{\beta}c_{1}(\cE)$ and
$\dim\overline{\mathcal M}_{0,n}(Y,\beta)=\dim Y+\int_{\beta}c_{1}(TY)+n-3$,
together with $\dim X=\dim Y-\rank\cE$ and
$c_{1}(TX)=c_{1}(TY)-c_{1}(\cE)$, this difference equals
$-\big(\dim X+\int_{\beta}c_{1}(TX)+n-3\big)$, and the total degree is
zero under the stated hypothesis.
\end{proof}
% \begin{remark}\label{rmk:beta0CI}
% For $\beta=0$ the only decorated trees have a single vertex and no edge, so
% $\mathrm{val}(v)=0$, $\mathrm{Tw}(\Gamma)=\prod_{i}(-w\mu_{i})$ and
% $n_{v}=n$.  The formula degenerates to
% \[
%   \int_{X}\gamma_{1}\gamma_{2}\gamma_{3}
%   =\sum_{w\in W_{Y}}\frac{1}{e_{w}}
%     \prod_{i=1}^{s}\big(-w\mu_{i}\big)
%     \prod_{k=1}^{3}\sigma^{u_{k}}\big|_{p_{w}} ,
% \]
% which is Corollary \ref{cor:ABBVflag} applied to
% $e^{\mathbb T}(\cE)\prod_{k}\sigma^{u_{k}}$, as it must be, since
% $[X]=e(\cE)\cap[Y]$.  For $n>3$ and $\beta=0$ the factor $I_{v}$ contains
% $\big(\sum_{e\ni v}\omega_{v,e}^{-1}\big)^{n-3}$ with an empty sum, hence
% vanishes, in agreement with $\int_{\overline{\mathcal M}_{0,n}}1=0$.
% \end{remark}

\bigskip

We move on to the description of quantum cohomology for flag varieties and complete intersections in flag varieties.
Later (Section \ref{sec:coarse}), we will explain how it can be used to define a coarser version of Hodge atoms, which will suffice for this paper.

\medskip

\section{Quantum cohomology}\label{sec:QH}

We now assemble the invariants computed in Section~\ref{sec:Loc_methods}
into the algebraic structure which most concerns us in this paper.  We
begin with the standard construction for a smooth complex projective
variety and then specialize first to flag varieties and afterwards to the
part of the quantum product of a complete intersection which is accessible
to our localization formula.

\medskip

\subsection{The Gromov--Witten potential and the quantum products}
\label{sec:QHgeneral}

Let $Z$ be a smooth complex projective variety.  In this subsection we
assume
\[
\rH^{\mathrm{odd}}(Z,\Q)=0,
\]
so that we may work with ordinary commutative variables rather than the
super-commutative formalism.  This hypothesis holds for flag varieties and
for all the fourfolds to which we later apply the full A-model
construction.  In the complete-intersection localization discussion below,
only even classes restricted from the flag variety will be inserted.

Write $\mathsf{NE}(Z)\subset\rH_2(Z,\Z)$ for the monoid of effective curve classes.  The \emph{Novikov ring} is the completion 
\[
\mathcal N_Z
:=
\left\{
\sum_{\beta\in\mathsf{NE}(Z)}c_\beta q^\beta
\ \middle|\
c_\beta\in\Q
\right\}
\]
of the monoid $\mathbb Q[\mathsf{NE}(Z)]$ along the maximal ideal $\{q^{\beta}:\beta\neq0\}$. Fix a homogeneous basis $T_0=1,T_1,\dots,T_s$ of $\rH^*(Z,\Q)$ and let $y_0,\dots,y_s$ be the corresponding dual
coordinates.  We write $\gamma=\sum_i y_iT_i$.

\begin{definition}\label{def:potential}
The \emph{genus-zero Gromov--Witten potential} of $Z$ is
\[
\Phi(\gamma)
:=
\sum_{\beta\in\mathsf{NE}(Z)}
\sideset{}{'}\sum_{m\ge0}
\frac{1}{m!}
\big\langle
\underbrace{\gamma,\dots,\gamma}_{m}
\big\rangle^Z_{0,\beta}\,
q^\beta
\in
\mathcal N_Z[[y_0,\dots,y_s]],
\]
where the prime means that the unstable degree-zero pairs $(\beta,m)=(0,m),
~m\le2$ are omitted.
\end{definition}

For primary genus-zero invariants, the only nonzero contribution with
$\beta=0$ occurs for $m=3$.  Consequently,
\[
\begin{aligned}
\Phi(y_0,\dots,y_s)
&=
\frac{1}{3!}
\sum_{a,b,c}
\left(
\int_ZT_a\smile T_b\smile T_c
\right)y_ay_by_c
\\[1mm]
&\quad+
\sum_{\beta\in\mathsf{NE}(Z)\smallsetminus\{0\}}
\sum_{m\ge0}
\frac{q^\beta}{m!}
\sum_{a_1,\dots,a_m}
\big\langle
T_{a_1},\dots,T_{a_m}
\big\rangle^Z_{0,\beta}
y_{a_1}\cdots y_{a_m}.
\end{aligned}
\]

\begin{definition}\label{def:poincare}
The \emph{Poincar\'e pairing} of $Z$ is
\[
g_{ij}:=\int_ZT_i\smile T_j.
\]
It is a nondegenerate symmetric bilinear form on
$\rH^*(Z,\Q)$.  We write $(g^{ij})$ for the inverse matrix.
\end{definition}

\begin{definition}\label{def:bigquantum}
Set
\[
Q\rH^*(Z)
:=
\mathcal N_Z[[y_0,\dots,y_s]]
\otimes_\Q\rH^*(Z,\Q).
\]
The \emph{big quantum product} is defined on the basis by
\[
T_i\star T_j
:=
\sum_{a,b}
\frac{\partial^3\Phi}
{\partial y_i\,\partial y_j\,\partial y_a}
\,g^{ab}\,T_b
\]
and extended
$\mathcal N_Z[[y_0,\dots,y_s]]$-bilinearly.
\end{definition}

\begin{theorem}[WDVV; see
\cites{Kontsevich1994,fulton1997notesstablemapsquantum}]
\label{thm:wdvv}
The product $\star$ is associative and commutative, with unit
$T_0=1$.
\end{theorem}

Associativity is equivalent to the WDVV equations for $\Phi$, which
express the two ways of degenerating a four-pointed rational curve; this
is the Splitting Axiom
\cites{Kontsevich1994,fulton1997notesstablemapsquantum}.
The ring $\big(Q\rH^*(Z),\star\big)$ is the \emph{big quantum cohomology} of $Z$.  

\begin{definition}\label{def:smallquantum}
Set
\[
Q\rH^*_{\mathrm{small}}(Z)
:=
\mathcal N_Z\otimes_\Q\rH^*(Z,\Q).
\]
The \emph{small quantum product} is
\[
T_i\star T_j
=
\sum_{a,b}
\left(
\sum_{\beta\in\mathsf{NE}(Z)}
\big\langle T_i,T_j,T_a\big\rangle^Z_{0,\beta}
q^\beta
\right)
g^{ab}T_b,
\]
extended $\mathcal N_Z$-bilinearly.
\end{definition}

We recall three properties which will be used repeatedly; see
\cites{Kontsevich1994,fulton1997notesstablemapsquantum}.

\begin{itemize}

\item
\emph{Grading.}
Declare
\[
\deg q^\beta
=
2\int_\beta c_1(TZ).
\]
Then $\star$ is homogeneous.  Indeed, a nonzero invariant $\big\langle T_i,T_j,T_a\big\rangle^Z_{0,\beta}$ must satisfy
\[
\deg T_i+\deg T_j+\deg T_a
=
2\left(
\dim Z+\int_\beta c_1(TZ)
\right),
\]
while $g^{ab}\neq0$ requires
\[
\deg T_a+\deg T_b=2\dim Z.
\]
Hence
\[
\deg q^\beta+\deg T_b
=
\deg T_i+\deg T_j.
\]

\item
\emph{Unit.}
The fundamental-class axiom gives
\[
\big\langle1,T_i,T_j\big\rangle^Z_{0,\beta}
=
\begin{cases}
0,&\beta\neq0,\\[1mm]
\displaystyle\int_ZT_i\smile T_j,&\beta=0.
\end{cases}
\]
Thus
\[
1\star T=T.
\]

\item
\emph{Divisor axiom.}
For $D\in\rH^2(Z,\Q)$ and $\beta\neq0$,
\[
\big\langle D,T_i,T_j\big\rangle^Z_{0,\beta}
=
\left(\int_\beta D\right)
\big\langle T_i,T_j\big\rangle^Z_{0,\beta}.
\]

\end{itemize}

\medskip

\subsection{The quantum cohomology of a flag variety}
\label{sec:QHflag}

Let $Z=F=G/P$ be a flag variety.  Proposition \ref{prop:schubertbasis} gives $\rH^{\mathrm{odd}}(F,\Q)=0$ and a homogeneous Schubert basis $\{\sigma_u\}_{u\in W_F}$. Moreover, by Corollary \ref{cor:effective}, every effective curve class has
a unique expression $\beta
=
\sum_{m=1}^{k}
b_m\overline{\alpha^\vee_{i_m}},
~
b_m\ge0$,
and hence $q^\beta=\prod_{m=1}^{k}q_m^{b_m}$. Thus, there is one Novikov variable for each element of $\Pi_F$. The Poincar\'e pairing takes a particularly simple form in the Schubert
basis.  For every $u\in W_F$ there is a unique $u^\vee\in W_F$ such that
\[
\ell(u^\vee)=\dim F-\ell(u)
\]
and
\[
\int_F\sigma_u\smile\sigma_v
=
\delta_{v,u^\vee}.
\]
Consequently, the Gram matrix is a permutation matrix and equals its
inverse.  In practice, we may equivalently evaluate the pairing by
Corollary \ref{cor:ABBVflag}.

The small quantum product is therefore
\[
\sigma_u\star\sigma_v
=
\sum_{\beta\in\mathsf{NE}(F)}
\sum_{w\in W_F}
\big\langle
\sigma_u,\sigma_v,\sigma_w
\big\rangle^F_{0,\beta}
q^\beta
\sigma_{w^\vee}.
\]
These are exactly the invariants computed by
Theorem \ref{thm:localization}.

\medskip

\subsection{The flag-ambient quantum data of a complete intersection}
\label{sec:QHci}

Let now
\[
F=G/P,
\qquad
\iota\colon X=\mathscr Z(F,\cE)\hookrightarrow F,
\]
where $\cE$ is a globally generated homogeneous vector bundle and $X$ is
smooth of the expected dimension.  This is the setting of
Theorem \ref{thm:localizationCI}.

The localization formula only sees insertions restricted from $F$.
Accordingly, set
\[
\rH^*_{\mathrm{amb}}(X)
:=
\operatorname{im}
\big(
\iota^*\colon\rH^*(F,\Q)\to\rH^*(X,\Q)
\big).
\]
In the notation used later in Section \ref{sec:coarse}, this is the
flag-ambient space
\[
\rH_F^*(X)=\rH^*_{\mathrm{amb}}(X).
\]
Since the cohomology of $F$ is generated by Schubert classes,
$\rH^*_{\mathrm{amb}}(X)$ is a graded subring of the classical cohomology
of $X$, consists entirely of classes of type $(p,p)$, and is concentrated
in even degree.

There are two separate issues which must be distinguished: the Poincar\'e
pairing on this subspace and the question whether quantum multiplication
preserves it.

\medskip

For $u,v\in W_F$, consider the Poincar\'e pairing
\[
g_{uv}
:=
\int_X
\iota^*\sigma_u\smile\iota^*\sigma_v.
\]
It is nondegenerate on
$\rH^*_{\mathrm{amb}}(X)$, then any maximal subfamily of the restricted Schubert classes having
nondegenerate Gram matrix is a basis of
$\rH^*_{\mathrm{amb}}(X)$. In all examples considered in this paper, the restricted pairing is nondegenerate,
and from now on we work in this situation. In particular, we choose a
homogeneous reduced Schubert basis $\sigma_1,\dots,\sigma_M$ of $\rH^*_{\mathrm{amb}}(X)$ as above, and write $(g^{ij})$ for the inverse Gram matrix.

\medskip

We next explain exactly what the localization formula computes.  If
$\beta\in\mathsf{NE}(F)$, define the pushed-forward three-point invariant
\[
\big\langle
\gamma_1,\gamma_2,\gamma_3
\big\rangle^{X/F}_{0,\beta}
:=
\sum_{\substack{
\beta'\in\mathsf{NE}(X)\\
\iota_*\beta'=\beta}}
\big\langle
\gamma_1,\gamma_2,\gamma_3
\big\rangle^X_{0,\beta'}.
\]
For flag-ambient insertions, this is precisely the quantity computed by
Theorem \ref{thm:localizationCI}.

Because the restricted Poincar\'e pairing is nondegenerate, we have an
orthogonal decomposition
\[
\rH^*(X,\Q)
=
\rH^*_{\mathrm{amb}}(X)
\oplus
\rH^*_{\mathrm{amb}}(X)^\perp
\]
and hence an orthogonal projection
\[
\operatorname{pr}_F
\colon
\rH^*(X,\Q)
\longrightarrow
\rH^*_{\mathrm{amb}}(X).
\]

For $\alpha,\eta\in\rH^*_{\mathrm{amb}}(X)$ we define their
\emph{projected flag-ambient small quantum product} by
\[
\alpha\star_F\eta
:=
\operatorname{pr}_F(\alpha\star\eta),
\]
after pushing the Novikov variables forward by $\iota_*$.  In the reduced
Schubert basis,
\[
\sigma_i\star_F\sigma_j
=
\sum_{\beta\in\mathsf{NE}(F)}
\sum_{a,b}
\big\langle
\sigma_i,\sigma_j,\sigma_a
\big\rangle^{X/F}_{0,\beta}
q^\beta
g^{ab}\sigma_b.
\]
Equivalently,
\[
\sigma_i\star_F\sigma_j
=
\sum_{\beta\in\mathsf{NE}(F)}
\sum_{a,b}
\left(
\sum_{\substack{
\beta'\in\mathsf{NE}(X)\\
\iota_*\beta'=\beta}}
\big\langle
\sigma_i,\sigma_j,\sigma_a
\big\rangle^X_{0,\beta'}
\right)
q^\beta g^{ab}\sigma_b.
\]

\medskip

\subsubsection{Quantum multiplication by the anticanonical class}

The operator used throughout the examples is the projected small quantum
multiplication by the anticanonical class.

By adjunction,
\[
c_1(TX)
=
\iota^*\widehat c_1(TX),
\qquad
\widehat c_1(TX)
:=
c_1(TF)-c_1(\cE)
\in\rH^2(F,\Q).
\]
If $\beta$ is a curve class of the flag variety, the expression which can be paired with $\beta$ is
$\widehat c_1(TX)$, while for
$\beta'\in\rH_2(X,\Z)$ one has
\[
\int_{\beta'}c_1(TX)
=
\int_{\iota_*\beta'}\widehat c_1(TX).
\]

\begin{definition}\label{def:ambient-anticanonical-operator}
The \emph{flag-ambient anticanonical operator} is
\[
A(q)
:=
\operatorname{pr}_F
\circ
\big(c_1(TX)\star_{\mathrm{small}}(-)\big)
\Big|_{\rH^*_{\mathrm{amb}}(X)}.
\]
In the reduced Schubert basis,
\[
A(q)\sigma_i
=
\sum_{\beta\in\mathsf{NE}(F)}
\sum_{j,k}
\left(
\sum_{\substack{
\beta'\in\mathsf{NE}(X)\\
\iota_*\beta'=\beta}}
\big\langle
c_1(TX),\sigma_i,\sigma_j
\big\rangle^X_{0,\beta'}
\right)
q^\beta g^{jk}\sigma_k.
\]
This is the matrix computed by the software.
\end{definition}

Equivalently, using the divisor axiom for $\beta'\neq0$,
\[
\big\langle
c_1(TX),\sigma_i,\sigma_j
\big\rangle^X_{0,\beta'}
=
\left(
\int_{\beta'}c_1(TX)
\right)
\big\langle
\sigma_i,\sigma_j
\big\rangle^X_{0,\beta'}
\]
and
\[
\int_{\beta'}c_1(TX)
=
\int_{\iota_*\beta'}\widehat c_1(TX).
\]
Thus, every coefficient required for $A(q)$ is among the invariants
computed by Theorem \ref{thm:localizationCI}. Proposition \ref{prop:finitebetas} below is immediate.

\begin{proposition}\label{prop:finitebetas}
Assume that $X$ is Fano.  Then only finitely many curve classes
$\beta'\in\mathsf{NE}(X)$ contribute to every entry of $A(q)$.
Consequently, the entries of $A(q)$ are polynomials in the pushed-forward
Novikov variables.
\end{proposition}

\bigskip

\section{Coarse Hodge Atoms}
\label{sec:coarse}

As initially conceptualized in \cite{KKPY}, the theory is based on the introduction of the concept of F-bundles, and the 
object of major interest is that of the \emph{A-model F-bundle}. In what follows, we shall not go in this direction. Instead, we use a coarser version
of Hodge atoms, which, in a nutshell, forgets the bundle and keeps only the base manifold information and a residual operator: quantum multiplication by the Euler vector field. This is sufficient for this paper and was idealized (and shall appear) in \cite{CKK-lectures}. The resulting geometric object is known as \emph{the A model pre-F manifold with an Euler vector field} (Definition \ref{def:preF}). In what follows, we essentially discuss this concept for the case of flag varieties and complete intersections in flag varieties, as we have been doing throughout. We then present the concept of coarse Hodge atoms and present general results.

\medskip

\subsection{Pre-F-manifolds and the A-model manifold}\label{sec:preF}

In summary, the idea behind what we shall call, following
\cite{CKK-lectures}, a pre-F-manifold with an Euler vector field is the
following.  Given a smooth projective variety $Z$ (which in our setup is
either a flag variety or a complete intersection in one), we would like to
construct a manifold whose ring of functions is the ring
$\mathcal N_{Z}[[y_{0},\dots,y_{s}]]$ in which the Gromov--Witten potential
lives.  More than that, we would like the potential to be an honest section
of its structure sheaf, so that the WDVV equations can be read as saying
that quantum multiplication is a family of associative products on the tangent bundle of
that manifold, one for each point.

We start with the heuristic.  The big quantum product varies with the point
at which the potential is differentiated, so a parameter space is needed,
and the natural candidates for its coordinates are those dual to a basis of
the cohomology together with the Novikov variables.  To make such a
construction viable over $\C$ we would have to know in advance that the
Gromov--Witten potential has a positive radius of convergence. This is
known, however, only in isolated cases.  Aiming not to depend on conjectural\footnote{that is not the only problem moving to the non-Archimedean setting avoids, but we will not enter this technicality here; the key word is \emph{Stokes phenomena}, which are absent in this setting}
statements, we shall instead move to the non-Archimedean world, which
justifies the incursion we make next.

\medskip

\subsubsection{Berkovich spaces: what we need}
 
This subsection may safely be skipped by readers already familiar with
Berkovich spaces; we shall not need the theory in extended depth as the hard work to well-ground the theory of Hodge atoms was already done in \cite{KKPY}.  The point to
keep in mind is the following.  Over $\C$ we have no control over the
convergence of the Gromov--Witten potential, and this is what initially obstructs the
construction sketched above.  Over a non-Archimedean field, the obstruction
disappears for a reason that is easy to state: the coefficients of the
potential are rational numbers. If one moves to the non-Archimedean setting, we can do so in a manner in which we set rational numbers with absolute value one. Thus, convergence
becomes a question about the monomials alone.  This will allow us to build, for
the varieties considered in this paper\footnote{although the construction is fully general \cites{CKK-lectures, KKPY}}, a space with the properties
described at the beginning of the section, such that the Gromov--Witten potential is an element of its structure sheaf.
 
\medskip
 
\begin{definition}\label{def:naf}
A \emph{non-Archimedean field} is a field $\Bbbk$ equipped with a map
$|\cdot|\colon\Bbbk\to\R_{\ge0}$, called a \emph{non-Archimedean absolute
value}, such that for all $a,b\in\Bbbk$:
\begin{defenum}
\item $|a|=0$ if and only if $a=0$, and $|1|=1$;
\item $|ab|=|a|\,|b|$;
\item $|a+b|\le\max\{|a|,|b|\}$,
\end{defenum}
and which is complete for the metric $d(a,b)=|a-b|$.  Condition (c) is the
\emph{ultrametric inequality}.  The subgroup $|\Bbbk^{*}|\subset\R_{>0}$ is
the \emph{value group}.
\end{definition}
 
\medskip
 
\begin{example}
Any field carries the \emph{trivial} absolute value, $|a|=1$ for $a\neq0$;
its value group is $\{1\}$.  The field $\Q_{p}$ of $p$-adic numbers carries
$|p^{m}u|=p^{-m}$ for $u$ a $p$-adic unit; its value group is $p^{\Z}$.  The
field of formal Laurent series $\Bbbk_{0}((t))$ over any field $\Bbbk_{0}$
carries $|f|=e^{-v(f)}$, where $v(f)$ is the order of vanishing at $t=0$;
here the absolute value is trivial on $\Bbbk_{0}$ and $|t|=e^{-1}$. 
\end{example}
 
\medskip

Let $\Bbbk$ be a non-Archimedean field.  We record three properties used
repeatedly below.
\begin{defenum}
\item If $|a|\neq|b|$ then $|a+b|=\max\{|a|,|b|\}$.
\item A series $\sum_{n\ge0}a_{n}$ converges if and only if $|a_{n}|\to0$.
\item The \emph{valuation ring} $\Bbbk^{\circ}=\{a\mid|a|\le1\}$ is a local
      ring with maximal ideal $\{|a|<1\}$.
\end{defenum}

\medskip
 
\begin{definition}\label{def:banach}
A \emph{non-Archimedean norm} on a ring $R$ is a map
$\|\cdot\|\colon R\to\R_{\ge0}$ such that $\|a\|=0$ if and only if
$a=0$, $\|1\|\le1$,
\[
  \|a+b\|\le\max\{\|a\|,\|b\|\},
  \qquad
  \|ab\|\le\|a\|\,\|b\|.
\]
A \emph{non-Archimedean Banach ring} is a ring with such a norm which is
complete for the induced metric.  A \emph{multiplicative seminorm} on $R$
is a map $|\cdot|_x\colon R\to\R_{\ge0}$ satisfying
$|0|_x=0$, $|1|_x=1$, $|ab|_x=|a|_x|b|_x$, and
\[
  |a+b|_x\le\max\{|a|_x,|b|_x\}.
\]
It is \emph{bounded} by $\|\cdot\|$ if $|a|_x\le\|a\|$ for all $a\in R$.
\end{definition}
 
Let $\Bbbk$ be a non-Archimedean field. Throughout, we use multi-index notation: for variables
$\xi=(\xi_{1},\dots,\xi_{n})$ and $I=(i_{1},\dots,i_{n})\in\N^{n}$, we write
$\xi^{I}=\xi_{1}^{i_{1}}\cdots\xi_{n}^{i_{n}}$ and $|I|=i_{1}+\dots+i_{n}$.
For $\rho=(\rho_{1},\dots,\rho_{n})\in(\R_{>0})^{n}$, we set
$\rho^{I}=\rho_{1}^{i_{1}}\cdots\rho_{n}^{i_{n}}$.  We endow the polynomial ring
$\Bbbk[\xi_{1},\dots,\xi_{n}]$ with the Gauss norm
$\big|\sum_{I}a_{I}\xi^{I}\big|:=\max_{I}|a_{I}|$. It extends the
absolute value of $\Bbbk$ and is multiplicative. However, it is not complete.
 
\begin{definition}\label{def:tate}
For $\rho\in(\R_{>0})^{n}$, the \emph{polydisk algebra} is the non-Archimedean Banach ring
\[
  \Bbbk\{\rho_{1}^{-1}\xi_{1},\dots,\rho_{n}^{-1}\xi_{n}\}
  :=\Big\{\ \sum_{I\in\N^{n}}a_{I}\,\xi^{I}\ \Big|\ a_{I}\in\Bbbk,\
    |a_{I}|\,\rho^{I}\longrightarrow0\ \text{as}\ |I|\to\infty\ \Big\}\]
    with norm
\[
  \Big\|\sum_{I}a_{I}\xi^{I}\Big\|:=\max_{I}|a_{I}|\rho^{I}.
\]
\end{definition}
 
\begin{definition}\label{def:analyticspectrum}
Let $R$ be a Banach ring.  Its \emph{Berkovich spectrum}
$\operatorname{Spec}_{\mathrm{an}}R$ is the set of multiplicative seminorms $\left\{|\cdot|_x:R\rightarrow \mathbb R_{\geq 0}\right\}$
on $R$ bounded by the norm of $R$, equipped with the weakest topology making
the evaluation maps
\[
 \operatorname{Spec}_{\mathrm{an}}R\longrightarrow\R_{\ge0},
 \qquad x\longmapsto|a|_{x},
\]
continuous for every $a\in R$.
\end{definition}
 
In the literature, the Berkovich spectrum of a Banach ring $R$ is usually
denoted $\mathcal M(R)$.  Our choice of writing it
$\operatorname{Spec}_{\mathrm{an}}R$ is meant to stress the parallel with
the notion of spectrum in algebraic geometry. The notation is further justified because sending a seminorm to its kernel, which is a prime ideal, defines a map
$\operatorname{Spec}_{\mathrm{an}}R\to\operatorname{Spec}R$ (\cite{jonsson_berkovich_2016}*{$\S$4.2}).  For $R\neq0$
the space $\operatorname{Spec}_{\mathrm{an}}R$ is nonempty compact and
Hausdorff (\cite{jonsson_berkovich_2016}*{Theorem 2.26}). 
 
\begin{example}\label{ex:berkovichdisc}
For $n=1$ and $\rho=1$ the space
$\operatorname{Spec}_{\mathrm{an}}\Bbbk\{\xi\}$ is the \emph{closed unit
disc}.  It contains the seminorms $\Bbbk\{\xi\}\ni P\mapsto|P(a)|$ for $a\in\Bbbk$ with
$|a|\le1$ but also many others. It is this
abundance that makes it compact and connected.  
\end{example}

General \emph{$\Bbbk$-analytic spaces}\footnote{it is far beyond the scope of this paper to give a precise definition; we ask the unfamiliar reader to consult \cite{jonsson_berkovich_2016}*{$\S$20.1} for additional details} are glued from the Berkovich spectra of simpler Banach rings\footnote{this is not exactly correct, but serves to keep the intuition}. Of importance to this text are
\begin{itemize}
    \item the Berkovich affine line
$\mathbb A^{1,\mathrm{an}}_{\Bbbk}$: it consists, as a set, of the multiplicative seminorms $\{|\cdot|_x\}$ in $\Bbbk[\xi]$ that extend the absolute value in $\Bbbk$ (i.e., $|a|_x=|a|$ for all $a\in\Bbbk$). Note that $\mathbb A^{1,\mathrm{an}}_{\Bbbk}$ itself \emph{is not} the Berkovich
spectrum of a Banach ring: imposing boundedness by the Gauss norm on
$\Bbbk[\xi]$ would force $|\xi|_{x}\le1$ and return the closed unit disc of
Example \ref{ex:berkovichdisc}.  Dropping boundedness and keeping only the
extension condition gives the increasing union
$\mathbb A^{1,\mathrm{an}}_{\Bbbk}=\bigcup_{r>0}
\operatorname{Spec}_{\mathrm{an}}\Bbbk\{r^{-1}\xi\}$ of the closed discs of
all radii, which is locally compact but not compact. 
\item the Berkovich multiplicative group $\mathbb G^{\mathrm{an}}_{m,\Bbbk}$: it consists, as a set, of the multiplicative seminorms in $\Bbbk[\xi,\xi^{-1}]$ that extend the absolute value in $\Bbbk$. 
\end{itemize}

\medskip
 
For this paper, we are interested in the completion $\Bbbk$ of the field of \emph{Puiseux series} over $\overline{\Q}$
\[
  \Bbbk_{\mathrm{Puis}}:=\bigcup_{n\ge1}\overline{\Q}\big(\big(t^{1/n}\big)\big),
  \qquad |f|:=e^{-v(f)}.
\]
Here, $v(f)$ is the smallest exponent occurring in $f$ with nonzero coefficient and
$v(0):=+\infty$. This field is algebraically closed of characteristic zero.  Its completion
\[
  \Bbbk\;:=\;\widehat{\Bbbk_{\mathrm{Puis}}},
\]
is again algebraically closed and $|\Bbbk^{*}|=e^{\Q}$. The absolute value
is trivial on the coefficients: $|a|=1$ for every $a\in\overline{\Q}^{\,*}$.
 
\medskip

\
\subsubsection{The A-model manifold in our setting}

Let $Z$ be a smooth projective variety.  The A-model construction itself does
not require the effective cone to be simplicial, and it is useful to keep this
generality here.  Put
\[
  N^{1}(Z)_{\Z}:=\mathsf{NS}(Z)/\mathrm{tors},
  \qquad
  r:=\rk N^{1}(Z)_{\Z}=\rk N_{1}(Z)_{\Z}.
\]
Choose an integral basis
$\beta_{1},\dots,\beta_{r}$ of $N_{1}(Z)_{\Z}$ and let
$T_{1},\dots,T_{r}$ be the dual basis of $N^{1}(Z)_{\Q}$.  We write
$q_{1},\dots,q_{r}$ for the corresponding Novikov coordinates and, for
$\beta\in N_{1}(Z)_{\Z}$, write $q^{\beta}$ for the corresponding Laurent
monomial.

When $Z=G/P$ is a flag variety, we take the Schubert divisor basis.  Then
Proposition \ref{prop:h2}, Proposition \ref{prop:ampleness}, and Corollary
\ref{cor:effective} identify the ample cone with the positive orthant and the
closed cone of effective curves with the nonnegative orthant.  For a complete
intersection in a flag variety we shall \emph{not} silently transfer these
cone statements from the ambient variety.  In the fourfold applications
below the integral divisor and numerical curve lattices are identified by
hypothesis \textup{(H0)} of Standing Assumption \ref{stand:setup}; the ample
and effective cones remain those of the variety under consideration unless a
separate argument identifies them.

Fix a homogeneous basis of the even cohomology
\[
  T_{0}=\mathbf 1,\quad T_{1},\dots,T_{r},\quad
  T_{r+1},\dots,T_{r+\ell},
\]
where the middle block is the chosen divisor basis.  Let
$p_{1},\dots,p_{r}$ be the coordinates dual to the divisors and
$y_{r+1},\dots,y_{r+\ell}$ the remaining even coordinates.  By the divisor
axiom, the Gromov--Witten potential depends on $p_m$ and $q_m$ only through
$q_me^{p_m}$, so we absorb the divisor coordinates into the Novikov
variables and keep the notation $q_m$.

Set once and for all
\[
  V^{\mathrm{ev}}_{\Q}
  :=\bigoplus_{\substack{a\ge4\\ a\ \mathrm{even}}}\rH^a(Z,\Q),
  \qquad
  V^{\mathrm{ev}}_{\Bbbk}
  :=V^{\mathrm{ev}}_{\Q}\otimes_{\Q}\Bbbk.
\]
Let $\mathcal U_Z\subset N^1(Z)_{\R}$ be the ample cone.  The Novikov
factor of the A-model is the tube domain
\[
  U_Z:=\mathrm{pr}^{-1}(-\mathcal U_Z)
  \subset
  \big(\mathbb G^{\mathrm{an}}_{m,\Bbbk}\big)^r,
\]
where
\[
  \mathrm{pr}(x)
  :=\big(\log|q_1|_x,\dots,\log|q_r|_x\big)
  \in N^1(Z)_{\R}
\]
under the chosen dual bases.  This is the formulation we shall use.  When
$\overline{\mathsf{NE}}(Z)$ is rational polyhedral, Kleiman's criterion
rewrites the same condition as finitely many strict inequalities on the
monomials corresponding to extremal rays.

The full A-model is a Berkovich analytic supermanifold when
$\rH^{\mathrm{odd}}(Z)\neq0$.  Since the fourfolds to which the spectral
criteria below are applied have no odd cohomology, we write explicitly only
the even part; the odd factor is the purely odd superspace of
\cites{CKK-lectures, KKPY} and is restored without changing the arguments.

\begin{definition}\label{def:amodel}
Assume $\rH^{\mathrm{odd}}(Z,\Q)=0$.  The \emph{A-model manifold} of $Z$ is
the $\Bbbk$-analytic space
\[
  B_Z:=U_Z\times U_{\mathrm{even}},
\]
where
\[
  U_{\mathrm{even}}
  =
  \mathbb A^{1,\mathrm{an}}_{\Bbbk}\times
  \Big\{
     x\in\big(V^{\mathrm{ev}}_{\Bbbk}\big)^{\mathrm{an}}
     \ \Big|\
     |y_{r+j}|_x<1\text{ for }1\le j\le\ell
  \Big\}.
\]
The affine-line coordinate $y_0$ is dual to the unit.
\end{definition}

When $\rH^{\mathrm{odd}}(Z)\neq0$, we use the full analytic supermanifold of
\cites{CKK-lectures, KKPY}, obtained by adjoining the purely odd factor whose
structure algebra is the exterior algebra on the coordinates dual to
$\rH^{\mathrm{odd}}(Z)$.  All statements about coarse atoms below are
understood in that full sense.

\begin{remark}\label{rmk:tube-is-polydisc}
For a flag variety the preceding definition becomes completely explicit.
Indeed, Proposition \ref{prop:ampleness} gives
\[
  U_Z
  =
  \big\{x\mid |q_m|_x<1,\ m=1,\dots,r\big\},
\]
so the Novikov factor is a punctured polydisc.  For a complete intersection,
we retain the intrinsic tube-domain definition unless its ample cone has
been identified separately.
\end{remark}

As in algebraic geometry, a $\Bbbk$-analytic space carries a structure
sheaf $\cO$ \cite{jonsson_berkovich_2016}*{\S19.1.4}, whose sections we call
analytic functions.  Before stating that the Gromov--Witten potential is one,
we recall at which points such a function can be evaluated.

\medskip

A point $x$ of a Berkovich spectrum is, by definition, a multiplicative
seminorm.  Thus an element $f\in R$ has a well-defined absolute value
$|f|_x$, but it need not have a value in $\Bbbk$.

\begin{definition}\label{def:kpoints}
Let $R$ be a non-Archimedean Banach ring containing $\Bbbk$ whose norm
restricts to $|\cdot|$ on $\Bbbk$.  A point
$x\in\operatorname{Spec}_{\mathrm{an}}R$ is a \emph{$\Bbbk$-point} if there
is a ring homomorphism
\[
  \chi_x\colon R\longrightarrow\Bbbk,
  \qquad
  \chi_x|_{\Bbbk}=\mathrm{id},
  \qquad
  |f|_x=|\chi_x(f)|
  \quad\text{for every }f\in R.
\]
The scalar $\chi_x(f)$ is the \emph{value} of $f$ at $x$, written $f(x)$.
\end{definition}

\begin{example}\label{ex:kpoints-polydisc}
On the polydisk algebra of Definition \ref{def:tate}, the $\Bbbk$-points are
the evaluations at tuples $a=(a_1,\dots,a_n)\in\Bbbk^n$ with
$|a_i|\le\rho_i$.  Indeed, if $f=\sum_I a_I\xi^I$, then
$|a_I|\,|a^I|\le|a_I|\rho^I\to0$, so the evaluation series converges.
Conversely, boundedness of a $\Bbbk$-point gives the same inequalities on
the coordinate values.
\end{example}

\begin{lemma}\label{lem:kpoints-of-B}
A tuple
\[
  b=
  \big(q_1,\dots,q_r,y_0,y_{r+1},\dots,y_{r+\ell}\big)
  \in(\Bbbk^*)^r\times\Bbbk\times\Bbbk^\ell
\]
defines a $\Bbbk$-point of $B_Z$ if and only if
\[
  \mathrm{pr}(b)\in-\mathcal U_Z,
  \qquad
  |y_{r+j}(b)|<1\quad(1\le j\le\ell).
\]
For a flag variety, this is equivalent to $|q_m|<1$ for every $m$.
\end{lemma}

\begin{proof}
The statement follows factor by factor from Example
\ref{ex:kpoints-polydisc}, from the description of
$\mathbb A^{1,\mathrm{an}}_{\Bbbk}$ as the union of all closed discs, and
from the definition of $U_Z$.  The final assertion is Remark
\ref{rmk:tube-is-polydisc}.
\end{proof}

\begin{theorem}[\cite{KKPY}*{Lemma 3.29}]\label{thm:GWanalytic}
The Gromov--Witten potential is an analytic function on $B_Z$:
\[
  \Phi\in\Gamma(B_Z,\cO_{B_Z}).
\]
In particular, it has a value $\Phi(b)\in\Bbbk$ at every $\Bbbk$-point.
\end{theorem}

\medskip

In the even case, the tangent sheaf is free:
\[
  TB_Z\cong\cO_{B_Z}\otimes_{\Bbbk}\rH^{\mathrm{even}}(Z,\Bbbk),
\]
with global frame
\[
  \big\{q_m\partial_{q_m}\big\}_{m=1}^r
  \cup
  \big\{\partial_{y_0},\partial_{y_{r+1}},\dots,
          \partial_{y_{r+\ell}}\big\}.
\]
Under this identification $\partial_{y_i}$ corresponds to $T_i$ and
$q_m\partial_{q_m}$ to the divisor class $T_m$.  Theorem
\ref{thm:GWanalytic} implies that the third derivatives of $\Phi$ are
analytic; hence, the big quantum product is a global section
\[
  \star\in
  \Gamma\Big(B_Z,\operatorname{Sym}^2(T^*B_Z)\otimes TB_Z\Big).
\]
At every $\Bbbk$-point it is commutative and unital, and associativity is the
WDVV equation.

\medskip

\subsubsection{The Euler field and pre-F-manifolds}

\begin{definition}\label{def:euler}
The \emph{Euler vector field} is
\[
  \mathsf{Eu}
  =
  \sum_{m=1}^r c_mq_m\partial_{q_m}
  +
  \sum_{i\in\{0,r+1,\dots,r+\ell\}}
  \Big(1-\tfrac12\deg T_i\Big)y_i\partial_{y_i},
\]
where $c_1(TZ)=\sum_{m=1}^r c_mT_m$.  Equivalently, at a $\Bbbk$-point $b$,
\[
  \mathsf{Eu}_b
  =
  c_1(TZ)+\frac{2\cdot\Id-\mathsf{Deg}}{2}(b),
  \qquad
  \mathsf{Deg}:=\bigoplus_a a\cdot\Id_{\rH^a(Z)}.
\]
\end{definition}

\begin{definition}[\cites{CKK-lectures}]\label{def:preF}
A \emph{pre-F-manifold with an Euler vector field} is a $\Bbbk$-analytic
space $M$ equipped with
\[
  \star\in\Gamma\big(M,\operatorname{Sym}^2(T^*M)\otimes TM\big),
  \qquad
  \mathsf{Eu}\in\Gamma(M,TM),
\]
such that $\star$ makes $TM$ into a sheaf of commutative, associative,
unital $\cO_M$-algebras and
\[
  \operatorname{Lie}_{\mathsf{Eu}}\star=\star.
\]
\end{definition}

\begin{proposition}\label{prop:BZisPreF}
$\big(B_Z,\star,\mathsf{Eu}\big)$ is a pre-F-manifold with unit
$\partial_{y_0}$.
\end{proposition}

\begin{proof}
Analyticity and associativity were established above, and
$\partial_{y_0}$ is the unit because
$\partial_{y_0}\partial_{y_i}\partial_{y_j}\Phi=g_{ij}$.
For the Euler identity, the dimension axiom gives the usual
quasi-homogeneity of the genus-zero potential with respect to the field of
Definition \ref{def:euler}. Differentiating this identity three times and
contracting with the inverse pairing Poincar\'e gives 
$\operatorname{Lie}_{\mathsf{Eu}}\star=\star$.
\end{proof}

\medskip

For a $\Bbbk$-point $b\in B_Z(\Bbbk)$, set
\[
  \boldsymbol\kappa_b:=\mathsf{Eu}_b\star_b
  \colon\rH^*(Z,\Bbbk)\longrightarrow\rH^*(Z,\Bbbk).
\]
Since $\Bbbk$ is algebraically closed, its generalized eigenspaces decompose
$\rH^*(Z,\Bbbk)$.  We call \emph{small point} a point $b_q$ for which all
non-Novikov coordinates vanish.  At a small point
$\mathsf{Eu}_{b_q}=c_1(TZ)$, hence $\boldsymbol\kappa_{b_q}$ is small quantum
multiplication by the anticanonical class.

\medskip

\subsection{Coarse Hodge atoms}\label{sec:atoms}

The decomposition of $\rH^{*}$ into generalized eigenspaces of $\boldsymbol\kappa_{b}$
is a decomposition of vector spaces.  To turn it into an invariant one has
to remember the Hodge theory.

\medskip

\subsubsection{The Hodge group}

From this point on, $\rH^*(Z)$ denotes the \emph{full} cohomology.  This is
important: the localization formula of Theorem \ref{thm:localizationCI}
computes only the flag-ambient block of the quantum operator, whereas the
Hodge group and Gu\'er\'e's invariants are defined on the full cohomology.  In
the fourfold applications below, Lemma \ref{lem:automatic} gives
$\rH^{\mathrm{odd}}(X)=0$, so the ordinary analytic space of Definition
\ref{def:amodel} is the full A-model manifold.

\medskip

Let $V:=\rH^*(Z,\Q)$, regarded as a graded polarizable $\Q$-Hodge structure:
the summand $\rH^a(Z,\Q)$ has weight $a$.  Let
$h\colon\mathbb S\to\mathrm{GL}(V_{\R})$ be the corresponding representation
of the Deligne torus and let
$\mathrm{Nm}\colon\mathbb S\to\mathbb G_{m,\R}$ be the norm.  The
\emph{Mumford--Tate group} $\mathsf{MT}(V)$ is the smallest $\Q$-algebraic
subgroup of $\mathrm{GL}(V)\times\mathbb G_m$ whose real points contain the
image of $(h,\mathrm{Nm})$.  If
$f\colon\mathsf{MT}(V)\to\mathbb G_m$ denotes the Tate character, the
\emph{Hodge group} is
\[
  \mathsf{Hod}:=\ker f.
\]
Its invariant vectors are precisely the rational Hodge classes:
\[
  V^{\mathsf{Hod}}
  =
  \bigoplus_{p\ge0}
  \big(\rH^{2p}(Z,\Q)\cap\rH^{p,p}(Z)\big)
  =:\rH^*(Z)^{\mathrm{Hdg}}.
\]

Write $\mathsf{Hod}_{\Bbbk}$ for the base change to $\Bbbk$.

\begin{lemma}\label{lem:hod-preserves}
The Hodge group preserves each graded piece $\rH^a(Z,\Q)$ and the Poincar\'e
pairing.
\end{lemma}
\begin{proof}
The Deligne-torus action preserves every cohomological degree, hence so does
its $\Q$-algebraic envelope.  For the pairing, the cup product
\[
  \rH^a(Z,\Q)\otimes\rH^{2\dim Z-a}(Z,\Q)
  \longrightarrow\Q(-\dim Z)
\]
is a morphism of Hodge structures.  The Mumford--Tate group acts on the Tate
object through a power of $f$, which is trivial on $\mathsf{Hod}=\ker f$.
\end{proof}

The decomposition
\[
  \rH^{\mathrm{even}}(Z,\Q)
  =
  \rH^0(Z,\Q)\oplus\rH^2(Z,\Q)\oplus V^{\mathrm{ev}}_{\Q}
\]
is therefore $\mathsf{Hod}$-stable.  The group acts trivially on the first
two summands: the unit is a Hodge class, and divisor classes are of type
$(1,1)$.  Put
\[
  V^{\mathrm{ev}}_{\mathrm{Hdg}}
  :=
  \big(V^{\mathrm{ev}}_{\Q}\cap\rH^*(Z)^{\mathrm{Hdg}}\big)
  \otimes_{\Q}\Bbbk,
\]
and denote by $\mathbb D(V^{\mathrm{ev}}_{\Bbbk})$ the open unit polydisc
appearing in Definition \ref{def:amodel}.

\begin{definition}\label{def:BZHod}
The \emph{Hodge locus} of $B_Z$ is
\[
  B_Z^{\mathsf{Hod}}
  :=
  U_Z\times\mathbb A^{1,\mathrm{an}}_{\Bbbk}\times
  \Big(
      \mathbb D(V^{\mathrm{ev}}_{\Bbbk})
      \cap
      (V^{\mathrm{ev}}_{\mathrm{Hdg}})^{\mathrm{an}}
  \Big).
\]
Thus the non-Novikov even parameter is constrained to the subspace of Hodge
classes.  In particular, every small point belongs to
$B_Z^{\mathsf{Hod}}$.  When odd cohomology is present, this denotes the
underlying even locus in the full analytic supermanifold, with the odd
factor retained in the tangent superspace.
\end{definition}

\begin{theorem}\label{thm:hodgeequivariance}
For every $b\in B_Z^{\mathsf{Hod}}$, the product $\star_b$ and the Euler
field $\mathsf{Eu}_b$ are $\mathsf{Hod}_{\Bbbk}$-equivariant.
\end{theorem}

\begin{proof}
The genus-zero Gromov--Witten tensors are obtained by pulling back
cohomology classes by evaluation maps and integrating their product against
the virtual fundamental class.  The latter is algebraic; consequently, the
correlator tensors are morphisms of Hodge structures, with the Tate twist
dictated by the virtual dimension.  After contraction with the inverse
Poincar\'e pairing, the structure constants of quantum multiplication are
therefore Hodge tensors.  Since $\mathsf{Hod}=\ker f$, it acts trivially on
the Tate factors.  Moreover, for
$b\in B_Z^{\mathsf{Hod}}$, the non-Novikov parameters are Hodge classes and
are fixed by $\mathsf{Hod}$, while the Novikov parameters are not acted on.
Thus
\[
  g(u\star_bv)=(gu)\star_b(gv).
\]
Finally, $c_1(TZ)$ is a Hodge class and Lemma \ref{lem:hod-preserves} shows
that $\mathsf{Hod}$ commutes with the grading operator, so
$g\mathsf{Eu}_b=\mathsf{Eu}_b$.
\end{proof}

\medskip

\subsubsection{Coarse Hodge atoms}
 
Fix $b\in B_{Z}^{\mathsf{Hod}}(\Bbbk)$ and let
$\boldsymbol\kappa_{b}=\mathsf{Eu}_{b}\star_{b}\in\End\big(T_{b}B_{Z}\big)$.
Since $\Bbbk$ is algebraically closed, the generalized eigenspaces of
$\boldsymbol\kappa_{b}$ decompose the tangent space:
\[
  T_{b}B_{Z}=\bigoplus_{\lambda\in\operatorname{spec}(\boldsymbol\kappa_{b})}E_{b,\lambda}.
\]
Because $b$ lies in the Hodge locus, Theorem \ref{thm:hodgeequivariance}
applies: the group $\mathsf{Hod}_{\Bbbk}$ acts linearly on
$T_{b}B_{Z}\cong \rH^{*}(Z,\Bbbk)$ and commutes with $\boldsymbol\kappa_{b}$.
Each generalized eigenspace is therefore $\mathsf{Hod}_{\Bbbk}$-stable, so
the displayed decomposition is one of $\mathsf{Hod}_{\Bbbk}$-representations.

Let $D(b)$ denote the number of distinct eigenvalues of $\boldsymbol\kappa_{b}$,
that is, the number of summands in the displayed decomposition, and set
\[
  \mathcal V:=\Big\{\,b\in B_{Z}^{\mathsf{Hod}}(\Bbbk)\ \Big|\
    D(b)=\max_{a\in B_{Z}^{\mathsf{Hod}}}D(a)\,\Big\}
  \ .
\]

\begin{lemma}\label{lem:V-open-dense}
The locus $\mathcal V$ is the set of $\Bbbk$-points of a nonempty analytic
open subset of $B_Z^{\mathsf{Hod}}$.
\end{lemma}

\begin{proof}
Choose a trivialization of $TB_Z$ and write
$p_b(\lambda)=\det(\lambda-\boldsymbol\kappa_b)$.  The coefficients of
$p_b$ are analytic functions of $b$.  The locus on which $p_b$ has fewer
than the maximal number of distinct roots is cut out by the appropriate
subresultants of $p_b$ and $\partial_\lambda p_b$, hence is a closed analytic
subset.  Its complement is nonempty by definition of the maximum.
\end{proof}

\begin{definition}\label{def:atom}
A \emph{coarse Hodge atom} of $Z$ is the isomorphism class, as a
representation of $\mathsf{Hod}_{\overline{\Q}}$, of a generalized
eigenspace $E_{b,\lambda}$ with $b\in\mathcal V(\Bbbk)$. This is well
defined because of Theorem \ref{thm:atomwelldefined} below.  The \emph{atomic
decomposition} of $Z$ is the resulting multiset of atoms, indexed by
$\operatorname{spec}(\boldsymbol\kappa_{b})$.
\end{definition}
 
\begin{theorem}[\cite{KKPY}*{Lemma 5.25}]\label{thm:atomwelldefined}
The isomorphism class of $E_{b,\lambda}$ does not depend on the choice of
$b\in\mathcal V(\Bbbk)$, and it is defined over $\overline{\Q}$. 
\end{theorem}
 
For an atom $\alpha$, we write $E^{\alpha}$ for an abstract
$\overline{\Q}$-linear representation of $\mathsf{Hod}_{\overline{\Q}}$ in
that class, and $E^{\alpha}_{\C}:=E^{\alpha}\otimes_{\overline{\Q}}\C$.
 
\begin{definition}\label{def:atominvariants}
Let $\alpha$ be a coarse Hodge atom.
\begin{defenum}
\item Its \emph{number of Hodge cycles} is
      $\rho_{\alpha}:=\dim_{\C}\big(E^{\alpha}_{\C}\big)^{\mathsf{Hod}_{\C}}$. \label{item:Hodge-class}
\item Its \emph{Hodge polynomial} is
      $P_{\alpha}(t):=\sum_{j}\dim_{\C}\big(E^{\alpha}_{\C}\big)_{j}\,t^{j}$,
      where $\big(E^{\alpha}_{\C}\big)_{j}$ is the subspace on which the
      cocharacter
      \[
        \mu\colon\mathbb G_{m,\C}\longrightarrow\mathsf{Hod}_{\C},
        \qquad \mu(u):=h_{\C}\big(u^{-1},u\big),
      \]
      acts by multiplication by $u^{j}$.  This grading is called the
      \emph{Hochschild grading}. \label{item:polynomial}
\end{defenum}
\end{definition}
 
\begin{remark}\label{rmk:mu}
\begin{itemize}
    \item Under the identification $\mathbb S_{\C}\cong\mathbb G_{m,\C}\times\mathbb G_{m,\C}$
the homomorphism $h_{\C}(u,v)$ acts on $\rH^{p,q}$ by $u^{-p}v^{-q}$ and the
norm is $\mathrm{Nm}_{\C}(u,v)=uv$.  Hence $\mu(u)=h_{\C}(u^{-1},u)$ has
norm $u^{-1}u=1$. Thus, its image lies in $\mathsf{Hod}_{\C}=\ker f$. Moreover, $\mu(u)$ acts on $\rH^{p,q}$ by
\[
  (u^{-1})^{-p}\,u^{-q}=u^{\,p-q}.
\]
\item   The polynomial
$P_{\alpha}$ does not depend on the choice of $\mu$ within its conjugacy
class.
\end{itemize}
\end{remark}
 
\begin{theorem}[\cite{KKPY}*{Proposition 5.28}]\label{thm:atomhodgeclass}
Every coarse Hodge atom $\alpha$ satisfies $\rho_{\alpha}\ge1$.
\end{theorem}
 
\medskip

\subsubsection{The rationality obstruction}
 
Atoms are not birational invariants, but their behaviour under blow-ups is
controlled: if $W\subset Z$ is a smooth center of codimension $c$, the
atomic decomposition of $\mathrm{Bl}_{W}Z$ consists of that of $Z$ together
with $c-1$ copies of that of $W$, mirroring the semiorthogonal decomposition
of the derived category of a blow-up.  Combined with weak factorization,
this yields an obstruction to rationality.
 
\begin{theorem}[Katzarkov--Kontsevich--Pantev--Yu \cite{KKPY}]
\label{non_rationality_criterion}
Let $Z$ be a smooth projective variety of dimension $n\ge2$.  Suppose that
some atom appearing in the atomic decomposition of $Z$ appears in the
atomic decomposition of no smooth projective variety of dimension at most
$n-2$.  Then $Z$ is not rational.
\end{theorem}

To use Theorem \ref{non_rationality_criterion} one needs the
low-dimensional catalogue of \cite{KKPY}.  We reproduce below the part of
that catalogue used later in the paper.  An entry such as
$3\,\rH^{\bullet}(\mathrm{pt})$ means that the decomposition consists of
three copies of the trivial one-dimensional atom.  The last column records
the support of the Hochschild grading (recall Definition
\ref{def:atominvariants}).
 
\begin{table}[ht]
\centering
\small
\begin{tabularx}{\textwidth}{c X c c}
\toprule
\textbf{Dim.} & \textbf{Variety} & \textbf{Coarse Hodge atoms}
& \textbf{Hochschild weights} \\
\midrule
$0$ & point $\mathrm{pt}$ & $\rH^{\bullet}(\mathrm{pt})$ & $p-q=0$ \\
\addlinespace
$1$ & $\PP^{1}$ & $2\,\rH^{\bullet}(\mathrm{pt})$ & $p-q=0$ \\
$1$ & curve $C$ with $\mathsf g(C)\ge1$ & $\rH^{\bullet}(C)$ & $|p-q|\le1$ \\
\addlinespace
$2$ & $\PP^{2}$ & $3\,\rH^{\bullet}(\mathrm{pt})$ & $p-q=0$ \\
$2$ & $\PP^{1}\times C$ with $\mathsf g(C)\ge1$
    & $2\,\rH^{\bullet}(C)$ & $|p-q|\le1$ \\
$2$ & abelian surface or K3 surface $S$ & $\rH^{\bullet}(S)$ & $|p-q|\le2$ \\
$2$ & Enriques, bielliptic, or elliptic with $p_{\mathsf g}=0$
    & $\rH^{\bullet}(S)$ & $|p-q|\le1$ \\
$2$ & general type, with ADE singularities resolved by $\tilde S$
    & $\rH^{\bullet}(\tilde S)$ & $|p-q|\le2$ \\
\bottomrule
\end{tabularx}
\caption{The part of the low-dimensional catalogue of coarse Hodge atoms
used below, after \cite{KKPY}*{Table~1}.}
\label{table:atoms-classification}
\end{table}

 \medskip

\subsubsection{Ambient, fixed, vanishing, and primitive cohomology}

From here on, we consider smooth projective varieties
\[
  X\ \xhookrightarrow{\ j\ }\ Y\ \xhookrightarrow{\ \iota_Y\ }\ F,
\]
where $F$ is a flag variety and $X$ is a smooth hyperplane section of $Y$.
We write $\iota:=\iota_Y\circ j$ and work under the following standing assumptions:

\begin{standing}\label{stand:setup}
Let $F$ be a flag variety with $\mathrm{Pic}(F)\cong\Z^d$, let $\cE$ be a
homogeneous vector bundle on $F$, and put
\[
  L:=\cO_F(\underbrace{1,\dots,1}_{\text{$d$ times}}),
  \qquad
  Y:=\mathscr Z(F,\cE),
  \qquad
  X:=\mathscr Z(F,\cE\oplus L).
\]
Assume that $Y$ and $X$ are smooth of dimensions $n+1$ and $n$,
respectively, with $n\ge3$, and that:
\begin{defenum}
\item[\textup{(H0)}]
restriction identifies the integral divisor lattices
\[
  \mathrm{Pic}(F)\xrightarrow{\ \sim\ }\mathrm{Pic}(Y)
  \xrightarrow{\ \sim\ }\mathrm{Pic}(X),
\]
and, dually, the numerical curve lattices used for the Novikov variables;
\item[\textup{(H1)}]
for every $k<n$, restriction induces an isomorphism
\[
  \iota_Y^*\colon\rH^k(F,\Q)\xrightarrow{\ \sim\ }\rH^k(Y,\Q);
\]
\item[\textup{(H2)}]
if $n=4$, then $\rH^4(Y,\Q)$ is of Hodge type $(2,2)$;
\item[\textup{(H3)}]
$Y$ is Fano.
\end{defenum}
We write $h:=c_1(L|_X)$.  In Subsection \ref{sec:atomsapp} we specialize to
$n=4$ and assume in addition that $X$ is Fano and $h^{3,1}(X)=1$.
\end{standing}

\begin{lemma}\label{lem:automatic}
Under Standing Assumption \ref{stand:setup}:
\begin{defenum}
\item for every $k<n$,
\[
  \iota^*\colon\rH^k(F,\Q)\xrightarrow{\ \sim\ }\rH^k(X,\Q),
\]
and these groups are algebraic and of type $(p,p)$;
\item $h^{2,0}(X)=0$, and if $n\ge4$ then
\[
  b_1(X)=b_3(X)=0;
\]
\item if $n=4$, then
\[
  h^{3,1}(Y)=0.
\]
\end{defenum}
\end{lemma}

\begin{proof}
Hypothesis \textup{(H1)} and weak Lefschetz for the hyperplane section
$X\subset Y$ give (a).  Flag-variety cohomology is generated by Schubert
classes.  Taking degrees one, two, and three gives (b).  Part (c) is exactly
\textup{(H2)}.
\end{proof}

We now distinguish the two cohomological pieces which were implicitly
identified in the earlier version of the argument.

\begin{definition}\label{def:threeparts}
\begin{defenum}
\item The \emph{flag-ambient cohomology} is
\[
  \rH_F^*(X):=\iota^*\rH^*(F,\Q).
\]
When no confusion is possible, we also write
$\rH^*_{\mathrm{amb}}(X)=\rH_F^*(X)$.
\item The \emph{$Y$-part} is
\[
  \rH_Y^*(X):=j^*\rH^*(Y,\Q).
\]
\item The \emph{vanishing cohomology} is
\[
  \rH^*(X)_{\mathrm{van}}
  :=
  \ker\Big(j_*\colon\rH^*(X,\Q)\to\rH^{*+2}(Y,\Q)\Big).
\]
\item For $k\le n$, the \emph{primitive cohomology} relative to $h$ is
\[
  \rH^k(X)_{\mathrm{prim}}
  :=
  \ker\Big(
     h^{n-k+1}\cup(-)\colon\rH^k(X,\Q)\to\rH^{2n-k+2}(X,\Q)
  \Big).
\]
\end{defenum}
\end{definition}

\begin{lemma}\label{lem:vananb}
Under Standing Assumption \ref{stand:setup}:
\begin{defenum}
\item $\rH^*(X)_{\mathrm{van}}$ is nonzero only in degree $n$, where
\[
  \rH^n(X,\Q)
  =
  \rH^n(X)_{\mathrm{van}}
  \oplus
  \rH_Y^n(X)
\]
is an orthogonal decomposition into sub-Hodge structures;
\item$\rH^n(X)_{\mathrm{van}}\subseteq\rH^n(X)_{\mathrm{prim}};$
\item for every $k\neq n$,
\[
  \rH^k(X,\Q)=\rH_Y^k(X)=\rH_F^k(X),
\]
whereas in degree $n$ there is only the inclusion
\[
  \rH_F^n(X)\subseteq\rH_Y^n(X).
\]
Consequently
\[
  \rH^*(X,\Q)
  =
  \rH_Y^*(X)\oplus\rH^*(X)_{\mathrm{van}},
\]
where the second summand is understood to occur only in degree $n$.
\end{defenum}
\end{lemma}

\begin{proof}
The middle-dimensional orthogonal decomposition is the standard vanishing
cohomology decomposition for a smooth hyperplane section; see
\cite{Voisin2003}*{\S2.3.3}.  If $k<n$ and $j_*\alpha=0$, then
$j^*j_*\alpha=h\cup\alpha=0$, and Hard Lefschetz makes
$h\cup(-)$ injective below the middle degree.  If $k>n$, then
$2n-k<n$, weak Lefschetz gives
\[
  j^*\colon\rH^{2n-k}(Y)\xrightarrow{\sim}\rH^{2n-k}(X),
\]
and adjunction of $j^*$ and $j_*$ for the Poincar\'e pairings makes $j_*$
injective on $\rH^k(X)$.  This proves (a).  Part (b) follows from
$j^*j_*\alpha=h\cup\alpha$ in degree $n$.

For (c), degrees below $n$ follow from Lemma
\ref{lem:automatic}\textup{(a)}.  Degrees above $n$ follow by Hard
Lefschetz from the corresponding degree below $n$, since $h$ is restricted
from $F$.  In the middle degree, functoriality gives only
\[
  \iota^*\rH^n(F)
  =j^*\iota_Y^*\rH^n(F)
  \subseteq j^*\rH^n(Y),
\]
and this inclusion may be strict.
\end{proof}

\begin{definition}\label{def:ambient-complete}
We say that the pair $X\subset Y\subset F$ is \emph{ambient-complete in the
middle degree} if
\[
  \iota_Y^*\colon\rH^n(F,\Q)\xrightarrow{\ \sim\ }\rH^n(Y,\Q).
  \tag{AC}
\]
Under \textup{(AC)}, Lemma \ref{lem:vananb} gives
\[
  \rH_Y^*(X)=\rH_F^*(X)=\rH^*_{\mathrm{amb}}(X).
\]
\end{definition}

\begin{definition}[\cite{Benedetti_Fay_Guere_Manivel_Perrin}*{Def.~3.2}]
\label{def:hodgegeneral}
The variety $X$ is \emph{Hodge general} if
\[
  \rH^n(X)^{\mathrm{Hdg}}
  =
  j^*\big(\rH^n(Y)^{\mathrm{Hdg}}\big).
\]
\end{definition}

\begin{lemma}\label{lem:hodgegeneral-equiv}
The variety $X$ is Hodge general if and only if
$\rH^n(X)_{\mathrm{van}}$ contains no nonzero rational Hodge class.
\end{lemma}

\begin{proof}
Lemma \ref{lem:vananb}\textup{(a)} is an orthogonal decomposition of
rational Hodge structures.  Taking rational Hodge classes gives the claimed
equivalence.
\end{proof}

\begin{proposition}[\cite{Benedetti_Fay_Guere_Manivel_Perrin}*{Prop.~3.1}]
\label{prop:hodgegeneral}
If $h^{3,1}(Y)<h^{3,1}(X)$ and $X$ is very general in its family, then $X$
is Hodge general.
\end{proposition}

\begin{lemma}\label{lem:hdg-is-amb}
Assume $n=4$.  If $X$ is Hodge general, then
\[
  \rH^*(X)^{\mathrm{Hdg}}=\rH_Y^*(X).
\]
If, in addition, \textup{(AC)} holds, then
\[
  \rH^*(X)^{\mathrm{Hdg}}=\rH^*_{\mathrm{amb}}(X).
\]
\end{lemma}

\begin{proof}
Outside degree four, the full cohomology is flag-ambient and algebraic by
Lemma \ref{lem:vananb}\textup{(c)}.  In degree four, Hodge generality gives
\[
  \rH^4(X)^{\mathrm{Hdg}}=j^*\rH^4(Y),
\]
and \textup{(H2)} says that all of $\rH^4(Y)$ is of type $(2,2)$.  The last
statement follows from \textup{(AC)}.
\end{proof}

\medskip

\subsubsection{The action on the vanishing cohomology}

Let $\pi$ be the monodromy group of a Lefschetz pencil of hyperplane
sections of $Y$ containing $X$.

\begin{lemma}\label{lem:monodromy}
The group $\pi$ acts trivially on $\rH_Y^*(X)$ and irreducibly and
nontrivially on
\[
  W:=\rH^n(X)_{\mathrm{van}}
\]
when $W\neq0$.  These statements remain true after extension of scalars to
any field of characteristic zero.
\end{lemma}

\begin{proof}
Every class in $\rH_Y^*(X)$ is restricted from the fixed variety $Y$, hence
is invariant under the pencil monodromy.  The vanishing cycles span $W$ and
form a single monodromy orbit up to sign.  The Picard--Lefschetz formula
\[
  T_\delta(\alpha)=\alpha\pm\langle\alpha,\delta\rangle\delta
\]
together with nondegeneracy of the intersection form on $W$ gives
irreducibility exactly as usual: a nonzero invariant subspace contains one
vanishing cycle and hence all of them.  The same formula excludes the
trivial representation when $W\neq0$.
\end{proof}

\begin{theorem}\label{theo:scalar_on_van}
Assume Standing Assumption \ref{stand:setup} and let $b$ be a small
$\Bbbk$-point of $B_X$.  Then $\boldsymbol\kappa_b$ preserves
\[
  \rH^*(X,\Bbbk)
  =
  \rH_Y^*(X)_{\Bbbk}\oplus W_{\Bbbk}
\]
and acts on $W_{\Bbbk}$ by a scalar.  Writing $B_b$ for the block on the
$Y$-part,
\[
  \boldsymbol\kappa_b
  =
  \begin{pmatrix}
    B_b&0\\[2pt]
    0&\lambda_{\mathrm{van}}(b)\,\mathrm{id}_W
  \end{pmatrix}.
\]
If \textup{(AC)} holds, then $B_b$ coincides with the flag-ambient block $A_b$
computed via Theorem \ref{thm:localizationCI}.
\end{theorem}

\begin{proof}
All non-Novikov coordinates of a small point vanish.  Moreover, weak
Lefschetz gives
$\rH^2(X,\Q)=j^*\rH^2(Y,\Q)$, so monodromy fixes the divisor and numerical
curve lattices and hence the Novikov coordinates.  Thus $g\cdot b=b$ for
all $g\in\pi$.

Deformation invariance of Gromov--Witten invariants gives monodromy
equivariance of the quantum product; compare the proof of
\cite{Benedetti_Fay_Guere_Manivel_Perrin}*{Thm.~4.1}.  At a small point
$\mathsf{Eu}_b=c_1(TX)$, which is monodromy invariant.  Hence
$\boldsymbol\kappa_b$ commutes with $\pi$.  Lemma \ref{lem:monodromy} says
that the first summand is trivial and the second irreducible and nontrivial,
so the off-diagonal blocks vanish.  Since $\Bbbk$ is algebraically closed,
Schur's lemma makes the restriction to $W_{\Bbbk}$ scalar.  Under
\textup{(AC)}, Definition \ref{def:ambient-complete} identifies the first
summand with the flag-ambient cohomology.
\end{proof}

\begin{lemma}\label{lem:fano-index-vanishing}
Assume $n=4$ and $X$ has Fano index at least two.  Then, at every small point
\[
  \lambda_{\mathrm{van}}(b)=0.
\]
\end{lemma}

\begin{proof}
The vanishing cohomology is concentrated in degree four.  A term of
$c_1(TX)\star(-)$ which maps degree four back to degree four must have curve
class $\beta$ satisfying
\[
  \langle c_1(TX),\beta\rangle=1.
\]
If $c_1(TX)=rH$ with $r\ge2$ and $H$ integral, no integral curve class can
satisfy this equality.  The classical term raises degree and hence has no
component from $W$ to itself.  Therefore the scalar block on $W$ vanishes.
\end{proof}

\medskip

\subsubsection{Gu\'er\'e's invariants and the evaluation-map dictionary}
\label{sec:atomsapp}

From now on
\[
  n=4,\qquad X\text{ is Fano},\qquad h^{3,1}(X)=1.
\]
Hodge generality will be imposed only when the invariant $\rho$ is used.

For a number field $K$, let $F_K$ denote Gu\'er\'e's Levi--Civita coefficient
field and let $\Bbbk_K\subset\Bbbk$ be the completion of the $K$-Puiseux
subfield.  We use the canonical valued-field embedding
\[
  \jmath_K\colon\Bbbk_K\hookrightarrow F_K,
  \qquad
  t^q\longmapsto a^q.
\]
All comparisons below are made after this scalar extension; generalized
eigenspace dimensions and Jordan defects are unchanged by it.

Gu\'er\'e formulates his invariants in terms of normalized graded evaluation
maps
\[
  \mathrm{ev}\colon\widehat R^*_{X,K}\longrightarrow S_K^*,
  \qquad
  S_K^*=F_K[u^{\pm1}],
\]
with $\deg u=1$; see \cite{JeremyCubic}*{Defs.~20, 25, 27}.

\begin{lemma}\label{lem:guere-dictionary}
For every $N\ge1$, let $b_N$ be the small point
\[
  q^\beta(b_N)=t^{N\langle c_1(TX),\beta\rangle},
  \qquad
  y_i(b_N)=0.
\]
Then,
\begin{defenum}
\item $b_N\in B_X^{\mathsf{Hod}}(\Bbbk)$;
\item there is a normalized $\Q$-evaluation map $\mathrm{ev}_N$, nonzero on
$Q'$, given on the Novikov and Hodge variables by
\[
  \mathrm{ev}_N(Q^\beta)
  =
  a^{N\langle c_1(TX),\beta\rangle}
  u^{2\langle c_1(TX),\beta\rangle},
  \qquad
  \mathrm{ev}_N(T_k)=0,
\]
and by $\mathrm{ev}_N(q')=u^{\deg q'}$ on Gu\'er\'e's auxiliary graded
variables;
\item after extension by $\jmath_\Q$, the evaluated quantum operator is
conjugate to $u^2\boldsymbol\kappa_{b_N}$.  More precisely, if
$\phi_i\in\rH^{d_i}(X)$ is homogeneous and
$D_u=\mathrm{diag}(u^{d_i})$, then
\[
  D_u[\mathrm{ev}_N(\kappa_\tau)]D_u^{-1}
  =
  u^2[\boldsymbol\kappa_{b_N}].
\]
Consequently, the generalized eigenspaces and Gu\'er\'e's integers
$\rho,\nu,\nu',\gamma$ (Definition \ref{def:atomintegers}) agree under the correspondence
$\alpha=u^2\lambda$.
\end{defenum}
\end{lemma}

\begin{proof}
Since $X$ is Fano, $c_1(TX)$ lies in the ample cone.  Because
$\log|t|<0$, the point $b_N$ satisfies
\[
  \mathrm{pr}(b_N)=N\log|t|\,c_1(TX)\in-\mathcal U_X,
\]
so it lies in $B_X$; all non-Novikov coordinates vanish, hence it lies in
the Hodge locus.

For (b), the displayed assignment is multiplicative in $\beta$, homogeneous
for Gu\'er\'e's grading, and satisfies the normalized valuation bounds.
\cite{JeremyCubic}*{Prop.~21} therefore extends it uniquely to a normalized
evaluation map, and the prescribed value on $Q'$ is nonzero.

For (c), if a term $q^\beta$ occurs in the $(i,j)$ entry, homogeneity gives
\[
  d_i=d_j+2-2\langle c_1(TX),\beta\rangle.
\]
Evaluation contributes
$u^{d_j+2-d_i}$, and conjugation by $D_u$ leaves the common factor $u^2$.
The remaining coefficient is exactly the image under $\jmath_\Q$ of the
coefficient obtained by evaluating at $b_N$.  Scalar extension,
conjugation, and multiplication by a nonzero scalar preserve generalized
eigenspace dimensions and Jordan form, while $D_u$ records precisely the
cohomological grading.  The four invariants therefore agree.
\end{proof}

\begin{definition}[\cite{JeremyCubic}*{Def.~25}]\label{def:atomintegers}
For a $K$-evaluation map $\mathrm{ev}$ and a generalized eigenspace
$E^X_{\mathrm{ev},\alpha}$ of $\mathrm{ev}(\kappa_\tau)$, we write 
\begin{align}
\rho_{\mathrm{ev},\alpha} & := \operatorname{rk}_{S_{K}^{*}} \left( E_{\mathrm{ev},\alpha}^{X} \cap \left( H^{*}(X)^{\mathrm{Hdg}} \otimes_{\mathbb{Q}} S_{K}^{*} \right) \right), \\
\nu_{\mathrm{ev},\alpha} & := \operatorname{rk}_{S_{\mathbb{C}}^{*}} \left( \left( E_{\mathrm{ev},\alpha}^{X} \otimes_{K} \mathbb{C} \right) \cap \left( H^{(2)}(X) \otimes_{\mathbb{C}} S_{\mathbb{C}}^{*} \right) \right), \\
{\nu'}_{\mathrm{ev},\alpha} & := \operatorname{rk}_{S_{\mathbb{C}}^{*}} \left( \left( E_{\mathrm{ev},\alpha}^{X} \otimes_{K} \mathbb{C} \right) \cap \left( H^{(1)}(X) \otimes_{\mathbb{C}} S_{\mathbb{C}}^{*} \right) \right), \\
\gamma_{\mathrm{ev},\alpha} & := \operatorname{rk}_{S_{K}^{*}} \left( \operatorname{ev}(\kappa_{\tau}) - \alpha \right)|_{ E_{\mathrm{ev},\alpha}^{X}},
\end{align}
% \[
%   \rho_{\mathrm{ev},\alpha},\qquad
%   \nu_{\mathrm{ev},\alpha},\qquad
%   \nu'_{\mathrm{ev},\alpha},\qquad
%   \gamma_{\mathrm{ev},\alpha}
% \]
for the four invariants of Gu\'er\'e's Definition~25, with exactly the scalar
extensions and grading conventions prescribed there. 
% In particular,
% $\gamma$ is the Jordan defect
% \[
%   \rk\big(\mathrm{ev}(\kappa_\tau)-\alpha\big)
%   \big|_{E^X_{\mathrm{ev},\alpha}}.
% \]
For $\mathrm{ev}=\mathrm{ev}_N$, Lemma \ref{lem:guere-dictionary} identifies
these numbers with those computed from $\boldsymbol\kappa_{b_N}$.
\end{definition}

\begin{definition}[\cite{JeremyCubic}*{Def.~27}]\label{def:properties}
Let $K$ be a number field.  Property $\clubsuit_K$ is the condition that,
outside the finite exceptional set allowed in Gu\'er\'e's Definition~27,
every normalized $K$-evaluation map nonzero on $Q'$ and every eigenvalue
satisfy
\[
  \nu=0\qquad\text{or}\qquad\rho\ge3.
\]
Property $\heartsuit_K$ is the analogous condition
\[
  \nu=0\qquad\text{or}\qquad\nu'\neq0
  \qquad\text{or}\qquad\gamma\ge2.
\]
\end{definition}

\begin{lemma}\label{lem:failure}
To prove that $X$ fails $\clubsuit_K$ or $\heartsuit_K$, it is enough to
exhibit infinitely many normalized $K$-evaluation maps, nonzero on $Q'$, for
which one eigenvalue violates the corresponding condition.
\end{lemma}

\begin{proof}
A finite exceptional set cannot contain infinitely many such evaluation
maps.
\end{proof}

\begin{theorem}[\cite{JeremyCubic}]\label{thm:imported}
\begin{defenum}
\item[\textup{(a)}]
\textup{(Cor.~46)} Every rational smooth complex projective variety of
dimension at most four satisfies $\clubsuit_{K_0}$, where
\[
  K_0=\Q(e^{2\pi\mathbf i/6},\sqrt2,\sqrt3).
\]
\item[\textup{(b)}]
\textup{(Rem.~29)} If $K_1\subseteq K_2$, then
\[
  \clubsuit_{K_2}\Longrightarrow\clubsuit_{K_1},
  \qquad
  \heartsuit_{K_2}\Longrightarrow\heartsuit_{K_1}.
\]
\item[\textup{(c)}]
\textup{(Prop.~53 and Rem.~52)} If a smooth projective surface fails
$\heartsuit_K$ for $K\supseteq\Q(\mathbf i)$, then its minimal model is a
K3 surface.
\item[\textup{(d)}]
\textup{(Rem.~42)} Starting from an arbitrary weak factorization, one can
choose compatible evaluation maps, after finite extension of the number
field when necessary, with disjoint spectra so that the one-sided blow-up
formula is additive for
\[
  \epsilon\in\{\rho,\gamma,\nu,\nu'\}.
\]
\item[\textup{(e)}]
\textup{(Prop.~54 and Rem.~42)} For a surface center with K3 minimal model
$\Sigma$, after a finite extension $\widehat K$ of the coefficient number
field there is a graded embedding
\[
  \rH^\bullet(\Sigma,\widehat K)_{S_{\widehat K}}(-1)
  \hookrightarrow
  E^X_{\widehat{\mathrm{ev}},\alpha}
\]
lying in the linear span of morphisms of rational Hodge structures.
\end{defenum}
\end{theorem}

\medskip

\subsubsection{The fixed block and the spectral invariants}

At a small point, Theorem \ref{theo:scalar_on_van} gives
\begin{equation}\label{eq:eigsplit}
  E_{b,\lambda}
  =
  E^{B_b}_\lambda
  \oplus
  \begin{cases}
    W_{\Bbbk},&\lambda=\lambda_{\mathrm{van}}(b),\\
    0,&\lambda\neq\lambda_{\mathrm{van}}(b).
  \end{cases}
\end{equation}
where $E^{B_b}_\lambda$ is the generalized eigenspace of the $Y$-block.

\begin{lemma}\label{lem:heart-invariants-from-A}
Let $b$ be a small point and let $m_a^Y(\lambda)$ and $m_g^Y(\lambda)$ be
the algebraic and geometric multiplicities of $\lambda$ as an eigenvalue of
$B_b$, both taken to be zero if $\lambda\notin\operatorname{spec}(B_b)$.
Then
\begin{defenum}
\item
\[
  \nu_{b,\lambda}
  =
  \begin{cases}
  1,&\lambda=\lambda_{\mathrm{van}}(b),\\
  0,&\text{otherwise}
  \end{cases}
\]
\item $\nu'_{b,\lambda}=0$ for every $\lambda$
\item
\[
  \gamma_{b,\lambda}
  =m_a^Y(\lambda)-m_g^Y(\lambda).
\]
If \textup{(AC)} holds, these multiplicities are precisely those of the
flag-ambient matrix $A_b$.
\end{defenum}
\end{lemma}

\begin{proof}
For a fourfold, Hochschild degree two consists of Hodge types
$(2,0)$, $(3,1)$, and $(4,2)$.  Lemma \ref{lem:automatic} gives
$h^{2,0}(X)=0$, and Serre duality gives $h^{4,2}(X)=0$.  Thus
$\rH^{(2)}(X)=\rH^{3,1}(X)$, of dimension one.  Since
$\rH^4(Y)$ is of type $(2,2)$, this line lies in $W_\C$.  The eigenspace
splitting above gives (a).  Lemma \ref{lem:automatic} and Poincar\'e duality
give $\rH^{\mathrm{odd}}(X)=0$, hence (b).  The vanishing summand is killed
by $\boldsymbol\kappa_b-\lambda$ whenever it is present, so the rank is
exactly the Jordan defect of the fixed block, proving (c).
\end{proof}

\begin{lemma}\label{lem:invariants-from-A}
Assume in addition that $X$ is Hodge general.  Then, with the notation of
Lemma \ref{lem:heart-invariants-from-A},
\[
  \rho_{b,\lambda}=m_a^Y(\lambda).
\]
Under \textup{(AC)}, this is the algebraic multiplicity of $\lambda$ in
$A_b$.
\end{lemma}

\begin{proof}
By Lemma \ref{lem:hdg-is-amb}, all rational Hodge classes lie in the
$Y$-part and the vanishing summand contains none.  Hence the Hodge-class
part of the generalized eigenspace is exactly $E^{B_b}_\lambda$.
\end{proof}

\begin{remark}\label{rmk:lambda-not-fixed-spectrum}
Only the fixed-block Jordan defect at $\lambda_{\mathrm{van}}$ is used in
Theorem \ref{theo:rational_Fano_K3}.  If
$\lambda_{\mathrm{van}}(b)\notin\operatorname{spec}(B_b)$, then
$m_a^Y=m_g^Y=0$ there and consequently $\gamma=0$ automatically.
\end{remark}

\begin{definition}\label{def:Aalg1}
Since $X$ is Fano, only finitely many curve classes can occur in each entry
of small quantum multiplication by $c_1(TX)$.  We denote by
\[
  B^{\mathrm{alg}}(1)
\]
the matrix of the $Y$-block after the algebraic augmentation
$q^\beta\mapsto1$, and by
\[
  A^{\mathrm{alg}}(1)
\]
the corresponding matrix on the flag-ambient block computed via Theorem \ref{thm:localizationCI}.  Under \textup{(AC)}, the two matrices coincide.
\end{definition}

\begin{lemma}\label{lem:scaling}
Let $0<|s|<1$ and let $b_s$ be the small point
\[
  q^\beta(b_s)=s^{\langle c_1(TX),\beta\rangle},
  \qquad
  y_i(b_s)=0.
\]
Then $b_s\in B_X^{\mathsf{Hod}}(\Bbbk)$.  In any homogeneous basis of the
$Y$-part, with basis element of degree $2p_i$ and
$D_s=\mathrm{diag}(s^{p_i})$, one has
\[
B_{b_s}=D_s^{-1}\big(sB^{\mathrm{alg}}(1)\big)D_s.
\]
The analogous identity holds for the flag-ambient block $A$.  Therefore,
algebraic multiplicities, geometric multiplicities, and Jordan defects are
independent of $s$.
\end{lemma}

\begin{proof}
Because $c_1(TX)$ is ample,
$\mathrm{pr}(b_s)=\log|s|\,c_1(TX)\in-\mathcal U_X$.  If a monomial
$q^\beta$ occurs in the $(i,j)$ entry, homogeneity gives
\[
  \langle c_1(TX),\beta\rangle=p_j-p_i+1,
\]
so that entry specializes to $a_{ij}s^{p_j-p_i+1}$.  Conjugation by $D_s$
gives the displayed formula.
\end{proof}

\medskip

\subsubsection{Two spectral consequences}

The first consequence uses Hodge generality through $\rho$.

\begin{theorem}\label{thm:irrat-from-A}
Assume Standing Assumption \ref{stand:setup} with $n=4$.  Suppose that $X$
is Fano, $h^{3,1}(X)=1$, and Hodge general.  If every eigenvalue of
$B^{\mathrm{alg}}(1)$ has algebraic multiplicity at most two, then $X$ is
irrational.  In particular, under \textup{(AC)} it is enough to check this
condition on the computable matrix $A^{\mathrm{alg}}(1)$.
\end{theorem}

\begin{proof}
For every $N\ge1$, use the $\Q$-evaluation map $\mathrm{ev}_N$ of Lemma
\ref{lem:guere-dictionary}.  Let $\alpha_N$ correspond to the scalar
eigenvalue on $W$.  Lemmas \ref{lem:scaling},
\ref{lem:heart-invariants-from-A}, and \ref{lem:invariants-from-A} give
\[
  \nu_{\mathrm{ev}_N,\alpha_N}=1,
  \qquad
  \rho_{\mathrm{ev}_N,\alpha_N}\le2.
\]
Thus, every $\mathrm{ev}_N$ violates $\clubsuit_\Q$.  The maps are distinct
as $N$ varies, hence $X$ fails $\clubsuit_\Q$.  By Theorem
\ref{thm:imported}\textup{(a),(b)}, every rational fourfold satisfies
$\clubsuit_\Q$.  Therefore, $X$ is irrational.
\end{proof}

The second consequence uses only $\nu,\nu'$, and $\gamma$ and hence does not
require Hodge generality.

\begin{theorem}\label{theo:rational_Fano_K3}
Assume Standing Assumption \ref{stand:setup} with $n=4$.  Suppose that $X$
is Fano and $h^{3,1}(X)=1$.  If every eigenvalue of
$B^{\mathrm{alg}}(1)$ has Jordan defect at most one and $X$ is rational,
then every weak factorization of $X\dashrightarrow\PP^4$ contains a smooth
surface center whose minimal model is a K3 surface.

More precisely, fix such a weak factorization.  After extending the initial
$\Q$-evaluation map first to
\[
  K_0=\Q(e^{2\pi\mathbf i/6},\sqrt2,\sqrt3)
  \supset\Q(\mathbf i)
\]
and then, if necessary, to a finite extension $\widetilde K$, one may choose
compatible evaluation maps as in \cite{JeremyCubic}*{Rem.~42}.  For one of
the surface centers, with K3 minimal model $\Sigma$, a further finite
extension $\widehat K/\widetilde K$ yields an embedding
\[
  \rH^\bullet(\Sigma,\widehat K)_{S_{\widehat K}}(-1)
  \hookrightarrow
  E^X_{\widehat{\mathrm{ev}},\alpha}
\]
lying in the linear span of morphisms of rational Hodge structures.
Under \textup{(AC)}, the Jordan-defect hypothesis can be checked on
$A^{\mathrm{alg}}(1)$.
\end{theorem}

\begin{proof}
Choose $\mathrm{ev}_N$ and let $\alpha_N$ correspond to the scalar action on
$W$.  By Lemmas \ref{lem:heart-invariants-from-A} and \ref{lem:scaling},
\[
  \nu_{\mathrm{ev}_N,\alpha_N}=1,
  \qquad
  \nu'_{\mathrm{ev}_N,\alpha_N}=0,
  \qquad
  \gamma_{\mathrm{ev}_N,\alpha_N}\le1.
\]
These numbers are unchanged after finite extension of the coefficient field.
Fix an arbitrary weak factorization and extend scalars to a finite
$\widetilde K\supseteq K_0$ large enough for the compatible evaluations of
Remark~42.  Since $K_0\supset\Q(\mathbf i)$, \cite{JeremyCubic}*{Proposition~53} applies over
$\widetilde K$.

\cite{JeremyCubic}*{Remark~42} gives, for
$\epsilon\in\{\nu,\nu',\gamma\}$, a one-sided equality
\[
  \epsilon^X_{\widetilde{\mathrm{ev}},\alpha}
  =
  \epsilon^{\PP^4}_{\widetilde{\mathrm{ev}},\alpha}
  +
  \sum_{(i,k)\in I}
  \epsilon^{W_i}_{\widetilde{\mathrm{ev}}_{i,k},\alpha},
\]
with nonnegative summands.  Since $\PP^4$ is of type $(p,p)$, its $\nu$ term
vanishes.  The equality $\nu^X=1$ forces exactly one center contribution to
have $\nu=1$; then $\nu'^X=0$ and $\gamma^X\le1$ force that same center to
have $\nu'=0$ and $\gamma\le1$.  It therefore fails $\heartsuit_{\widetilde
K}$.  Centers in a weak factorization of a fourfold have dimension at most
two, and points and curves have no Hochschild-degree-two contribution, so
the center is a surface.  Theorem \ref{thm:imported}\textup{(c)} gives a K3
minimal model.  Finally, Theorem \ref{thm:imported}\textup{(e)} gives the
asserted embedding after a further finite extension.
\end{proof}

\medskip

\subsubsection{The numerical criterion}

For comparison, we record the criterion of
\cite{Benedetti_Fay_Guere_Manivel_Perrin}, which uses no quantum
computation.

\begin{theorem}[\cite{Benedetti_Fay_Guere_Manivel_Perrin}*{Thm.~4.1}]
\label{thm:BFGMPcriterion}
Let $Y$ be a smooth Fano fivefold in $\PP^M$ and let $X$ be a smooth Fano
hyperplane section satisfying
\[
  b_1(X)=b_3(X)=0,
  \qquad
  h^{3,1}(X)>h^{3,1}(Y),
  \qquad
  b_4(X)_{\mathrm{van}}\ge10+12h^{3,1}(X).
\]
If $X$ is Hodge general, then $X$ is irrational.
\end{theorem}

\begin{corollary}\label{cor:numerical}
Assume Standing Assumption \ref{stand:setup} with $n=4$, and suppose that
$X$ is Fano, $h^{3,1}(X)=1$, and Hodge general.  Then $\rk\rH^4(X)^{\mathrm{Hdg}}=b_4(Y)$, and $X$ is irrational as soon as
\[
  b_4(Y)\le h^{2,2}(X)-20.
\]
Equivalently,
\[
  \rk\rH^4(X)^{\mathrm{Hdg}}
  \le h^{2,2}(X)-20.
\]
\end{corollary}

\begin{proof}
Hypothesis \textup{(H3)} says that $Y$ is Fano, and Lemma
\ref{lem:automatic} gives $b_1(X)=b_3(X)=0$ and
$h^{3,1}(Y)=0<h^{3,1}(X)=1$.  Hodge generality and \textup{(H2)} give
\[
  \rH^4(X)^{\mathrm{Hdg}}=j^*\rH^4(Y),
\]
and weak Lefschetz makes $j^*$ injective in degree four.  Thus the rank of
the Hodge classes is $b_4(Y)$.

Since $X$ is Fano, $h^{4,0}(X)=0$, so
\[
  b_4(X)=h^{2,2}(X)+2.
\]
Lemma \ref{lem:vananb} gives
\[
  b_4(X)_{\mathrm{van}}
  =b_4(X)-b_4(Y)
  =h^{2,2}(X)+2-b_4(Y).
\]
The inequality of Theorem \ref{thm:BFGMPcriterion} is therefore equivalent
to $b_4(Y)\le h^{2,2}(X)-20$.
\end{proof}

\begin{corollary}\label{cor:numerical-combinatorial}
Under the hypotheses of Corollary \ref{cor:numerical}, assume in addition
that \textup{(AC)} holds.  Then
\[
  b_4(Y)=b_4(F),
\]
and $X$ is irrational as soon as
\[
  b_4(F)\le h^{2,2}(X)-20,
\]
where $b_4(F)$ is the number of codimension-two Schubert classes on $F$.
\end{corollary}

\begin{proof}
Condition \textup{(AC)} identifies $\rH^4(F,\Q)$ with $\rH^4(Y,\Q)$.
Proposition \ref{prop:schubertbasis} identifies $b_4(F)$ with the number of
codimension-two Schubert classes.  The claim follows from Corollary
\ref{cor:numerical}.
\end{proof}

%===============================================================================================================================================================================================================================================================================

\bigskip

\section{Examples}
In this section, we illustrate how the software we have developed computes the matrix of small
quantum multiplication by the first Chern class. For each considered variety, we record,
in an input and an output box, the data required to run the computation and
the matrix it returns; we then read off what these say about rationality.

\medskip

\begin{example}[Cubic fourfolds]\label{ex:cubic4fold}

Let
\[
X=\{f=0\}\subset\PP^{5}
\]
be a smooth cubic fourfold, and let $h=c_{1}\big(\cO_{X}(1)\big)$.
For the quantum computation, we use directly the realization
\[
X=\mathscr Z\big(\PP^{5},\cO_{\PP^{5}}(3)\big).
\]

For the birational applications of this section, we also realize $X$ as a
hyperplane section of a smooth cubic fivefold.  More precisely, write
$\PP^{5}=\{x_{6}=0\}\subset\PP^{6}$ and choose a general quadratic form
$g(x_{0},\dots,x_{6})$.  Then
\[
Y
:=
\big\{
f(x_{0},\dots,x_{5})+x_{6}g(x_{0},\dots,x_{6})=0
\big\}
\subset\PP^{6}
\]
is smooth for general $g$, and
\[
X=Y\cap\{x_{6}=0\}.
\]
Indeed, $Y$ is automatically smooth along $X$, since the derivatives of
$f$ with respect to $x_{0},\dots,x_{5}$ do not vanish simultaneously on
the smooth hypersurface $X$, while smoothness away from the base locus
follows from Bertini.

Thus, for the notation of Standing Assumption \ref{stand:setup}, we may
take
\[
F=\PP^{6},
\qquad
\cE=\cO_{F}(3),
\qquad
Y=\mathscr Z(F,\cE),
\qquad
X=\mathscr Z\big(F,\cE\oplus\cO_{F}(1)\big).
\]
The standing assumptions hold with $n=4$.  Moreover, by the Lefschetz
hyperplane theorem,
\[
\rH^{4}(F,\Q)\xrightarrow{\sim}\rH^{4}(Y,\Q),
\]
so condition \textup{(AC)} of Definition \ref{def:ambient-complete} holds.
Consequently,
\[
\rH_{F}^{*}(X)
=
\rH_{Y}^{*}(X)
=
\rH_{\mathrm{amb}}^{*}(X)
=
\operatorname{Span}_{\Q}\{1,h,h^{2},h^{3},h^{4}\}.
\]
In particular, in this example the flag-ambient block computed by
localization is the full monodromy-fixed block appearing in
Theorems \ref{thm:irrat-from-A} and \ref{theo:rational_Fano_K3}.

\medskip

The small quantum multiplication matrix is obtained from the following
input.

\begin{inputbox}[Input]
algebra  A5         keep  1
bundle   O(3)
\end{inputbox}

\begin{outputbox}[Output]
\noindent
Matrix of $c_{1}(TX)\star$ in the reduced Schubert basis
$1,h,h^{2},h^{3},h^{4}$, with Novikov variable $y_{1}$:
{\setlength{\arraycolsep}{3.5pt}\small
\[
A(y_{1})
=
\begin{pmatrix}
0 & 0 & 18y_{1} & 0 & 0 \\
3 & 0 & 0 & 45y_{1} & 0 \\
0 & 3 & 0 & 0 & 18y_{1} \\
0 & 0 & 3 & 0 & 0 \\
0 & 0 & 0 & 3 & 0
\end{pmatrix}.
\]
}

\noindent
Grading operator:
\[
\mathrm{diag}\Big(\ell(w)-\frac12\dim X\Big)
=
\mathrm{diag}(-2,-1,0,1,2).
\]

\medskip

\noindent
Characteristic polynomial:
\[
\chi_{A(y_{1})}(\lambda)
=
\lambda^{2}\big(\lambda^{3}-729y_{1}\big).
\]

\medskip

\noindent
Eigenvalues at the algebraic specialization $y_{1}=1$:
\[
\big\{
0,\ 0,\ 9,\ -4.5-7.79423\,\mathbf i,\ -4.5+7.79423\,\mathbf i
\big\}.
\]
\end{outputbox}

This matrix is not a new quantum-cohomology computation: after identifying
$Q=y_{1}$, it is exactly the matrix displayed by Gu\'er\'e in
\cite{JeremyCubic}*{Ex.~16}, where it is attributed to Givental and to the
proof of \cite{KKPY}*{Thm.~6.8}. 

The characteristic polynomial factors as
\[
\lambda^{2}\big(\lambda^{3}-729y_{1}\big),
\]
so the nonzero eigenvalues are
\[
9\,\zeta^{k}y_{1}^{1/3},
\qquad
k=0,1,2,
\]
where $\zeta$ is a primitive cube root of unity.  At $y_{1}=1$ these are
\[
9,
\qquad
-\frac92+\frac{9\sqrt3}{2}\,\mathbf i,
\qquad
-\frac92-\frac{9\sqrt3}{2}\,\mathbf i.
\]
The eigenvalue $0$ has algebraic multiplicity two and geometric
multiplicity one; equivalently, its Jordan defect is one.  All nonzero
eigenvalues are simple.

Gu\'er\'e also observes in \cite{JeremyCubic}*{Ex.~16} that small quantum
multiplication by $c_{1}(TX)$ vanishes on the primitive cohomology.  In the
notation of this paper,
\[
\lambda_{\mathrm{van}}=0.
\]
This is also an immediate consequence of
Lemma \ref{lem:fano-index-vanishing}, since a cubic fourfold has Fano index
three.

\medskip

The Hodge diamond of a smooth cubic fourfold is
\[
\begin{array}{ccccccccc}
&&&&1&&&&\\
&&&0&&0&&&\\
&&0&&1&&0&&\\
&0&&0&&0&&0&\\
0&&1&&21&&1&&0\\
&0&&0&&0&&0&\\
&&0&&1&&0&&\\
&&&0&&0&&&\\
&&&&1&&&&
\end{array}.
\]
Since \textup{(AC)} holds, the part restricted from $Y$ coincides with the
part restricted from $F$.  Thus the Hodge diamond decomposes into the
monodromy-fixed ambient part, in \textcolor{blue}{blue}, and the vanishing
part, in \textcolor{red}{red}:
\begin{equation}
\setlength{\arraycolsep}{3pt}
\begin{array}{ccccccccc}
&&&& \textcolor{blue}{1} &&&&\\
&&&0&&0&&&\\
&&0&&\textcolor{blue}{1}&&0&&\\
&0&&0&&0&&0&\\
0&&\textcolor{red}{1}&&
\mathclap{\textcolor{blue}{1}+\textcolor{red}{20}}
&&\textcolor{red}{1}&&0\\
&0&&0&&0&&0&\\
&&0&&\textcolor{blue}{1}&&0&&\\
&&&0&&0&&&\\
&&&&\textcolor{blue}{1}&&&&
\end{array}
\quad=\quad
\begin{array}{ccccccccc}
&&&& \textcolor{blue}{1} &&&&\\
&&&0&&0&&&\\
&&0&&\textcolor{blue}{1}&&0&&\\
&0&&0&&0&&0&\\
0&&0&&\textcolor{blue}{1}&&0&&0\\
&0&&0&&0&&0&\\
&&0&&\textcolor{blue}{1}&&0&&\\
&&&0&&0&&&\\
&&&&\textcolor{blue}{1}&&&&
\end{array}
\quad+\quad
\begin{array}{ccccccccc}
&&&&0&&&&\\
&&&0&&0&&&\\
&&0&&0&&0&&\\
&0&&0&&0&&0&\\
0&&\textcolor{red}{1}&&\textcolor{red}{20}
&&\textcolor{red}{1}&&0\\
&0&&0&&0&&0&\\
&&0&&0&&0&&\\
&&&0&&0&&&\\
&&&&0&&&&
\end{array}.
\end{equation}

Hence
\[
h^{3,1}(X)=1,
\qquad
h^{2,2}(X)=21,
\qquad
b_{4}(X)=23.
\]
Furthermore,
\[
b_{4}(Y)=1,
\]
and therefore
\[
b_{4}(X)_{\mathrm{van}}
=
b_{4}(X)-b_{4}(Y)
=
23-1
=
22.
\]
Since
\[
j^{*}\rH^{4}(Y,\Q)=\Q h^{2},
\]
we have in this case
\[
\rH^{4}(X)_{\mathrm{van}}
=
\rH^{4}(X)_{\mathrm{prim}},
\]
with Hodge numbers
\[
\big(
h_{\mathrm{van}}^{3,1},
h_{\mathrm{van}}^{2,2},
h_{\mathrm{van}}^{1,3}
\big)
=
(1,20,1).
\]

For a very general cubic fourfold,
Proposition \ref{prop:hodgegeneral} applies because
\[
h^{3,1}(Y)=0<1=h^{3,1}(X).
\]
Thus $X$ is Hodge general and
\[
\rH^{4}(X)^{\mathrm{Hdg}}
=
\Q h^{2},
\qquad
\rk\,\rH^{4}(X)^{\mathrm{Hdg}}=1.
\]

\medskip

The three results of the preceding subsection have the following
consequences.

\medskip

\emph{The numerical irrationality criterion.}
Since \textup{(AC)} holds, Corollary
\ref{cor:numerical-combinatorial} applies.  Here
\[
b_{4}(F)=b_{4}(\PP^{6})=1,
\]
the unique codimension-two Schubert class being $h^{2}$.  Hence
\[
b_{4}(F)\le h^{2,2}(X)-20
\]
becomes
\[
1\le21-20=1.
\]
The inequality therefore holds with equality.  Equivalently,
\[
b_{4}(X)_{\mathrm{van}}
=
22
=
10+12h^{3,1}(X).
\]
Thus every Hodge-general cubic fourfold is irrational.

\medskip

\emph{The spectral irrationality criterion.}
Under \textup{(AC)} one has
\[
B^{\mathrm{alg}}(1)=A^{\mathrm{alg}}(1)=A(1).
\]
Every eigenvalue of this matrix has algebraic multiplicity at most two,
the only repeated eigenvalue being $0$.  Therefore Theorem \ref{thm:irrat-from-A} applies to every Hodge-general cubic fourfold and again gives its irrationality.

In particular, either the numerical or the spectral criterion recovers
the theorem of Katzarkov--Kontsevich--Pantev--Yu:
\begin{theorem}[Katzarkov--Kontsevich--Pantev--Yu, \cite{KKPY}]
A very general cubic fourfold is irrational.
\end{theorem}

\medskip

\emph{The structural criterion for rational special members.}
The situation is different for Theorem \ref{theo:rational_Fano_K3}:
Hodge generality is not assumed there.  Since the matrix
$A^{\mathrm{alg}}(1)$ is deformation invariant in the smooth cubic family
and has Jordan defect one at $0$ and zero at every other eigenvalue, the
hypothesis of Theorem \ref{theo:rational_Fano_K3} is satisfied for every
smooth cubic fourfold.

Consequently, if a smooth cubic fourfold $X$ were rational, every weak
factorization
\[
X\dashrightarrow\PP^{4}
\]
would contain a smooth surface center whose minimal model is a K3 surface.
Moreover, after choosing compatible evaluation maps and making the finite
field extensions appearing in Theorem
\ref{theo:rational_Fano_K3}, the cohomology of such a K3 surface embeds
into the corresponding evaluated generalized eigenspace of $X$ in the
linear span of morphisms of rational Hodge structures.

For cubic fourfolds, Gu\'er\'e proves the stronger statement
\cite{JeremyCubic}*{Thm.~56}: if $X$ is rational, then there exists a
projective K3 surface $\Sigma$ and an isomorphism of rational Hodge
structures
\[
\rH^{4}(X,\Q)_{\mathrm{prim}}
\cong
\rH^{2}(\Sigma,\Q)(-1).
\]
Thus the revised structural criterion is compatible with the complete
cubic calculation and recovers the K3-center mechanism uniformly, while
Gu\'er\'e's case-specific argument supplies the stronger Hodge
isomorphism.
\end{example}

\medskip

\begin{example}[$Z_{66}=$ K\"uchle $(c5)$]\label{ex:kuchle-c5}
Let
\[
F:=\mathrm{Gr}(3,7),
\]
with tautological subbundle $\cU$ of rank $3$ and quotient bundle $\cQ$ of
rank $4$.  Let
\[
Y
:=
\mathscr Z\big(F,\Lambda^{2}\cU^{\vee}\oplus\Lambda^{3}\cQ\big)
\]
be a general K\"uchle fivefold of type $(c5)$, and let
\[
X
:=
\mathscr Z\big(F,
\Lambda^{2}\cU^{\vee}\oplus\Lambda^{3}\cQ\oplus\cO_F(1)\big)
\subset Y
\]
be a smooth hyperplane section.

Notice that, since $\operatorname{rk}\cU=3$ and $\operatorname{rk}\cQ=4$,
\[
\Lambda^{2}\cU^{\vee}\cong\cU(1),
\qquad
\Lambda^{3}\cQ\cong\cQ^{\vee}(1),
\]
so this is the standard realization of the K\"uchle varieties of type
$(c5)$ used in \cite{Kuznetsov2016}.

Kuznetsov proves that $Y$ is a smooth Fano fivefold of index two,
\[
-K_Y=2H,
\qquad H:=c_1\big(\cO_F(1)|_Y\big),
\]
and that its Hodge diamond is diagonal, with
\[
h^{0,0}(Y)=h^{1,1}(Y)=h^{4,4}(Y)=h^{5,5}(Y)=1,
\qquad
h^{2,2}(Y)=h^{3,3}(Y)=4
\]
\cite{Kuznetsov2016}.  In particular,
\[
b_4(Y)=4,
\qquad
h^{3,1}(Y)=0.
\]
By adjunction,
\[
-K_X=H|_X,
\]
so $X$ is a Fano fourfold of index one.  Its numerical invariants are
\[
(-K_X)^4=66,
\qquad
h^{3,1}(X)=1,
\qquad
h^{2,2}(X)=24,
\qquad
b_4(X)=26;
\]
see \cite{Kuznetsov2016} and the references therein.

These facts verify Standing Assumption \ref{stand:setup} with $n=4$.
Indeed, the Picard groups are generated by the corresponding hyperplane
classes, the required comparison in degrees $<4$ holds, $\rH^4(Y,\Q)$ is
entirely of type $(2,2)$, and $Y$ is Fano.  Notice, however, that the
additional ambient-completeness condition \textup{(AC)} of
Definition \ref{def:ambient-complete} does \emph{not} hold.  Namely,
\[
b_4(F)=2
\qquad\text{whereas}\qquad
b_4(Y)=4.
\]
Thus
\[
\rH_F^4(X)
=
\iota^*\rH^4(F,\Q)
\subsetneq
j^*\rH^4(Y,\Q)
=
\rH_Y^4(X),
\]
with
\[
\dim\rH_F^4(X)=2,
\qquad
\dim\rH_Y^4(X)=4.
\]
This distinction is essential in interpreting the quantum computation
below.

\medskip

The flag-ambient matrix of small quantum multiplication by $c_1(TX)$ is
computed by the following input.

\begin{inputbox}[Input]
algebra  A6         keep  3
bundle   osum(wedge(2,dual(S(3))), wedge(3,Q(3)), O(1))
         # equivalent: osum(wedge(2,dual(S(3))),
         #                  tensor(dual(Q(3)),O(1)), O(1))
\end{inputbox}

\begin{outputbox}[Output]
\noindent
Matrix of $c_{1}(TX)\star$ in the reduced Schubert basis, with quantum
variable $y_{3}$:
{\setlength{\arraycolsep}{3.5pt}\small
\[
A(y_3)=
\begin{pmatrix}
0 & 72\,y_{3}^{2} & 732\,y_{3}^{3} & 888\,y_{3}^{3} & 2496\,y_{3}^{4} & 18720\,y_{3}^{5}\\[2pt]
1 & 5\,y_{3}      & 92\,y_{3}^{2}  & 112\,y_{3}^{2} & 324\,y_{3}^{3}  & 2496\,y_{3}^{4}\\[2pt]
0 & 1             & 0              & 4\,y_{3}       & 12\,y_{3}^{2}   & 108\,y_{3}^{3}\\[2pt]
0 & 1             & 6\,y_{3}       & 4\,y_{3}       & 24\,y_{3}^{2}   & 180\,y_{3}^{3}\\[2pt]
0 & 0             & 5              & 6              & 5\,y_{3}        & 72\,y_{3}^{2}\\[2pt]
0 & 0             & 0              & 0              & 1               & 0
\end{pmatrix}.
\]
}

\noindent
Grading operator:
\[
\mathrm{diag}\Big(\ell(w)-\frac12\dim X\Big)
=
\mathrm{diag}(-2,-1,0,0,1,2).
\]

\medskip

At the algebraic specialization $y_3=1$, the characteristic polynomial is
\[
\chi_{A^{\mathrm{alg}}(1)}(\lambda)
=
(\lambda+4)^2(\lambda+5)
\big(\lambda^3-27\lambda^2-216\lambda-432\big).
\]
Thus, the eigenvalues are approximately
\[
\big\{
33.7741424,\,
-5,\,
-4,\,
-4,\,
-3.3870712+1.1483027\,\mathbf i,\,
-3.3870712-1.1483027\,\mathbf i
\big\}.
\]
The eigenvalue $-4$ has algebraic multiplicity two and geometric
multiplicity one; all the remaining eigenvalues are simple.
\end{outputbox}

\medskip

It is important to emphasize what this computation represents.  The matrix
$A(y_3)$ acts only on
\[
\rH_F^\bullet(X)=\iota^*\rH^\bullet(F,\Q),
\]
the part of the cohomology restricted from the Grassmannian.  Since
\textup{(AC)} fails, this is strictly smaller in degree four than the full
monodromy-fixed part
\[
\rH_Y^\bullet(X)=j^*\rH^\bullet(Y,\Q).
\]
In particular, there are two additional degree-four $(2,2)$-classes
restricted from $Y$ which are not seen by the displayed $6\times6$ matrix.
Consequently, the repeated eigenvalue $-4$ of $A^{\mathrm{alg}}(1)$ cannot,
from this computation alone, be identified with
$\lambda_{\mathrm{van}}$, nor can the Jordan data of
$B^{\mathrm{alg}}(1)$ be read from the matrix above.

\medskip

The Hodge diamond of $X$ is
\[
\begin{array}{ccccccccc}
&&&& 1 &&&&\\
&&&0&&0&&&\\
&&0&&1&&0&&\\
&0&&0&&0&&0&\\
0&&1&&24&&1&&0\\
&0&&0&&0&&0&\\
&&0&&1&&0&&\\
&&&0&&0&&&\\
&&&&1&&&&
\end{array}.
\]
For the purposes of the Lefschetz-monodromy decomposition, the relevant
splitting is not the decomposition into the part restricted from $F$ and
its orthogonal complement, but rather
\[
\rH^\bullet(X,\Q)
=
\rH_Y^\bullet(X)
\oplus
\rH^4(X)_{\mathrm{van}}.
\]
At the level of Hodge numbers, this takes the form
\[
\setlength{\arraycolsep}{3pt}
\begin{array}{ccccccccc}
&&&& \textcolor{blue}{1} &&&&\\
&&&0&&0&&&\\
&&0&&\textcolor{blue}{1}&&0&&\\
&0&&0&&0&&0&\\
0&&\textcolor{red}{1}&&
\mathclap{\textcolor{blue}{4}+\textcolor{red}{20}}
&&\textcolor{red}{1}&&0\\
&0&&0&&0&&0&\\
&&0&&\textcolor{blue}{1}&&0&&\\
&&&0&&0&&&\\
&&&&\textcolor{blue}{1}&&&&
\end{array}
\quad=\quad
\begin{array}{ccccccccc}
&&&& \textcolor{blue}{1} &&&&\\
&&&0&&0&&&\\
&&0&&\textcolor{blue}{1}&&0&&\\
&0&&0&&0&&0&\\
0&&0&&\textcolor{blue}{4}&&0&&0\\
&0&&0&&0&&0&\\
&&0&&\textcolor{blue}{1}&&0&&\\
&&&0&&0&&&\\
&&&&\textcolor{blue}{1}&&&&
\end{array}
\quad+\quad
\begin{array}{ccccccccc}
&&&&0&&&&\\
&&&0&&0&&&\\
&&0&&0&&0&&\\
&0&&0&&0&&0&\\
0&&\textcolor{red}{1}&&\textcolor{red}{20}&&\textcolor{red}{1}&&0\\
&0&&0&&0&&0&\\
&&0&&0&&0&&\\
&&&0&&0&&&\\
&&&&0&&&&
\end{array}.
\]
Here blue denotes the part restricted from $Y$, while red denotes the
vanishing cohomology.  Therefore
\[
h^{3,1}(X)=1,
\qquad
h^{2,2}(X)=24,
\qquad
b_4(X)=26,
\]
but
\[
\boxed{
b_4(X)_{\mathrm{van}}
=
b_4(X)-b_4(Y)
=
26-4
=
22.
}
\]
In particular, the vanishing Hodge structure has Hodge numbers
\[
\big(
h^{3,1}_{\mathrm{van}},
h^{2,2}_{\mathrm{van}},
h^{1,3}_{\mathrm{van}}
\big)
=
(1,20,1).
\]
The flag-ambient part accounts for only two of the four blue
$(2,2)$-classes in the middle degree:
\[
\dim\rH_F^4(X)=2
<
4=\dim\rH_Y^4(X).
\]

\medskip

For a very general member $X$, Proposition \ref{prop:hodgegeneral} applies
because
\[
h^{3,1}(Y)=0<1=h^{3,1}(X).
\]
Hence $X$ is Hodge general and
\[
\rH^4(X)^{\mathrm{Hdg}}
=
j^*\rH^4(Y,\Q),
\qquad
\rk\,\rH^4(X)^{\mathrm{Hdg}}=4.
\]

\medskip

\emph{The numerical criterion.}
This example is an application of Corollary \ref{cor:numerical}, but
\emph{not} of Corollary \ref{cor:numerical-combinatorial}.  Indeed,
\[
b_4(Y)=4,
\qquad
h^{2,2}(X)-20=24-20=4,
\]
so the numerical inequality holds with equality:
\[
b_4(Y)\le h^{2,2}(X)-20.
\]
Equivalently,
\[
b_4(X)_{\mathrm{van}}
=
22
=
10+12h^{3,1}(X).
\]
Thus Corollary \ref{cor:numerical} recovers the irrationality of a
Hodge-general K\"uchle fourfold of type $(c5)$ established in
\cite{Benedetti_Fay_Guere_Manivel_Perrin}*{Prop.~5.1}:

\begin{theorem}[Benedetti--Fay--Gu\'er\'e--Manivel--Perrin,
\cite{Benedetti_Fay_Guere_Manivel_Perrin}*{Prop.~5.1}]
Every Hodge-general K\"uchle fourfold of type $(c5)$ is irrational.
In particular, a very general K\"uchle fourfold of type $(c5)$ is
irrational.
\end{theorem}
\end{example}

\medskip
    
\begin{example}[Gushel--Mukai fourfolds]\label{ex:GM4fold}

We consider first an ordinary Gushel--Mukai fourfold.  Thus
\[
F:=\mathrm{Gr}(2,5)\subset\PP^{9}
\]
is the Pl\"ucker Grassmannian and
\[
X
=
F\cap H\cap Q
=
\mathscr Z\big(F,\cO_F(1)\oplus\cO_F(2)\big),
\]
where $H$ is a hyperplane and $Q$ a quadric, and the intersection is
smooth.  This is the standard realization of an ordinary
Gushel--Mukai fourfold, which is a Fano fourfold of index two.

For the framework of Standing Assumption \ref{stand:setup}, set
\[
Y:=\mathscr Z\big(F,\cO_F(2)\big).
\]
After replacing the equation of the quadric, if necessary, by another
extension having the same restriction to $H$, we may assume that $Y$ is
smooth and still contains $X$ as the hyperplane section
\[
X=Y\cap H.
\]
Indeed, if $H=\{\ell=0\}$ and $q$ is one quadratic equation defining
$X$ on $H$, then the quadrics
\[
q+\ell m,
\qquad
m\in H^{0}\big(\PP^{9},\cO_{\PP^{9}}(1)\big),
\]
all restrict to the same equation on $H$.  Since $X$ is smooth, the
corresponding quadric section of $F$ is smooth along the base locus
$X$, while a general choice of $m$ is smooth away from $X$ by Bertini.

Adjunction gives
\[
-K_F=5H_F,
\qquad
-K_Y=3H_F|_Y,
\qquad
-K_X=2h,
\]
where
\[
H_F=c_1\big(\cO_F(1)\big),
\qquad
h=H_F|_X.
\]
Thus both $Y$ and $X$ are Fano.  The Lefschetz hyperplane theorem,
applied first to the smooth quadric section $Y\subset F$ and then to
the hyperplane section $X\subset Y$, gives the cohomological and Picard
comparisons required by Standing Assumption \ref{stand:setup}.  In
particular,
\[
\rH^{4}(F,\Q)\xrightarrow{\sim}\rH^{4}(Y,\Q).
\]
Hence condition \textup{(AC)} of Definition
\ref{def:ambient-complete} holds, and
\[
\rH_F^{\bullet}(X)
=
\rH_Y^{\bullet}(X).
\]
In particular, unlike the K\"uchle $(c5)$ example, the flag-ambient
quantum block computed by localization is the entire
monodromy-fixed block entering the spectral criteria.

\medskip

The quantum computation is obtained from the following input.

\begin{inputbox}[Input]
algebra  A4         keep  2
bundle   osum(O(1), O(2))
\end{inputbox}

\begin{outputbox}[Output]
\noindent
Matrix of $c_{1}(TX)\star$ in the reduced Schubert basis, with quantum
variable $y_{2}$:
{\setlength{\arraycolsep}{3.5pt}\small
\[
A(y_{2})
=
\begin{pmatrix}
0&12y_{2}&0&0&96y_{2}^{2}&0\\
2&0&12y_{2}&20y_{2}&0&48y_{2}^{2}\\
0&2&0&0&8y_{2}&0\\
0&2&0&0&16y_{2}&0\\
0&0&2&3&0&6y_{2}\\
0&0&0&0&4&0
\end{pmatrix}.
\]
}

\noindent
Grading operator:
\[
\mathrm{diag}\Big(\ell(w)-\frac12\dim X\Big)
=
\mathrm{diag}(-2,-1,0,0,1,2).
\]

\medskip

\noindent
Characteristic polynomial:
\[
\chi_{A(y_{2})}(\lambda)
=
\lambda^{2}
\big(
\lambda^{4}
-176y_{2}\lambda^{2}
-256y_{2}^{2}
\big).
\]

\medskip

At the algebraic specialization $y_{1}=1$, the four nonzero eigenvalues
are
\[
\pm\sqrt{88+40\sqrt5},
\qquad
\pm\mathbf i\sqrt{40\sqrt5-88},
\]
that is, numerically,
\[
\pm13.3207627,
\qquad
\pm1.2011324\,\mathbf i,
\]
together with the eigenvalue $0$ of algebraic multiplicity two.
\end{outputbox}

The kernel of $A^{\mathrm{alg}}(1)$ is two-dimensional.  Consequently,
at $\lambda=0$ one has
\[
m_a(0)=2,
\qquad
m_g(0)=2,
\qquad
m_a(0)-m_g(0)=0.
\]
Thus, in contrast with the cubic fourfold, the zero eigenvalue is
semisimple on the ambient block.

\medskip

A smooth Gushel--Mukai fourfold has Hodge diamond
\[
\begin{array}{ccccccccc}
&&&&1&&&&\\
&&&0&&0&&&\\
&&0&&1&&0&&\\
&0&&0&&0&&0&\\
0&&1&&22&&1&&0\\
&0&&0&&0&&0&\\
&&0&&1&&0&&\\
&&&0&&0&&&\\
&&&&1&&&&
\end{array}.
\]
Since condition \textup{(AC)} holds, the cohomology restricted from $Y$
coincides with that restricted from $F$.  Thus the Hodge diamond splits
into the monodromy-fixed ambient part, in \textcolor{blue}{blue}, and the
vanishing part, in \textcolor{red}{red}:
\begin{equation}
\setlength{\arraycolsep}{3pt}
\begin{array}{ccccccccc}
&&&& \textcolor{blue}{1} &&&&\\
&&&0&&0&&&\\
&&0&&\textcolor{blue}{1}&&0&&\\
&0&&0&&0&&0&\\
0&&\textcolor{red}{1}&&
\mathclap{\textcolor{blue}{2}+\textcolor{red}{20}}
&&\textcolor{red}{1}&&0\\
&0&&0&&0&&0&\\
&&0&&\textcolor{blue}{1}&&0&&\\
&&&0&&0&&&\\
&&&&\textcolor{blue}{1}&&&&
\end{array}
\quad=\quad
\begin{array}{ccccccccc}
&&&& \textcolor{blue}{1} &&&&\\
&&&0&&0&&&\\
&&0&&\textcolor{blue}{1}&&0&&\\
&0&&0&&0&&0&\\
0&&0&&\textcolor{blue}{2}&&0&&0\\
&0&&0&&0&&0&\\
&&0&&\textcolor{blue}{1}&&0&&\\
&&&0&&0&&&\\
&&&&\textcolor{blue}{1}&&&&
\end{array}
\quad+\quad
\begin{array}{ccccccccc}
&&&&0&&&&\\
&&&0&&0&&&\\
&&0&&0&&0&&\\
&0&&0&&0&&0&\\
0&&\textcolor{red}{1}&&\textcolor{red}{20}
&&\textcolor{red}{1}&&0\\
&0&&0&&0&&0&\\
&&0&&0&&0&&\\
&&&0&&0&&&\\
&&&&0&&&&
\end{array}.
\end{equation}

In particular,
\[
h^{3,1}(X)=1,
\qquad
h^{2,2}(X)=22,
\qquad
b_4(X)=24.
\]
Moreover,
\[
b_4(Y)=b_4(F)=2,
\]
and therefore
\[
\boxed{
b_4(X)_{\mathrm{van}}
=
b_4(X)-b_4(Y)
=
24-2
=
22.
}
\]
The vanishing Hodge structure consequently has Hodge numbers
\[
\big(
h^{3,1}_{\mathrm{van}},
h^{2,2}_{\mathrm{van}},
h^{1,3}_{\mathrm{van}}
\big)
=
(1,20,1).
\]

For a very general Gushel--Mukai fourfold,
Proposition \ref{prop:hodgegeneral} applies because
\[
h^{3,1}(Y)=0<1=h^{3,1}(X).
\]
Hence $X$ is Hodge general and
\[
\rH^{4}(X)^{\mathrm{Hdg}}
=
j^{*}\rH^{4}(Y,\Q)
=
\iota^{*}\rH^{4}(F,\Q),
\]
so that
\[
\rk\,\rH^{4}(X)^{\mathrm{Hdg}}=2.
\]

\medskip

\emph{The numerical irrationality criterion.}
Since condition \textup{(AC)} holds,
Corollary \ref{cor:numerical-combinatorial} applies.  The Grassmannian
$F=\mathrm{Gr}(2,5)$ has exactly two Schubert classes in codimension two,
namely
\[
\sigma_{2},
\qquad
\sigma_{1,1},
\]
and hence
\[
b_4(F)=2.
\]
The numerical inequality becomes
\[
b_4(F)\le h^{2,2}(X)-20,
\]
that is,
\[
2\le22-20=2.
\]
Equivalently,
\[
b_4(X)_{\mathrm{van}}
=
22
=
10+12h^{3,1}(X).
\]
Thus the numerical criterion proves that every Hodge-general member is
irrational.

\medskip

\emph{The spectral irrationality criterion.}
Again by \textup{(AC)},
\[
B^{\mathrm{alg}}(1)
=
A^{\mathrm{alg}}(1)
=
A(1).
\]
Every eigenvalue of this matrix has algebraic multiplicity at most two.  Therefore, Theorem \ref{thm:irrat-from-A} applies to a
Hodge-general Gushel--Mukai fourfold and gives its irrationality.

Thus, either the numerical or the spectral criterion recovers the result
of Benedetti--Manivel--Perrin:
\begin{theorem}[Benedetti--Manivel--Perrin,
\cite{Benedetti_Vladimiro_Manivel}]
A Hodge-general Gushel--Mukai fourfold is irrational.
In particular, a very general Gushel--Mukai fourfold is irrational.
\end{theorem}

\medskip

\emph{The structural criterion for rational special members.}
Theorem \ref{theo:rational_Fano_K3} does not require Hodge generality.
Moreover, the Jordan defects of $A^{\mathrm{alg}}(1)$ are all zero:
the four nonzero eigenvalues are simple and the two-dimensional
zero eigenspace is semisimple.  Hence, the Jordan-defect hypothesis of
Theorem \ref{theo:rational_Fano_K3} is satisfied.

Since a Gushel--Mukai fourfold has Fano index two,
Lemma \ref{lem:fano-index-vanishing} moreover gives
\[
\lambda_{\mathrm{van}}=0
\]
at the small points. Consequently, for every smooth ordinary Gushel--Mukai fourfold in the
family considered above, if $X$ were rational, then every weak
factorization
\[
X\dashrightarrow\PP^{4}
\]
would contain a smooth surface center whose minimal model is a K3
surface.  After choosing compatible evaluation maps as in
Theorem \ref{theo:rational_Fano_K3} and making the prescribed finite
extensions of coefficient fields, the cohomology of such a K3 surface
embeds into the relevant evaluated generalized eigenspace of $X$ in the
linear span of morphisms of rational Hodge structures. Benedetti--Manivel--Perrin obtain a stronger, case-specific conclusion
in \cite{Benedetti_Vladimiro_Manivel}. By deforming the small quantum
product inside the big quantum cohomology, they prove that if a
Gushel--Mukai fourfold is rational, then its rational primitive
cohomology is Hodge isomorphic, up to the appropriate Tate twist, to the
degree-two rational cohomology of a projective K3 surface. 
\end{example}

\medskip

\begin{example}[Quartic fourfolds]\label{ex:quartic4fold}

Let
\[
X=\mathscr Z\big(\PP^{5},\cO_{\PP^{5}}(4)\big)
\subset\PP^{5}
\]
be a smooth quartic fourfold.  By adjunction,
\[
-K_X=2h,
\qquad
h=c_1\big(\cO_X(1)\big),
\]
so $X$ is a Fano fourfold of index two.

If one wishes to place $X$ in the Lefschetz-pencil framework of
Standing Assumption \ref{stand:setup}, one may choose a smooth quartic
fivefold
\[
Y\subset\PP^{6}
\]
having $X$ as a hyperplane section.  Indeed, if
$X=\{f(x_0,\dots,x_5)=0\}$, then for a general cubic polynomial $g$ the
hypersurface
\[
Y=
\big\{
f(x_0,\dots,x_5)+x_6g(x_0,\dots,x_6)=0
\big\}
\subset\PP^{6}
\]
is smooth and satisfies
\[
X=Y\cap\{x_6=0\}.
\]
The Lefschetz hyperplane theorem gives
\[
\rH^{4}(\PP^{6},\Q)\xrightarrow{\sim}\rH^{4}(Y,\Q),
\]
so condition \textup{(AC)} holds.  Thus the flag-ambient and
monodromy-fixed parts coincide in this example.

\medskip

The quantum computation is obtained from the following input.

\begin{inputbox}[Input]
algebra  A5         keep  1
bundle   O(4)
\end{inputbox}

\begin{outputbox}[Output]
\noindent
Matrix of $c_{1}(TX)\star$ in the reduced Schubert basis
$1,h,h^2,h^3,h^4$, with Novikov variable $y_{1}$:
{\setlength{\arraycolsep}{3.5pt}\small
\[
A(y_1)=
\begin{pmatrix}
0 & 48y_1 & 0 & 5568y_1^2 & 0 \\
2 & 0 & 208y_1 & 0 & 5568y_1^2 \\
0 & 2 & 0 & 208y_1 & 0 \\
0 & 0 & 2 & 0 & 48y_1 \\
0 & 0 & 0 & 2 & 0
\end{pmatrix}.
\]
}

\noindent
Grading operator:
\[
\mathrm{diag}\Big(\ell(w)-\frac12\dim X\Big)
=
\mathrm{diag}(-2,-1,0,1,2).
\]

\medskip

\noindent
Characteristic polynomial:
\[
\chi_{A(y_1)}(\lambda)
=
\lambda^3\big(\lambda^2-1024y_1\big).
\]

\medskip

\noindent
Eigenvalues at the algebraic specialization $y_1=1$:
\[
\{-32,\ 32,\ 0,\ 0,\ 0\}.
\]
\end{outputbox}

The zero eigenvalue has algebraic multiplicity three and geometric
multiplicity one.  Hence its generalized eigenspace consists of a single
Jordan block of size three and
\[
m_a(0)=3,
\qquad
m_g(0)=1,
\qquad
m_a(0)-m_g(0)=2.
\]
The two nonzero eigenvalues are simple.

\medskip

The Hodge diamond of $X$ is
\[
\begin{array}{ccccccccc}
&&&&1&&&&\\
&&&0&&0&&&\\
&&0&&1&&0&&\\
&0&&0&&0&&0&\\
0&&21&&142&&21&&0\\
&0&&0&&0&&0&\\
&&0&&1&&0&&\\
&&&0&&0&&&\\
&&&&1&&&&
\end{array}.
\]
Since condition \textup{(AC)} holds, the part restricted from $Y$
coincides with the projective ambient part.  Thus the Hodge diamond splits
into the monodromy-fixed ambient part, in \textcolor{blue}{blue}, and the
vanishing part, in \textcolor{red}{red}:
\begin{equation}
\setlength{\arraycolsep}{3pt}
\begin{array}{ccccccccc}
&&&& \textcolor{blue}{1} &&&&\\
&&&0&&0&&&\\
&&0&&\textcolor{blue}{1}&&0&&\\
&0&&0&&0&&0&\\
0&&\textcolor{red}{21}&&
\mathclap{\textcolor{blue}{1}+\textcolor{red}{141}}
&&\textcolor{red}{21}&&0\\
&0&&0&&0&&0&\\
&&0&&\textcolor{blue}{1}&&0&&\\
&&&0&&0&&&\\
&&&&\textcolor{blue}{1}&&&&
\end{array}
\quad=\quad
\begin{array}{ccccccccc}
&&&& \textcolor{blue}{1} &&&&\\
&&&0&&0&&&\\
&&0&&\textcolor{blue}{1}&&0&&\\
&0&&0&&0&&0&\\
0&&0&&\textcolor{blue}{1}&&0&&0\\
&0&&0&&0&&0&\\
&&0&&\textcolor{blue}{1}&&0&&\\
&&&0&&0&&&\\
&&&&\textcolor{blue}{1}&&&&
\end{array}
\quad+\quad
\begin{array}{ccccccccc}
&&&&0&&&&\\
&&&0&&0&&&\\
&&0&&0&&0&&\\
&0&&0&&0&&0&\\
0&&\textcolor{red}{21}&&\textcolor{red}{141}
&&\textcolor{red}{21}&&0\\
&0&&0&&0&&0&\\
&&0&&0&&0&&\\
&&&0&&0&&&\\
&&&&0&&&&
\end{array}.
\end{equation}

Consequently,
\[
h^{3,1}(X)=21,
\qquad
h^{2,2}(X)=142,
\qquad
b_4(X)=184,
\]
and
\[
b_4(X)_{\mathrm{van}}
=
184-b_4(Y)
=
184-1
=
183.
\]
The vanishing Hodge structure has Hodge numbers
\[
\big(
h^{3,1}_{\mathrm{van}},
h^{2,2}_{\mathrm{van}},
h^{1,3}_{\mathrm{van}}
\big)
=
(21,141,21).
\]

For a very general quartic fourfold, Proposition
\ref{prop:hodgegeneral} applies, since
\[
h^{3,1}(Y)=0<21=h^{3,1}(X).
\]
Hence
\[
\rH^4(X)^{\mathrm{Hdg}}
=
j^*\rH^4(Y,\Q)
=
\Q h^2,
\qquad
\rk\,\rH^4(X)^{\mathrm{Hdg}}=1.
\]

Since $X$ has Fano index two, the scalar by which small quantum
multiplication by $c_1(TX)$ acts on the vanishing cohomology is zero.
Indeed, a scalar contribution preserving degree four would have to arise
from a curve class $\beta$ satisfying
\[
c_1(TX)\cdot\beta=1,
\]
which is impossible because $c_1(TX)=2h$.  Therefore
\[
c_1(TX)\star_{\mathrm{small}}
\big|_{\rH^4(X)_{\mathrm{van}}}
=
0.
\]

It follows that at a small point the full generalized zero eigenspace
contains both the $183$-dimensional vanishing cohomology and the
three-dimensional generalized zero eigenspace of the ambient matrix.
For a very general $X$, the former contains no Hodge classes, whereas the
latter consists of Hodge classes.  Hence, in the notation of
Definition \ref{def:atomintegers},
\[
\rho_{b,0}=3,
\qquad
\nu_{b,0}=21,
\qquad
\nu'_{b,0}=0,
\qquad
\gamma_{b,0}=2
\]
at the corresponding small points.

These values explain why the small quantum computation does not yield
either of the atomic obstructions used above: the condition
$\rho\le2$ needed to violate $\clubsuit$ fails, while the condition
$\gamma\le1$ needed to violate $\heartsuit$ fails as well.  Moreover,
the quartic fourfold is not of K3 type, since $h^{3,1}(X)=21$, so the
specialized fourfold criteria of Theorems
\ref{thm:irrat-from-A} and \ref{theo:rational_Fano_K3} are not applicable.

The numerical criterion of
Theorem \ref{thm:BFGMPcriterion} does not apply either, since
\[
b_4(X)_{\mathrm{van}}
=
183
<
10+12h^{3,1}(X)
=
262.
\]

Thus, the computation at the small point does not, by itself, provide an obstruction to rationality. However, the techniques developed in \cite{carazolli2026quantum} allow one to overcome this limitation and establish the irrationality of the very general quartic fourfold. We will discuss these techniques in subsequent examples.

Independently, Totaro proves in \cite{Totaro2015}*{Thm.~2.1} that a very
general quartic fourfold has non-universally-trivial $\mathsf{CH}_0$ and
is therefore not stably rational.  In particular, it is irrational.

\end{example}

\medskip

\begin{example}[Intersections of a quadric and a cubic in $\PP^{6}$]
\label{ex:23fourfold}

Let
\[
X
=
\mathscr Z\big(
\PP^{6},
\cO_{\PP^{6}}(2)\oplus\cO_{\PP^{6}}(3)
\big)
\]
be a smooth complete intersection of a quadric and a cubic.  Adjunction
gives
\[
-K_X=2h,
\qquad
h=c_1\big(\cO_X(1)\big),
\]
so $X$ is again a Fano fourfold of index two.

As in the preceding example, $X$ may be placed in the Lefschetz-pencil
framework by choosing a smooth complete intersection
\[
Y
=
\mathscr Z\big(
\PP^{7},
\cO_{\PP^{7}}(2)\oplus\cO_{\PP^{7}}(3)
\big)
\]
containing $X$ as a hyperplane section.  Explicitly, if
\[
X=\{q=c=0\}\subset\{x_7=0\}\cong\PP^6,
\]
one may replace the two equations in $\PP^7$ by
\[
q+x_7\ell,
\qquad
c+x_7g,
\]
with $\ell$ linear and $g$ quadratic and general.  Smoothness holds
along $X$ because $X$ is smooth, and away from the base locus for a
general choice by Bertini.

The fivefold $Y$ is Fano, with
\[
-K_Y=3H|_Y.
\]
Moreover, the Lefschetz theorem for smooth complete intersections gives
\[
\rH^4(\PP^7,\Q)\xrightarrow{\sim}\rH^4(Y,\Q).
\]
Thus condition \textup{(AC)} holds and the projective ambient block is
again the full monodromy-fixed block.

\medskip

The quantum computation is obtained from the following input.

\begin{inputbox}[Input]
algebra  A6         keep  1
bundle   osum(O(2), O(3))
\end{inputbox}

\begin{outputbox}[Output]
\noindent
Matrix of $c_{1}(TX)\star$ in the reduced Schubert basis
$1,h,h^2,h^3,h^4$, with Novikov variable $y_{1}$:
{\setlength{\arraycolsep}{3.5pt}\small
\[
A(y_1)
=
\begin{pmatrix}
0 & 24y_1 & 0 & 1152y_1^2 & 0 \\
2 & 0 & 84y_1 & 0 & 1152y_1^2 \\
0 & 2 & 0 & 84y_1 & 0 \\
0 & 0 & 2 & 0 & 24y_1 \\
0 & 0 & 0 & 2 & 0
\end{pmatrix}.
\]
}

\noindent
Grading operator:
\[
\mathrm{diag}\Big(\ell(w)-\frac12\dim X\Big)
=
\mathrm{diag}(-2,-1,0,1,2).
\]

\medskip

\noindent
Characteristic polynomial:
\[
\chi_{A(y_1)}(\lambda)
=
\lambda^3\big(\lambda^2-432y_1\big).
\]

\medskip

\noindent
Eigenvalues at the algebraic specialization $y_1=1$:
\[
\{-12\sqrt3,\ 12\sqrt3,\ 0,\ 0,\ 0\}.
\]
\end{outputbox}

Again, the zero eigenvalue has algebraic multiplicity three and geometric
multiplicity one:
\[
m_a(0)=3,
\qquad
m_g(0)=1,
\qquad
m_a(0)-m_g(0)=2.
\]
Thus its ambient generalized eigenspace is a single Jordan block of
size three, while the two nonzero eigenvalues are simple.

\medskip

The Hodge diamond of $X$ is
\[
\begin{array}{ccccccccc}
&&&&1&&&&\\
&&&0&&0&&&\\
&&0&&1&&0&&\\
&0&&0&&0&&0&\\
0&&8&&70&&8&&0\\
&0&&0&&0&&0&\\
&&0&&1&&0&&\\
&&&0&&0&&&\\
&&&&1&&&&
\end{array}.
\]
Since condition \textup{(AC)} holds, it decomposes into the
monodromy-fixed ambient part, in \textcolor{blue}{blue}, and the
vanishing part, in \textcolor{red}{red}:
\begin{equation}
\setlength{\arraycolsep}{3pt}
\begin{array}{ccccccccc}
&&&& \textcolor{blue}{1} &&&&\\
&&&0&&0&&&\\
&&0&&\textcolor{blue}{1}&&0&&\\
&0&&0&&0&&0&\\
0&&\textcolor{red}{8}&&
\mathclap{\textcolor{blue}{1}+\textcolor{red}{69}}
&&\textcolor{red}{8}&&0\\
&0&&0&&0&&0&\\
&&0&&\textcolor{blue}{1}&&0&&\\
&&&0&&0&&&\\
&&&&\textcolor{blue}{1}&&&&
\end{array}
\quad=\quad
\begin{array}{ccccccccc}
&&&& \textcolor{blue}{1} &&&&\\
&&&0&&0&&&\\
&&0&&\textcolor{blue}{1}&&0&&\\
&0&&0&&0&&0&\\
0&&0&&\textcolor{blue}{1}&&0&&0\\
&0&&0&&0&&0&\\
&&0&&\textcolor{blue}{1}&&0&&\\
&&&0&&0&&&\\
&&&&\textcolor{blue}{1}&&&&
\end{array}
\quad+\quad
\begin{array}{ccccccccc}
&&&&0&&&&\\
&&&0&&0&&&\\
&&0&&0&&0&&\\
&0&&0&&0&&0&\\
0&&\textcolor{red}{8}&&\textcolor{red}{69}
&&\textcolor{red}{8}&&0\\
&0&&0&&0&&0&\\
&&0&&0&&0&&\\
&&&0&&0&&&\\
&&&&0&&&&
\end{array}.
\end{equation}

Thus
\[
h^{3,1}(X)=8,
\qquad
h^{2,2}(X)=70,
\qquad
b_4(X)=86.
\]
Since
\[
b_4(Y)=1,
\]
the vanishing cohomology has dimension
\[
b_4(X)_{\mathrm{van}}
=
86-1
=
85,
\]
with Hodge numbers
\[
\big(
h^{3,1}_{\mathrm{van}},
h^{2,2}_{\mathrm{van}},
h^{1,3}_{\mathrm{van}}
\big)
=
(8,69,8).
\]

For a very general member, Proposition \ref{prop:hodgegeneral} applies
because
\[
h^{3,1}(Y)=0<8=h^{3,1}(X).
\]
Hence
\[
\rH^4(X)^{\mathrm{Hdg}}
=
j^*\rH^4(Y,\Q)
=
\Q h^2,
\qquad
\rk\,\rH^4(X)^{\mathrm{Hdg}}=1.
\]

Since $X$ has Fano index two, the same grading argument as above shows
that small quantum multiplication by $c_1(TX)$ acts by zero on the
vanishing cohomology:
\[
c_1(TX)\star_{\mathrm{small}}
\big|_{\rH^4(X)_{\mathrm{van}}}
=
0.
\]
Thus, at the corresponding small points, the full generalized zero
eigenspace contains the $85$-dimensional vanishing cohomology together
with the three-dimensional ambient generalized zero eigenspace.

For a very general member, this gives
\[
\rho_{b,0}=3,
\qquad
\nu_{b,0}=8,
\qquad
\nu'_{b,0}=0,
\qquad
\gamma_{b,0}=2.
\]
Thus the small-point data again produce neither of the atomic
obstructions used above: $\rho_{b,0}=3$ prevents a violation of
$\clubsuit$, while $\gamma_{b,0}=2$ prevents a violation of
$\heartsuit$.

Moreover, this fourfold is not of K3 type because
\[
h^{3,1}(X)=8,
\]
so Theorems \ref{thm:irrat-from-A} and
\ref{theo:rational_Fano_K3}, in their K3-type form, are not applicable.
The numerical criterion of Theorem \ref{thm:BFGMPcriterion} also does not
apply, since
\[
b_4(X)_{\mathrm{van}}
=
85
<
10+12h^{3,1}(X)
=
106.
\]

Independently, Nicaise--Ottem prove in
\cite{Nicaise2022}*{Thm.~7.1} that a very general complete intersection
of a quadric and a cubic in $\PP^6$ is not stably rational, and hence is
irrational.

\end{example}

\medskip

\begin{example}[Intersection of three quadrics in $\PP^{7}$]
\label{ex:three-quadrics}

Let
\[
X
=
\mathscr Z\big(
\PP^{7},
\cO_{\PP^{7}}(2)\oplus
\cO_{\PP^{7}}(2)\oplus
\cO_{\PP^{7}}(2)
\big)
\]
be a smooth complete intersection of three quadrics.  By adjunction,
\[
-K_X=2h,
\qquad
h=c_1\big(\cO_X(1)\big),
\]
so $X$ is a Fano fourfold of index two.

As in the preceding examples, $X$ may be placed in the Lefschetz-pencil
framework by choosing a smooth fivefold
\[
Y
=
\mathscr Z\big(
\PP^{8},
\cO_{\PP^{8}}(2)^{\oplus3}
\big)
\]
containing $X$ as a hyperplane section.  Indeed, after writing
$\PP^7=\{x_8=0\}\subset\PP^8$, the three defining equations of $X$
may be extended to general quadrics in $\PP^8$ with the prescribed
restrictions to $\PP^7$; for a general choice, their common zero locus is
smooth.  Adjunction gives
\[
-K_Y=3H|_Y,
\]
so $Y$ is Fano.

Moreover, the Lefschetz theorem gives
\[
\rH^4(\PP^8,\Q)\xrightarrow{\sim}\rH^4(Y,\Q),
\]
and hence condition \textup{(AC)} holds.  Thus in this example the
projective ambient part coincides with the full monodromy-fixed part.

\medskip

The quantum computation is obtained from the following input.

\begin{inputbox}[Input]
algebra  A7         keep  1
bundle   osum(O(2), O(2), O(2))
\end{inputbox}

\begin{outputbox}[Output]
\noindent
Matrix of $c_{1}(TX)\star$ in the reduced Schubert basis
$1,h,h^2,h^3,h^4$, with Novikov variable $y_{1}$:
{\setlength{\arraycolsep}{3.5pt}\small
\[
A(y_1)
=
\begin{pmatrix}
0 & 16y_1 & 0 & 448y_1^2 & 0 \\
2 & 0 & 48y_1 & 0 & 448y_1^2 \\
0 & 2 & 0 & 48y_1 & 0 \\
0 & 0 & 2 & 0 & 16y_1 \\
0 & 0 & 0 & 2 & 0
\end{pmatrix}.
\]
}

\noindent
Grading operator:
\[
\mathrm{diag}\Big(\ell(w)-\frac12\dim X\Big)
=
\mathrm{diag}(-2,-1,0,1,2).
\]

\medskip

\noindent
Characteristic polynomial:
\[
\chi_{A(y_1)}(\lambda)
=
\lambda^3\big(\lambda^2-256y_1\big).
\]

\medskip

\noindent
Eigenvalues at the algebraic specialization $y_1=1$:
\[
\{-16,\ 16,\ 0,\ 0,\ 0\}.
\]
\end{outputbox}

The zero eigenvalue has algebraic multiplicity three and geometric
multiplicity one.  Therefore
\[
m_a(0)=3,
\qquad
m_g(0)=1,
\qquad
m_a(0)-m_g(0)=2,
\]
so its ambient generalized eigenspace is a single Jordan block of size
three.  The eigenvalues $\pm16$ are simple.

\medskip

The Hodge diamond of $X$ is
\[
\begin{array}{ccccccccc}
&&&&1&&&&\\
&&&0&&0&&&\\
&&0&&1&&0&&\\
&0&&0&&0&&0&\\
0&&3&&38&&3&&0\\
&0&&0&&0&&0&\\
&&0&&1&&0&&\\
&&&0&&0&&&\\
&&&&1&&&&
\end{array}.
\]
Since condition \textup{(AC)} holds, the Hodge diamond decomposes into
the monodromy-fixed ambient part, in \textcolor{blue}{blue}, and the
vanishing part, in \textcolor{red}{red}:
\begin{equation}
\setlength{\arraycolsep}{3pt}
\begin{array}{ccccccccc}
&&&& \textcolor{blue}{1} &&&&\\
&&&0&&0&&&\\
&&0&&\textcolor{blue}{1}&&0&&\\
&0&&0&&0&&0&\\
0&&\textcolor{red}{3}&&
\mathclap{\textcolor{blue}{1}+\textcolor{red}{37}}
&&\textcolor{red}{3}&&0\\
&0&&0&&0&&0&\\
&&0&&\textcolor{blue}{1}&&0&&\\
&&&0&&0&&&\\
&&&&\textcolor{blue}{1}&&&&
\end{array}
\quad=\quad
\begin{array}{ccccccccc}
&&&& \textcolor{blue}{1} &&&&\\
&&&0&&0&&&\\
&&0&&\textcolor{blue}{1}&&0&&\\
&0&&0&&0&&0&\\
0&&0&&\textcolor{blue}{1}&&0&&0\\
&0&&0&&0&&0&\\
&&0&&\textcolor{blue}{1}&&0&&\\
&&&0&&0&&&\\
&&&&\textcolor{blue}{1}&&&&
\end{array}
\quad+\quad
\begin{array}{ccccccccc}
&&&&0&&&&\\
&&&0&&0&&&\\
&&0&&0&&0&&\\
&0&&0&&0&&0&\\
0&&\textcolor{red}{3}&&\textcolor{red}{37}
&&\textcolor{red}{3}&&0\\
&0&&0&&0&&0&\\
&&0&&0&&0&&\\
&&&0&&0&&&\\
&&&&0&&&&
\end{array}.
\end{equation}

Thus
\[
h^{3,1}(X)=3,
\qquad
h^{2,2}(X)=38,
\qquad
b_4(X)=44.
\]
Since
\[
b_4(Y)=1,
\]
we obtain
\[
\boxed{
b_4(X)_{\mathrm{van}}
=
b_4(X)-b_4(Y)
=
44-1
=
43.
}
\]
The vanishing Hodge structure has Hodge numbers
\[
\big(
h^{3,1}_{\mathrm{van}},
h^{2,2}_{\mathrm{van}},
h^{1,3}_{\mathrm{van}}
\big)
=
(3,37,3).
\]

For a very general member,
Proposition \ref{prop:hodgegeneral} applies because
\[
h^{3,1}(Y)=0<3=h^{3,1}(X).
\]
Hence
\[
\rH^4(X)^{\mathrm{Hdg}}
=
j^*\rH^4(Y,\Q)
=
\Q h^2,
\qquad
\rk\,\rH^4(X)^{\mathrm{Hdg}}=1.
\]

Since $X$ has Fano index two, the scalar action of small quantum
multiplication by $c_1(TX)$ on the vanishing cohomology is zero:
\[
c_1(TX)\star_{\mathrm{small}}
\big|_{\rH^4(X)_{\mathrm{van}}}
=
0.
\]
Thus the full generalized zero eigenspace at a small point contains the
$43$-dimensional vanishing cohomology together with the
three-dimensional ambient generalized zero eigenspace.

For a very general member, the corresponding small-point invariants are
therefore
\[
\rho_{b,0}=3,
\qquad
\nu_{b,0}=3,
\qquad
\nu'_{b,0}=0,
\qquad
\gamma_{b,0}=2.
\]
Consequently, the small quantum computation yields neither of the coarse
atomic obstructions considered above: the inequality
$\rho_{b,0}\le2$ needed to violate $\clubsuit$ fails, while
$\gamma_{b,0}\le1$ needed to violate $\heartsuit$ also fails.

Moreover,
\[
h^{3,1}(X)=3,
\]
so $X$ is not of K3 type and the specialized criteria of
Theorems \ref{thm:irrat-from-A} and
\ref{theo:rational_Fano_K3} do not apply.  The numerical criterion of
Theorem \ref{thm:BFGMPcriterion} does not apply either, since
\[
b_4(X)_{\mathrm{van}}
=
43
<
10+12h^{3,1}(X)
=
46.
\]

Thus, at the level of the coarse small-point data developed in this
section, the situation is analogous to the quartic and $(2,3)$ examples:
the generalized zero eigenspace is only a small-point candidate and may
split further at a generic point of the Hodge locus.  The displayed
matrix alone does not produce an irrationality obstruction.

\medskip

There is, however, additional structure in this example which is invisible
to the coarse invariants above.  Associated with a very general net of
quadrics is a double cover
\[
\pi: S\longrightarrow\PP^2
\]
branched along the discriminant octic, together with a nontrivial
$2$-torsion Brauer class
\[
\alpha\in\operatorname{Br}(S)[2].
\]
The surface $S$ is a smooth surface of general type.  Donagi--Pantev
study the corresponding twisted category and its integral numerical
$K$-lattice in \cite{Ron_Pantev}.  They prove that
\[
K_{\mathrm{num}}(S,\alpha)
\]
is not isomorphic, with its Euler pairing, to the numerical
$K$-lattice of any smooth projective surface.

In the language of the enhanced theory of Hodge atoms, the twisted
integral Hodge structure associated with $(S,\alpha)$ determines an
enhanced Hodge atom of $X$.  The lattice obstruction of Donagi--Pantev
shows that this enhanced atom cannot occur as the enhanced atom of a
smooth projective surface.  Applying the corresponding enhanced
nonrationality criterion yields a new proof of the irrationality of the
very general intersection of three quadrics in $\PP^7$:

\begin{theorem}[Donagi--Pantev, \cite{Ron_Pantev}]
A very general complete intersection of three quadrics in $\PP^{7}$ is
irrational.
\end{theorem}

\end{example}

\begin{example}[Fano K3-29, \cite{Bernardara2026}]
\label{ex:FanoK3-29}
    Let
    \[
    X = \mathscr{L}(\mathbb{P}^2\times \mathbb{P}^{4},\mathcal{O}_{\mathbb{P}^{2}\times \mathbb{P}^{4}}(1,1)\oplus \mathcal{O}_{\mathbb{P}^{2}\times \mathbb{P}^{4}}(1,2))
    \]
    a smooth section of $\mathbb{P}^2\times \mathbb{P}^{4}$. By adjunction one has
    \[
    -K_X = h_1 + 2h_2, \qquad h_1 = c_1(\mathcal{O}_{\mathbb{P}^2}(1)) \textrm{ and } h_{2} = c_{1}(\mathcal{O}_{\mathbb{P}^{4}}(1)) 
    \]
    so $X$ is a Fano fourfold with Fano index $1$. As in the examples above, $X$ can be seen as a hyperplane section of a smooth fivefold
    \[
    Y =  \mathscr{L}(\mathbb{P}^2\times \mathbb{P}^{4}, \mathcal{O}_{\mathbb{P}^{2}\times \mathbb{P}^{4}}(1,2))
    \]
    by adjunction the canonical bundle is
    \[
    -K_{Y} = 2h_1 + 3h_2
    \]
    then $Y$ is Fano. From Lefschetz theorem
    \[
    \rH^{4}(\mathbb{P}^{2} \times \mathbb{P}^{4},\Q) \overset{\sim}{\longrightarrow} \rH^{4}(Y,\Q)
    \]
    Thus, condition (AC) holds and the ambient part from $\mathbb{P}^2 \times \mathbb{P}^{4}$ is the full monodromy-fixed block.
    The quantum computation is obtained from the following input
    \begin{inputbox}[Input]
        algebra  A2xA4         keep  1,3
        bundle   osum(O(1,2), O(1,1))
    \end{inputbox}

    \begin{outputbox}[Output]
        \noindent
        Matrix of $c_{1}(TX)\star$ in the reduced Schubert basis, with Novikov variable $y_{1},y_3$:
        {\setlength{\arraycolsep}{3.5pt}\small
        \[
        A(y_1,y_3)
        =
        \begin{pmatrix}
            0 & 0 & 4y_3 & 30y_1y_3 & 12y_1y_3 & 30y_1y_3 &
            56y_1^2y_3 & 32y_1^2y_3 & 15y_1^3y_3+60y_1y_3^2 \\[2mm]
            1 & -y_1 & 0 & 4y_3 & 0 & 10y_3 &
            33y_1y_3 & 18y_1y_3 & 12y_1^2y_3 \\[2mm]
            2 & 3y_1 & 0 & 0 & 0 & 4y_3 &
            54y_1y_3 & 24y_1y_3 & 48y_1^2y_3 \\[2mm]
            0 & 2 & 1 & -y_1 & 0 & 0 &
            4y_3 & 0 & 18y_1y_3 \\[2mm]
            0 & 1 & 0 & 0 & -y_1 & 0 &
            10y_3 & 4y_3 & 9y_1y_3 \\[2mm]
            0 & 0 & 2 & 3y_1 & 2y_1 & 0 &
            0 & 0 & 18y_1y_3 \\[2mm]
            0 & 0 & 0 & 2 & 0 & 4 &
            \frac{7}{2}y_1 & 3y_1 & 0 \\[2mm]
            0 & 0 & 0 & 1 & 2 & -\frac{7}{2} &
            -\frac{21}{4}y_1 & -\frac{9}{2}y_1 & 6y_3 \\[2mm]
            0 & 0 & 0 & 0 & 0 & 0 &
            \frac{8}{3} & \frac{4}{3} & 0
        \end{pmatrix}.
        \]
        }

        \noindent
        Grading operator:
        \[
        \mathrm{diag}\Big(\ell(w)-\frac12\dim X\Big)
        =
        \mathrm{diag}(-2,-1,-1,0,0,0,1,1,2).
        \]

        \medskip

        \noindent
        Characteristic polynomial:
    \[
\lambda^9
+ 4\lambda^8 y_1
+ \lambda^7\left(6y_1^2 - 48y_3\right)
+ \lambda^6\left(4y_1^3 - 1616y_1y_3\right)
+ \lambda^5\left(y_1^4 - 10360y_1^2y_3 + 768y_3^2\right)
\]
\[
\qquad
+ \lambda^4\left(-29364y_1^3y_3 - 2048y_1y_3^2\right)
+ \lambda^3\left(-44497y_1^4y_3 - 16000y_1^2y_3^2 - 4096y_3^3\right)
\]
\[
\qquad
+ \lambda^2\left(-37675y_1^5y_3 - 28032y_1^3y_3^2 - 12288y_1y_3^3\right)
\]
\[
\qquad
+ \lambda\left(-16875y_1^6y_3 - 20096y_1^4y_3^2 - 12288y_1^2y_3^3\right)
\]
\[
\qquad
- 3125y_1^7y_3
- 5248y_1^5y_3^2
- 4096y_1^3y_3^3.
\]
        \medskip

        \noindent
        Eigenvalues at the algebraic specialization $y_1=y_3=1$:
        \[
        \{ 13.3572, \, -5.40029 \pm 8.20735\mathrm{i}, \, -0.926637 \pm 2.19525\mathrm{i}, \, -1.70334, \, -1, \, -1, \, -1 \}
        \]
    \end{outputbox}
    The $-1$ eigenvalue has algebraic multiplicity three and geometric
    multiplicity two. Therefore
    \[
    m_a(-1)=3,
    \qquad
    m_g(-1)=2,
    \qquad
    m_a(-1)-m_g(-1)=1,
    \]
    so its ambient generalized eigenspace are one Jordan block of size
    two and a Jordan block of size one.  The other eigenvalues are simple.
    The Hodge diamond of $X$ is
    \[
    \begin{array}{ccccccccc}
    &&&&1&&&&\\
    &&&0&&0&&&\\
    &&0&&2&&0&&\\
    &0&&0&&0&&0&\\
    0&&1&&22&&1&&0\\
    &0&&0&&0&&0&\\
    &&0&&2&&0&&\\
    &&&0&&0&&&\\   
    &&&&1&&&&
    \end{array}.
    \]
    Since condition \textup{(AC)} holds, the Hodge diamond decomposes into
    the monodromy-fixed ambient part, in \textcolor{blue}{blue}, and the
    vanishing part, in \textcolor{red}{red}:
    \begin{equation}
    \setlength{\arraycolsep}{3pt}
    \begin{array}{ccccccccc}
    &&&& \textcolor{blue}{1} &&&&\\
    &&&0&&0&&&\\
    &&0&&\textcolor{blue}{2}&&0&&\\
    &0&&0&&0&&0&\\
    0&&\textcolor{red}{1}&&
    \mathclap{\textcolor{blue}{3}+\textcolor{red}{19}}
    &&\textcolor{red}{1}&&0\\
    &0&&0&&0&&0&\\
    &&0&&\textcolor{blue}{2}&&0&&\\
    &&&0&&0&&&\\
    &&&&\textcolor{blue}{1}&&&&
    \end{array}
    \quad=\quad
    \begin{array}{ccccccccc}
    &&&& \textcolor{blue}{1} &&&&\\
    &&&0&&0&&&\\
    &&0&&\textcolor{blue}{2}&&0&&\\
    &0&&0&&0&&0&\\
    0&&0&&\textcolor{blue}{3}&&0&&0\\
    &0&&0&&0&&0&\\
    &&0&&\textcolor{blue}{2}&&0&&\\
    &&&0&&0&&&\\
    &&&&\textcolor{blue}{1}&&&&
    \end{array}
    \quad+\quad
    \begin{array}{ccccccccc}
    &&&&0&&&&\\
    &&&0&&0&&&\\
    &&0&&0&&0&&\\
    &0&&0&&0&&0&\\
    0&&\textcolor{red}{1}&&\textcolor{red}{19}
    &&\textcolor{red}{1}&&0\\
    &0&&0&&0&&0&\\
    &&0&&0&&0&&\\
    &&&0&&0&&&\\
    &&&&0&&&&
    \end{array}.
    \end{equation}
    It is known that $X$ is rational,~\cite{Bernardara2026}*{K3-29}.
    
    \emph{The structural criterion.} Knowing that the Jordan defect is one for every eigenvalue, from Theorem~\ref{theo:rational_Fano_K3} we have that every weak factorization
    \[
    X\dashrightarrow\PP^{4}
    \]
    contains a smooth surface center whose minimal model is a K3
    surface. This coincides with the realization of $X$ as the blow up along a K3 surface of degree $8$ blown up at one point.
\end{example}

\begin{example}[A hyperplane and a quadric in $\operatorname{LG}(3,6)$]
\label{ex:E_V5}

Let
\[
F=\operatorname{LG}(3,6)
\subset\mathbb P^{13}
\]
be the Lagrangian Grassmannian in its minimal
homogeneous embedding. It parametrizes
three-dimensional Lagrangian subspaces of a
six-dimensional symplectic vector space.

Let $h=c_1(\mathcal O_F(1))$. We consider the smooth complete intersection
\[
X=F\cap H_0\cap Q,
\]
where $H_0$ is a general hyperplane and $Q$ is
a general quadric hypersurface. Equivalently,
\[
X=
\mathscr Z\left(
F,\mathcal O_F(1)\oplus\mathcal O_F(2)
\right).
\]

The Lagrangian Grassmannian has
\[
\dim F=6,
\qquad
-K_F=4h,
\]
By adjunction,
\[
\dim X=4,
\qquad
-K_X=h|_X.
\]
Thus $X$ is a Fano fourfold of index one and
Picard rank one.

\medskip

For the birational applications of this section,
we regard $X$ as a hyperplane section of the
smooth fivefold
\[
Y=F\cap Q.
\]
Thus, in the notation of Standing Assumption
\ref{stand:setup}, we take
\[
F=\operatorname{LG}(3,6),
\qquad
\mathcal E=\mathcal O_F(2),
\]
\[
Y=\mathscr Z(F,\mathcal E),
\qquad
X=\mathscr Z\big(
F,\mathcal E\oplus\mathcal O_F(1)
\big).
\]

The standing assumptions hold with $n=4$ and by the Sommese's theorem,
\[
\rH^4(F,\mathbb Q)
\xrightarrow{\sim}
\rH^4(Y,\mathbb Q).
\]
Since
\[
\rH^4(F,\mathbb Q)=\mathbb Q h^2,
\]
condition \textup{(AC)} of Definition
\ref{def:ambient-complete} holds.

Consequently,
\[
\rH^*_{\mathrm{amb}}(X)=\rH_F^*(X)=\rH_Y^*(X)=
\operatorname{Span}_{\mathbb Q}
\{1,h,h^2,h^3,h^4\},
\]
and by Lemma \ref{lem:hdg-is-amb} we get that the very general one satisfies
\[
 \rH^*(X)^{\mathrm{Hdg}}=\rH^*_{\mathrm{amb}}(X).
\]

%In particular, the flag-ambient block computed
%by localization is the full monodromy-fixed
%block appearing in Theorems
%\ref{thm:irrat-from-A} and
%\ref{theo:rational_Fano_K3}.

\medskip

The small quantum multiplication is
computed from the following input.

\begin{inputbox}[Input]
algebra  C3         keep  3
bundle   osum(O(1),O(2))
\end{inputbox}

\begin{outputbox}[Output]

\noindent
Matrix of $c_1(TX)\star$ in this basis,
with Novikov variable $y_1$:

{\setlength{\arraycolsep}{3pt}\small
\[
A(y_1)=
\begin{pmatrix}
0&
352y_1^2&
9600y_1^3&
92096y_1^4&
1346560y_1^5\\[3pt]
1&
16y_1&
624y_1^2&
6144y_1^3&
92096y_1^4\\[3pt]
0&
2&
24y_1&
312y_1^2&
4800y_1^3\\[3pt]
0&
0&
4&
16y_1&
352y_1^2\\[3pt]
0&
0&
0&
1&
0
\end{pmatrix}.
\]
}

\noindent
Grading operator:
\[
\mathrm{diag}\Big(\ell(w)-\frac12\dim X\Big)=
\operatorname{diag}
\left(-2,-1,0,1,2\right).
\]

\medskip

\noindent

At the algebraic specialization $y_1=1$ its characteristic polynomial factors as
\[
\chi_{A(1)}(\lambda)
=
(\lambda+8)^3
(\lambda^2-80\lambda-448).
\]

\end{outputbox}

The eigenvalues of $A(1)$ are therefore
\[
-8,\qquad
40-32\sqrt2,\qquad
40+32\sqrt2,
\]
with respective algebraic multiplicities $3,1,1$.
And the Hodge diamond of $X$ is

\[
\begin{array}{ccccccccc}
&&&&1&&&&\\
&&&0&&0&&&\\
&&0&&1&&0&&\\
&0&&0&&0&&0&\\
0&&14&&100&&14&&0\\
&0&&0&&0&&0&\\
&&0&&1&&0&&\\
&&&0&&0&&&\\
&&&&1&&&&
\end{array}.
\]
  Since condition \textup{(AC)} holds, the Hodge diamond decomposes into
    the monodromy-fixed ambient part, in \textcolor{blue}{blue}, and the
    vanishing part, in \textcolor{red}{red}:
    \begin{equation}
    \setlength{\arraycolsep}{3pt}
    \begin{array}{ccccccccc}
    &&&& \textcolor{blue}{1} &&&&\\
    &&&0&&0&&&\\
    &&0&&\textcolor{blue}{1}&&0&&\\
    &0&&0&&0&&0&\\
    0&&\textcolor{red}{14}&&
    \mathclap{\textcolor{blue}{1}+\textcolor{red}{99}}
    &&\textcolor{red}{14}&&0\\
    &0&&0&&0&&0&\\
    &&0&&\textcolor{blue}{1}&&0&&\\
    &&&0&&0&&&\\
    &&&&\textcolor{blue}{1}&&&&
    \end{array}
    \quad=\quad
    \begin{array}{ccccccccc}
    &&&& \textcolor{blue}{1} &&&&\\
    &&&0&&0&&&\\
    &&0&&\textcolor{blue}{1}&&0&&\\
    &0&&0&&0&&0&\\
    0&&0&&\textcolor{blue}{1}&&0&&0\\
    &0&&0&&0&&0&\\
    &&0&&\textcolor{blue}{1}&&0&&\\
    &&&0&&0&&&\\
    &&&&\textcolor{blue}{1}&&&&
    \end{array}
    \quad+\quad
    \begin{array}{ccccccccc}
    &&&&0&&&&\\
    &&&0&&0&&&\\
    &&0&&0&&0&&\\
    &0&&0&&0&&0&\\
    0&&\textcolor{red}{14}&&\textcolor{red}{99}
    &&\textcolor{red}{14}&&0\\
    &0&&0&&0&&0&\\
    &&0&&0&&0&&\\
    &&&0&&0&&&\\
    &&&&0&&&&
    \end{array}.
    \end{equation}

The rationality criteria developed in the preceding sections
do not, by themselves, provide an obstruction to the
rationality of $X$. In particular, the presence of an
eigenvalue of algebraic multiplicity three does not
immediately imply irrationality. However, the computational
machinery developed for the developed software can be extended to
calculate the quantum-connection invariants introduced in
\cite{carazolli2026quantum}.

We first compute the cyclic determinant
\[
\Delta(1):=
\det\bigl(
\mathbf{1},A(1)\mathbf{1},
A(1)^2\mathbf{1},\ldots,A(1)^4\mathbf{1}
\bigr)
=128\neq 0.
\]
Thus, $\mathbf{1}$ is a cyclic vector for $A(1)$.
Since, for a very general $X$, the ambient cohomology
coincides with the Hodge-invariant cohomology,
Lemma~3.9 of \cite{carazolli2026quantum}
implies that the corresponding small-locus point is
generic in the Hodge-fixed A-model parameter space.

We next compute the regularized exponents of the quantum connection associated with each eigenvalue.
The resulting data are summarized in the following table.

\begin{center}
\small
\renewcommand{\arraystretch}{1.35}
\begin{tabular}{c|c|c}
$\lambda$ of $A(1)$
& $m$
& Exponents mod $\mathbb Z$ \\
\hline
$-8$ &  3
& $\{0,(\frac12)^{\oplus2}\}$ \\
$40-32\sqrt2$ 
& 1 & $\{0\}$ \\
$40+32\sqrt2$ 
& 1 & $\{0\}$ 
\end{tabular}
\end{center}

In particular, the cyclic generic Hodge factor
corresponding to the triple eigenvalue $\lambda=-8$ has
rank and regularized exponent multiset
\[
m=3,\qquad
F=\left\{
\overline{0},
\overline{\frac12},
\overline{\frac12}
\right\}.
\]
Here the bar denotes the class modulo $\mathbb Z$.

Theorem $4.5$ of \cite{carazolli2026quantum} implies that, for a rational smooth
projective fourfold, every cyclic generic Hodge
factor satisfying these hypotheses must have
one of the following pairs of invariants:
\[
(m,F)\in\left\{
\left(1,\left\{\overline0 \right\}\right),
\left(2,\left\{ \overline{\frac12},\overline{\frac12}\right\}\right),
\left(3,\left\{\overline0,\overline0,\overline0 \right\}\right)
\right\}.
\]
Our rank-three factor has exponent multiset
$\left\{\overline{0},
\overline{\frac12},
\overline{\frac12} \right\}$,
which is not among the permitted possibilities.
We therefore obtain the following result.

\begin{theorem}
\label{thm:irrational_E_V5}
A very general smooth complete intersection of
a hyperplane and a quadric in
$\operatorname{LG}(3,6)$ is irrational.
\end{theorem}

\end{example}

\begin{example}[A quadric section of the adjoint $G_2$ variety]
\label{ex:E_V9}
Let
\[
F=G_2/P_2\subset\mathbb P^{13}
\]
be the five-dimensional adjoint $G_2$ variety, embedded as the closed
$G_2$-orbit in $\mathbb P(\mathfrak g_2)$. Let
$h=c_1(\mathcal O_F(1))$ and 
\[
X=F\cap Q
=\mathscr Z\bigl(F,\mathcal O_F(2)\bigr),
\]
where $Q$ is a general quadric in $\mathbb P^{13}$.
The variety $F$ has
\[
\dim F=5,\qquad -K_F=3h.
\]
By adjunction,
\[
\dim X=4,\qquad -K_X=h|_X.
\]
Thus $X$ is a Fano fourfold of index one and Picard rank one.

\begin{proposition}
\label{prop:G2-hodge-ambient}
Let $X \in |\mathcal{O}_F(2)|$ be very general, and let $i \colon X \hookrightarrow F$
denote the inclusion. Then every rational Hodge class on $X$ is ambient.
More precisely, denoting by $\operatorname{Hdg}^*(X)$ the ring of rational Hodge
classes on $X$ and writing $h$ for $i^*h$, we have
\[
\rH_{\mathrm{amb}}^*(X) = \rH_F^*(X)
= \operatorname{Span}_{\mathbb{Q}}\{1, h, h^2, h^3, h^4\}.
\]
\end{proposition}
\begin{proof}
Set $L = \mathcal{O}_F(2)$. Since $\dim F = 5$ and $L$ is ample,
Sommese's theorem and Poincar\'e duality imply that
$i^* \colon \rH^k(F,\mathbb{Q}) \to \rH^k(X,\mathbb{Q})$ is an isomorphism
for all $k \neq 4$ and is injective for $k = 4$. In particular, all cohomology
classes of $X$ outside the middle degree are restrictions of classes on $F$,
and $\rH^{\mathrm{odd}}(X,\mathbb{Q}) = 0$.

It therefore suffices to prove that in degree four,
\[
\rH_{\mathrm{amb}}^4(X) = \rH_F^4(X) = \mathbb{Q}h^2.
\]
Let $\rH^4(X,\mathbb{Q})_{\mathrm{van}}$ be the vanishing cohomology. We have an orthogonal decomposition of rational Hodge structures
\[
\rH^4(X,\mathbb{Q}) = \mathbb{Q}h^2 \oplus \rH^4(X,\mathbb{Q})_{\mathrm{van}}.
\]
To show that there are no nonzero Hodge classes in $\rH^4(X,\mathbb{Q})_{\mathrm{van}}$
for $X$ very general, we apply the generalized Noether--Lefschetz theorem of
Voisin~\cite[Theorem~4.17]{VoisinHodgeLoci}.

The line bundle $\mathcal{O}_F(2)$ is very ample, and $|\mathcal{O}_F(2)|$
contains a divisor with a unique ordinary double point. Indeed, fixing $p \in F$
with ideal sheaf $\mathcal{I}_p \subset \mathcal{O}_F$, the linear subsystem
$|\mathcal{O}_F(2) \otimes \mathcal{I}_p^2|$ is base-point-free on $F \setminus \{p\}$
because $\mathcal{O}_F(1)$ is very ample (for any $q \in F \setminus \{p\}$, a
hyperplane section $H \in |\mathcal{O}_F(1)|$ passing through $p$ and missing $q$
yields $2H \in |\mathcal{O}_F(2) \otimes \mathcal{I}_p^2|$ missing $q$). By Bertini's
theorem, a general member of $|\mathcal{O}_F(2) \otimes \mathcal{I}_p^2|$ is smooth
away from $p$. Moreover, products of sections in
$\rH^0(F, \mathcal{O}_F(1) \otimes \mathcal{I}_p)$ surject onto
$\operatorname{Sym}^2(T_p^*F) \cong \mathcal{I}_{p,p}^2 / \mathcal{I}_{p,p}^3$, so
the locus of sections in $\rH^0(F, \mathcal{O}_F(2) \otimes \mathcal{I}_p^2)$ with a
nondegenerate quadratic part at $p$ is dense and Zariski open. A general divisor in
$|\mathcal{O}_F(2) \otimes \mathcal{I}_p^2|$ therefore has an ordinary double point
at $p$ as its unique singularity.

Finally, the Hodge diamond of $X$ gives
\[
h^{1,3}(X) = 14.
\]
Since $\mathbb{Q}h^2$ is of Hodge type $(2,2)$, we obtain
\[
h^{1,3}(X)_{\mathrm{van}} = h^{1,3}(X) = 14 > 0.
\]
Thus the hypotheses of~\cite[Theorem~4.17]{VoisinHodgeLoci} are satisfied. It follows
that, for a very general $X \in |\mathcal{O}_F(2)|$, the vanishing cohomology
$\rH^4(X,\mathbb{Q})_{\mathrm{van}}$ contains no nonzero rational Hodge classes.
\end{proof}

\medskip

\medskip
The small quantum multiplication is computed using
\begin{inputbox}[Input]
algebra  G2         keep  2
bundle   O(2)
\end{inputbox}

\begin{outputbox}[Output]
\noindent
Matrix of $c_1(TX)\star$ with Novikov variable $y_1$:
{\setlength{\arraycolsep}{3pt}\small
\[
A(y_1)=
\begin{pmatrix}
0&252y_1^2&4032y_1^3&68256y_1^4&285120y_1^5\\[3pt]
1&12y_1&312y_1^2&5328y_1^3&22752y_1^4\\[3pt]
0&3&18y_1&468y_1^2&2016y_1^3\\[3pt]
0&0&2&12y_1&84y_1^2\\[3pt]
0&0&0&3&0
\end{pmatrix}.
\]
}
\noindent
Grading operator:
\[
\mathrm{diag}\Big(\ell(w)-\frac12\dim X\Big)=\operatorname{diag}(-2,-1,0,1,2).
\]
At $y_1=1$ the characteristic polynomial is
\[
\chi_{A(1)}(\lambda)
=(\lambda+6)^3(\lambda^2-60\lambda-828).
\]
\end{outputbox}

The eigenvalues are therefore
\[
-6,\qquad 30-24\sqrt3,\qquad 30+24\sqrt3,
\]
with algebraic multiplicities $3,1,1$, respectively.
The Hodge diamond, including the primitive cohomology, is
\[
\begin{array}{ccccccccc}
&&&&1&&&&\\
&&&0&&0&&&\\
&&0&&1&&0&&\\
&0&&0&&0&&0&\\
0&&14&&100&&14&&0\\
&0&&0&&0&&0&\\
&&0&&1&&0&&\\
&&&0&&0&&&\\
&&&&1&&&&
\end{array}.
\]
Condition \textup{(AC)} identifies the ambient contribution
with the monodromy-fixed part. The Hodge diamond splits into
an ambient part, in \textcolor{blue}{blue}, and a vanishing part,
in \textcolor{red}{red}:
\begin{equation}
\setlength{\arraycolsep}{3pt}
\begin{array}{ccccccccc}
&&&&\textcolor{blue}{1}&&&&\\
&&&0&&0&&&\\
&&0&&\textcolor{blue}{1}&&0&&\\
&0&&0&&0&&0&\\
0&&\textcolor{red}{14}&&
\mathclap{\textcolor{blue}{1}+\textcolor{red}{99}}
&&\textcolor{red}{14}&&0\\
&0&&0&&0&&0&\\
&&0&&\textcolor{blue}{1}&&0&&\\
&&&0&&0&&&\\
&&&&\textcolor{blue}{1}&&&&
\end{array}
\quad=\quad
\begin{array}{ccccccccc}
&&&&\textcolor{blue}{1}&&&&\\
&&&0&&0&&&\\
&&0&&\textcolor{blue}{1}&&0&&\\
&0&&0&&0&&0&\\
0&&0&&\textcolor{blue}{1}&&0&&0\\
&0&&0&&0&&0&\\
&&0&&\textcolor{blue}{1}&&0&&\\
&&&0&&0&&&\\
&&&&\textcolor{blue}{1}&&&&
\end{array}
\quad+\quad
\begin{array}{ccccccccc}
&&&&0&&&&\\
&&&0&&0&&&\\
&&0&&0&&0&&\\
&0&&0&&0&&0&\\
0&&\textcolor{red}{14}&&\textcolor{red}{99}&&\textcolor{red}{14}&&0\\
&0&&0&&0&&0&\\
&&0&&0&&0&&\\
&&&0&&0&&&\\
&&&&0&&&&
\end{array}.
\end{equation}

As in Example~\ref{ex:E_V5}, the multiplicity-three eigenvalue
alone does not obstruct rationality. We therefore compute the
regularized exponents of the quantum connection using
\cite{carazolli2026quantum}.
First, the cyclic determinant at $y_1=1$ is
\[
\Delta(1):=\det\bigl(\mathbf 1,A(1)\mathbf 1,
A(1)^2\mathbf 1,A(1)^3\mathbf 1,A(1)^4\mathbf 1\bigr)
=324\neq0.
\]
Thus $\mathbf 1$ is a cyclic vector. For a very general $X$,
the Hodge-fixed cohomology agrees with the ambient cohomology;
Lemma~3.9 of \cite{carazolli2026quantum} then shows that this
small-locus point represents a generic Hodge factor.

The connection exponents are
\begin{center}
\small
\renewcommand{\arraystretch}{1.35}
\begin{tabular}{c|c|c}
$\lambda$ of $A(1)$ & $m$ & Exponents mod $\mathbb Z$\\
\hline
$-6$ & $3$ & $\{0,(\frac12)^{\oplus2}\}$\\
$30-24\sqrt3$ & $1$ & $\{0\}$\\
$30+24\sqrt3$ & $1$ & $\{0\}$
\end{tabular}
\end{center}
In particular, the triple eigenvalue determines a cyclic generic
Hodge factor with invariants
\[
m=3,\qquad
F=\left\{\overline0,\overline{\frac12},
\overline{\frac12}\right\}.
\]
For a rational smooth projective fourfold,
Theorem~4.5 of \cite{carazolli2026quantum} permits only
\[
(m,F)\in\left\{
\left(1,\left\{\overline0 \right\}\right),
\left(2,\left\{ \overline{\frac12},\overline{\frac12}\right\}\right),
\left(3,\left\{\overline0,\overline0,\overline0 \right\}\right)
\right\}.
\]
The computed pair is not among these possibilities. Hence:
\begin{theorem}
\label{thm:irrational_E_V9}
A very general smooth quadric section of the adjoint
$G_2$ variety $G_2/P_2$ is irrational.
\end{theorem}
\end{example}

\begin{example}[Five hyperplanes and a quadric in the spinor tenfold]
\label{ex:E_V7}
Let
\[
F=\operatorname{OGr}^{+}(5,10)\subset\mathbb P^{15}
\]
be the identity connected component of the orthogonal Grassmannian of
maximal isotropic five-dimensional subspaces of a ten-dimensional
quadratic vector space, this is called the spinor tenfold. Let
$h=c_1(\mathcal O_F(1))$.
The geometry of the spinor tenfold and its linear sections has been
extensively studied; see, for instance,
\cite{kuznetsov2017linear,liu2025linear}.
A particularly relevant result, established by Kuznetsov
in \cite{kuznetsov2017linear}, is that every smooth linear section
\[
Y_k=F\cap H_1\cap\cdots\cap H_k,
\qquad 1\leq k\leq 5,
\]
is rational.

In this example, we consider a smooth complete intersection obtained
by imposing five general hyperplane equations and one general
quadric equation:
\[
X=F\cap H_1\cap\cdots\cap H_5\cap Q,
\]
where $H_i\subset\mathbb P^{15}$ are hyperplanes and
$Q\subset\mathbb P^{15}$ is a quadric hypersurface.
Equivalently, $X$ is the zero locus of a general section of
$\mathcal O_F(1)^{\oplus 5}\oplus\mathcal O_F(2)$, namely,
\[
X=\mathscr Z\left(F,
\mathcal O_F(1)^{\oplus 5}\oplus\mathcal O_F(2)\right).
\]
The spinor tenfold satisfies
\[
\dim F=10,\qquad -K_F=8h
\]
By adjunction,
\[
\dim X=4,\qquad -K_X=h|_X
\]
Thus $X$ is a Fano fourfold of index one and Picard rank one.

\medskip
For the hyperplane-section setup of Standing Assumption
\ref{stand:setup}, distinguish one of these hyperplanes, say
$H_5$, and set
\[
\begin{aligned}
F&=\operatorname{OGr}^{+}(5,10),\\
\mathcal E&=\mathcal O_F(1)^{\oplus4}\oplus\mathcal O_F(2),\\
Y&=\mathscr Z(F,\mathcal E)
=F\cap H_1\cap\cdots\cap H_4\cap Q,\\
X&=\mathscr Z\bigl(F,\mathcal E\oplus\mathcal O_F(1)\bigr)
=Y\cap H_5.
\end{aligned}
\]
The variety $Y$ is a smooth fivefold. Repeated application of
the Lefschetz hyperplane theorem gives
\[
\rH^4(F,\mathbb Q)\xrightarrow{\sim}\rH^4(Y,\mathbb Q).
\]
Since $H^4(F,\mathbb Q)=\mathbb Q h^2$, condition
\textup{(AC)} of Definition~\ref{def:ambient-complete}
holds. In particular,
\[
\rH_{\mathrm{amb}}^*(X)
=\rH_F^*(X)=\rH_Y^*(X)=\operatorname{Span}_{\mathbb Q}\{1,h,h^2,h^3,h^4\},
\]
and by Lemma \ref{lem:hdg-is-amb} we get that the very general one satisfies
\[
 \rH^*(X)^{\mathrm{Hdg}}=\rH^*_{\mathrm{amb}}(X).
\]

\medskip
The small quantum multiplication is computed from
\begin{inputbox}[Input]
algebra  D5         keep  4
bundle   osum(O(1),O(1),O(1),O(1),O(1),O(2))
\end{inputbox}

\begin{outputbox}[Output]
\noindent
Matrix of $c_1(TX)\star$ with Novikov variable $y_1$:
{\setlength{\arraycolsep}{2.5pt}\small
\[
A(y_1)=
\begin{pmatrix}
0&676y_1^2&55584y_1^3&1773408y_1^4&9838080y_1^5\\[3pt]
1&24y_1&2736y_1^2&89376y_1^3&506688y_1^4\\[3pt]
0&1&38y_1&1596y_1^2&9264y_1^3\\[3pt]
0&0&\frac{12}{7}&24y_1&\frac{1352}{7}y_1^2\\[3pt]
0&0&0&\frac72&0
\end{pmatrix}.
\]
}
\noindent
Grading operator:
\[
\mathrm{diag}\Big(\ell(w)-\frac12\dim X\Big)=\operatorname{diag}(-2,-1,0,1,2).
\]
At $y_1=1$ the characteristic polynomial factors as
\[
\chi_{A(1)}(\lambda)
=(\lambda+10)^3(\lambda^2-116\lambda-1244).
\]
\end{outputbox}

The eigenvalues are
\[
-10,\qquad58-48\sqrt2,\qquad58+48\sqrt2,
\]
with respective algebraic multiplicities $3,1,1$.
The full Hodge diamond of $X$ is
\[
\begin{array}{ccccccccc}
&&&&1&&&&\\
&&&0&&0&&&\\
&&0&&1&&0&&\\
&0&&0&&0&&0&\\
0&&16&&112&&16&&0\\
&0&&0&&0&&0&\\
&&0&&1&&0&&\\
&&&0&&0&&&\\
&&&&1&&&&
\end{array}.
\]
Since condition \textup{(AC)} holds, the ambient and vanishing
contributions (respectively \textcolor{blue}{blue} and
\textcolor{red}{red}) decompose as
\begin{equation}
\setlength{\arraycolsep}{3pt}
\begin{array}{ccccccccc}
&&&&\textcolor{blue}{1}&&&&\\
&&&0&&0&&&\\
&&0&&\textcolor{blue}{1}&&0&&\\
&0&&0&&0&&0&\\
0&&\textcolor{red}{16}&&
\mathclap{\textcolor{blue}{1}+\textcolor{red}{111}}
&&\textcolor{red}{16}&&0\\
&0&&0&&0&&0&\\
&&0&&\textcolor{blue}{1}&&0&&\\
&&&0&&0&&&\\
&&&&\textcolor{blue}{1}&&&&
\end{array}
\quad=\quad
\begin{array}{ccccccccc}
&&&&\textcolor{blue}{1}&&&&\\
&&&0&&0&&&\\
&&0&&\textcolor{blue}{1}&&0&&\\
&0&&0&&0&&0&\\
0&&0&&\textcolor{blue}{1}&&0&&0\\
&0&&0&&0&&0&\\
&&0&&\textcolor{blue}{1}&&0&&\\
&&&0&&0&&&\\
&&&&\textcolor{blue}{1}&&&&
\end{array}
\quad+\quad
\begin{array}{ccccccccc}
&&&&0&&&&\\
&&&0&&0&&&\\
&&0&&0&&0&&\\
&0&&0&&0&&0&\\
0&&\textcolor{red}{16}&&\textcolor{red}{111}&&\textcolor{red}{16}&&0\\
&0&&0&&0&&0&\\
&&0&&0&&0&&\\
&&&0&&0&&&\\
&&&&0&&&&
\end{array}.
\end{equation}

The eigenvalue of multiplicity three alone does not immediately
provide a rationality obstruction. To apply the
criterion of \cite{carazolli2026quantum}, we calculate the cyclic
determinant
\[
\Delta(1):=
\det\bigl(\mathbf1,A(1)\mathbf1,A(1)^2\mathbf1,
A(1)^3\mathbf1,A(1)^4\mathbf1\bigr)
=\frac{72}{7}\neq0.
\]
Thus $\mathbf1$ is cyclic. For a very general member,
the Hodge-fixed cohomology is precisely the ambient cohomology,
and Lemma~3.9 of \cite{carazolli2026quantum} identifies the
selected small-locus factor as generic.

The regularized exponents are
\begin{center}
\small
\renewcommand{\arraystretch}{1.35}
\begin{tabular}{c|c|c}
$\lambda$ of $A(1)$ & $m$ & Exponents mod $\mathbb Z$\\
\hline
$-10$ & $3$ & $\{0,(\frac12)^{\oplus2}\}$\\
$58-48\sqrt2$ & $1$ & $\{0\}$\\
$58+48\sqrt2$ & $1$ & $\{0\}$
\end{tabular}
\end{center}
In particular, the cyclic generic Hodge factor corresponding
to $\lambda=-10$ has invariants
\[
m=3,\qquad
F=\left\{\overline0,\overline{\frac12},
\overline{\frac12}\right\}.
\]
This is not among the three possibilities allowed for a rational
smooth projective fourfold by Theorem~4.5 of
\cite{carazolli2026quantum}:
\[
(m,F)\in\left\{
\left(1,\left\{\overline0 \right\}\right),
\left(2,\left\{ \overline{\frac12},\overline{\frac12}\right\}\right),
\left(3,\left\{\overline0,\overline0,\overline0 \right\}\right)
\right\}.
\]
We conclude:
\begin{theorem}
\label{thm:irrational_E_V7}
A very general smooth complete intersection of five hyperplanes
and a quadric in the spinor tenfold
$\operatorname{OGr}^{+}(5,10)$ is irrational.
\end{theorem}
\end{example}

\bibliographystyle{alpha}
\bibliography{min.bib}

%%%%%%%%%%%%%%%%%%%%%%%%%%%%%%%%%%%%%%%%%%%%%%%%%%%%%%%%%%%%%%%%%%%%%%%%%%%%%%%%%%%%%%%%%%%%%%%%%%%%%%%%%%%%%%%%%%%%%%%%%%%%%%%

%%%%%%%%%%%%%%%%%%%%%%%APENDICE%%%%%%%%%%%%%%%%%%%%%%

\appendix

\section{Homogeneous bundles over Grassmannians and flag varieties}
\label{app:bundles}
 
The localization formula of Theorem \ref{thm:localizationCI} takes as input,
for the bundle $\cE$ being cut, only its \emph{weights}: the characters by
which $\mathbb T$ acts on the fibers over the fixed points.  The purpose of
this appendix is to record those weights, together with the first Chern
classes needed to apply Theorem \ref{thm:ci}, for the bundles used in this
paper.  Everything follows from one computation, that of the tautological
bundle, together with the elementary behaviour of weights under linear
algebra operations.
 
\subsection*{Conventions}
 
Let $Y=G/P$ and let $\cE=G\times_{P}V$ be the homogeneous bundle associated
with a $P$-module $V$, with the convention
$(g,v)\sim(gp,f(p)^{-1}v)$ of Section~\ref{sec:picard} and the action
$t\cdot[g,v]=[tg,v]$.  We write $\mathrm{Wt}(\cE)$ for the multiset of
$\mathbb T$-weights of $V$.  By Proposition \ref{prop:linebundleweights}
these are the fiber weights of $\cE$ at the base point, and the fiber
weights at $p_{w}$ are the $w\nu$ with $\nu\in\mathrm{Wt}(\cE)$.  Two
consequences will be used constantly.
 
\begin{itemize}
\item A homogeneous line bundle with fiber weight $\nu$ at the base point is
      $\cO_{Y}(-\nu)$; if $\nu=\sum_{j}c_{j}\omega_{i_{j}}$ then, by
      Proposition \ref{prop:picard},
      $c_{1}=-\sum_{j}c_{j}\sigma_{s_{\alpha_{i_{j}}}}$.
\item In the notation of Proposition \ref{prop:twist}, where the fiber
      weights at $p_{w}$ were written $-w\mu_{i}$, one has
      $\mu_{i}=-\nu_{i}$ with $\nu_{i}\in\mathrm{Wt}(\cE)$. 
\end{itemize}
 
\subsection*{The tautological bundle}
 
Let $\Gr(k,n)=\SL_{n}(\C)/P_{k}$, and keep the notation of
Example \ref{ExRoots}: $e_{1},\dots,e_{n}$ is the canonical basis of
$\C^{n}$ and $\varepsilon_{i}$ is the character
$\varepsilon_{i}(\mathrm{diag}(a_{1},\dots,a_{n}))=a_{i}$, subject to
$\sum_{i}\varepsilon_{i}=0$.  Put $V_{k}=\mathrm{span}\{e_{1},\dots,e_{k}\}$,
whose stabilizer is $P_{k}$, so that $V_{k}$ is a $P_{k}$-module and
\[
  \mathcal U=\SL_{n}(\C)\times_{P_{k}}V_{k}
\]
is the \emph{tautological subbundle}, whose fiber over a point of the
Grassmannian is the corresponding $k$-plane.
 
For $t=\mathrm{diag}(a_{1},\dots,a_{n})\in\mathbb T$ the equivalence
$(\mathrm{\Id},v)\sim(t,f(t)^{-1}v)$ gives $(t,v)\sim(\mathrm{\Id},f(t)v)$,
whence
\[
  t\cdot[\mathrm{\Id},e_{i}]=[t,e_{i}]=[\mathrm{\Id},f(t)e_{i}]
  =\varepsilon_{i}(t)\,[\mathrm{\Id},e_{i}],
\]
so that
\[
  \mathrm{Wt}(\mathcal U)=\{\varepsilon_{1},\dots,\varepsilon_{k}\}.
\]
This is what one expects: the fiber of $\mathcal U$ at the base point
\emph{is} $V_{k}$, on which the torus acts with weights
$\varepsilon_{1},\dots,\varepsilon_{k}$.
 
Consequently $\det\mathcal U$ has fiber weight
$\varepsilon_{1}+\dots+\varepsilon_{k}=\omega_{k}$ at the base point, so by
the dictionary above $\det\mathcal U=\cO_{Y}(-\omega_{k})=\cO(-1)$ and
\[
  c_{1}(\mathcal U)=-\sigma_{s_{\alpha_{k}}},
  \qquad
  c_{1}(\mathcal U^{\vee})=c_{1}(\mathcal Q)=\sigma_{s_{\alpha_{k}}} ,
\]
the last because $\mathcal Q=(\C^{n}\otimes\cO_{Y})/\mathcal U$ has
$c_{1}(\mathcal Q)=-c_{1}(\mathcal U)$.
 
\subsection*{Operations}
 
Let $E$ and $F$ be homogeneous bundles of ranks $r$ and $s$ with weights
$\nu_{1},\dots,\nu_{r}$ and $\eta_{1},\dots,\eta_{s}$, and let
$\varphi\colon Z\to Y$ be equivariant.  Weights and first Chern classes
behave as follows.
 
\[
\begin{array}{c|c|c}
\text{Operation} & \text{Weights} & c_{1}
\\ \hline
E\otimes F & \nu_{i}+\eta_{j} & s\,c_{1}(E)+r\,c_{1}(F)
\\[2mm]
E^{\vee} & -\nu_{i} & -c_{1}(E)
\\[2mm]
\Lambda^{m}E
& \nu_{i_{1}}+\cdots+\nu_{i_{m}},\quad i_{1}<\cdots<i_{m}
& \displaystyle\binom{r-1}{m-1}c_{1}(E)
\\[3mm]
\operatorname{Sym}^{m}E
& \displaystyle a_{1}\nu_{1}+\cdots+a_{r}\nu_{r},
  \quad a_{i}\ge0,\ \textstyle\sum_{i}a_{i}=m
& \displaystyle\binom{r+m-1}{m-1}c_{1}(E)
\\[4mm]
E/S & \operatorname{Wt}(E)\smallsetminus\operatorname{Wt}(S) & c_{1}(E)-c_{1}(S)
\\[2mm]
\varphi^{*}E & \operatorname{Wt}(E)\ \text{at}\ \varphi(z),\ z\in Z^{\mathbb T}
& \varphi^{*}c_{1}(E)
\end{array}
\]
 
Here $S\subset E$ is a homogeneous subbundle, and the difference of weights
is a difference of multisets.  The Chern class entries follow from the
weight entries by summing, since $c_{1}$ of a homogeneous bundle is $c_{1}$
of its determinant.
 
\subsection*{The bundles used in this paper}
 
On $\Gr(k,n)$ we use $\mathcal U^{\vee}$, the universal quotient
$\mathcal Q$, and their exterior and symmetric powers.  All of these are
generated by global sections: $\mathcal U^{\vee}$ and $\mathcal Q$ are, and taking exterior and symmetric
powers of bundles generated by global sections preserves this property.
 
On a partial flag variety $\mathrm{Fl}(k_{1},\dots,k_{\ell},n)$ we use, for
the projections $\pi_{i}\colon\mathrm{Fl}(k_{1},\dots,k_{\ell},n)\to\Gr(k_{i},n)$,
the pullbacks $\mathcal U^{\vee}_{k_{i}}=\pi_{i}^{*}\mathcal U^{\vee}$ and
$\mathcal Q_{k_{i}}=\pi_{i}^{*}\mathcal Q$. All are also generated by global sections.

\begin{table}[H]
\centering
\renewcommand{\arraystretch}{1.5}
\begin{tabular}{c|c|c}
\hline
Bundle & Fiber weights at the base point & $c_{1}$ \\
\hline
$\mathcal{U}$
& $\varepsilon_{1},\dots,\varepsilon_{k}$
& $-\sigma_{s_{\alpha_{k}}}$
\\ \hline
$\mathcal{U}^{\vee}$
& $-\varepsilon_{1},\dots,-\varepsilon_{k}$
& $\sigma_{s_{\alpha_{k}}}$
\\ \hline
$\mathcal{Q}$
& $\varepsilon_{k+1},\dots,\varepsilon_{n}$
& $\sigma_{s_{\alpha_{k}}}$
\\ \hline
$\Lambda^{r}\mathcal{U}^{\vee}$
& $-\big(\varepsilon_{a_{1}}+\cdots+\varepsilon_{a_{r}}\big),
   \quad 1\le a_{1}<\cdots<a_{r}\le k$
& $\displaystyle\binom{k-1}{r-1}\sigma_{s_{\alpha_{k}}}$
\\ \hline
$\operatorname{Sym}^{r}\mathcal{U}^{\vee}$
& $-\big(m_{1}\varepsilon_{1}+\cdots+m_{k}\varepsilon_{k}\big),
   \quad m_{a}\ge0,\ \sum_{a}m_{a}=r$
& $\displaystyle\binom{k+r-1}{r-1}\sigma_{s_{\alpha_{k}}}$
\\ \hline
$\Lambda^{r}\mathcal{Q}$
& $\varepsilon_{j_{1}}+\cdots+\varepsilon_{j_{r}},
   \quad k+1\le j_{1}<\cdots<j_{r}\le n$
& $\displaystyle\binom{n-k-1}{r-1}\sigma_{s_{\alpha_{k}}}$
\\ \hline
$\operatorname{Sym}^{r}\mathcal{Q}$
& $m_{k+1}\varepsilon_{k+1}+\cdots+m_{n}\varepsilon_{n},
   \quad m_{a}\ge0,\ \sum_{a}m_{a}=r$
& $\displaystyle\binom{n-k+r-1}{r-1}\sigma_{s_{\alpha_{k}}}$
\\ \hline
$\mathcal{U}^{\vee}_{k_{i}}$
& $-\varepsilon_{1},\dots,-\varepsilon_{k_{i}}$
& $\sigma_{s_{\alpha_{k_{i}}}}$
\\ \hline
$\mathcal{Q}_{k_{i}}$
& $\varepsilon_{k_{i}+1},\dots,\varepsilon_{n}$
& $\sigma_{s_{\alpha_{k_{i}}}}$

\\ \hline
\end{tabular}
\caption{Weights and first Chern classes.  The fiber weights at $p_{w}$ are
obtained by applying $w$ to those listed. The weights $\mu_{i}$ of
Proposition \ref{prop:twist} are their negatives.}
\end{table}
 
As a check, $\det\mathcal Q$ has weight
$\varepsilon_{k+1}+\cdots+\varepsilon_{n}=-\omega_{k}$, so
$\det\mathcal Q=\cO_{Y}(\omega_{k})=\cO(1)$, in agreement with
$c_{1}(\mathcal Q)=-c_{1}(\mathcal U)$.

\medskip

\section{Two sentences on Bruhat order}
\label{app:bruhat}

Bruhat order enters this paper as follows: the restriction of an equivariant Schubert class to a fixed point vanishes unless the two
indices are comparable.  We collect here the definition and the few properties we use.  Nothing is proved, and the exposition is far from complete. Check \cite{Humphreys1990}*{Ch.~5} for a complementary exposition.

\begin{definition}\label{def:bruhat}
Let $W$ be the Weyl group with simple reflections $\Pi$.  For $u,v\in W$ we write $u\le v$ if some reduced expression $v=s_{v_{1}}\cdots s_{v_{\ell}}$
admits a subword $s_{v_{i_{1}}}\cdots s_{v_{i_{k}}}$, with $i_{1}<\dots<i_{k}$, which is a reduced expression for $u$.  This is the \emph{Bruhat order} on $W$.
\end{definition}

\begin{proposition}\label{prop:subword}
The relation of Definition \ref{def:bruhat} is a partial order, and it does
not depend on the chosen reduced expression for $v$: if $u\le v$ holds for
one reduced expression of $v$, it holds for every one.  Moreover $u\le v$
implies $\ell(u)\le\ell(v)$, with equality only for $u=v$. The minimum is
$\Id$ and the maximum is the longest element $w_{0}$.
\end{proposition}

\begin{remark}[Geometric meaning]\label{rmk:bruhatgeom}
With the conventions of Definition \ref{def:schubert}, where
$\Omega_{w}=\overline{B^{-}wP/P}$ has codimension $\ell(w)$, one has
\[
  u\le v \iff \Omega_{u}\supseteq\Omega_{v},
\]
so the Bruhat order is the containment order of Schubert varieties, reversed
relative to the order of their codimensions. In particular, the fixed point
$p_{v}$ lies in $\Omega_{u}$ when $u\le v$.
\end{remark}

\begin{corollary}\label{cor:billeyvanishing}
$\sigma^{u}\big|_{p_{v}}=0$ unless $u\le v$.
\end{corollary}

\subsection*{The parabolic case}

Our Schubert classes are indexed by $W_{Y}$, the minimal length
representatives of $W/W_{P}$, so we need the induced order there.

\begin{proposition}\label{prop:bruhatparabolic}
Let $\overline{w}\in W_{Y}$ denote the minimal representative of $wW_{P}$.
Then
\begin{defenum}
\item $u\le v$ in $W$ implies $\overline{u}\le\overline{v}$ in $W_Y$. The projection $W\to W_{Y}$ is order preserving;
\item for $u,v\in W_{Y}$, the order induced on $W_{Y}$ from $W$ coincides
      with the Bruhat order of the quotient $W/W_{P}$, namely
      $uW_{P}\le vW_{P}$ if and only if $u'\le v'$ for some, equivalently
      for all minimal representatives.
\end{defenum}
\end{proposition}

%===============================================================================================================================================================================================================================================================================

\medskip

\section{On the software}
\label{app:software}

In this appendix, we give precise statements describing what the software computes and explain why the assumptions of Theorem \ref{thm:ci} provide an ideal setting for its application. We then discuss how the software can be used when these assumptions are not satisfied. Finally, we explain how to provide the necessary input data to the code.

\medskip

In order to compute Gromov--Witten invariants for a flag variety $F$,  the software has implemented the localization formula proved in Theorem~\ref{thm:localization}. The inputs for the calculation are the cohomology classes, given in the Schubert cell description presented in Section~\ref{sec:flags}, and an effective class $\beta \in \rH_{2}(F)$, given in terms of the coroots as in Proposition~\ref{prop:h2}.

For a complete intersection $X:= \mathscr{Z}(F,\mathcal{E})$, the formula implemented is the one proved in Theorem~\ref{thm:localizationCI}. All bundles that can be used are presented in Appendix~\ref{app:bundles} with their respective weights.

The inputs required to compute the sum of the Gromov--Witten invariant are the cohomology classes coming from the ambient $\rH_{F}(X)$ and given in the Schubert cell description, and a homology class $\beta \in \rH_2(F)$, in the coroots. 

In Section ~\ref{def:ambient-anticanonical-operator}, we have explained that the software computes the matrix
\[
A(q)\sigma_i
=
\sum_{\beta\in\mathsf{NE}(F)}
\sum_{j,k}
\left(
\sum_{\substack{
\beta'\in\mathsf{NE}(X)\\
\iota_*\beta'=\beta}}
\big\langle
c_1(TX),\sigma_i,\sigma_j
\big\rangle^X_{0,\beta'}
\right)
q^\beta g^{jk}\sigma_k.
\]
where
\[
g_{jk}
=
\sum_{\substack{\beta\in H_2(X,\mathbb Z)\\ \iota_*\beta=0}}
\left\langle \sigma_j,\sigma_k,1\right\rangle_{0,\beta}
\]
and $g^{jk}$ denotes the inverse matrix components to $(g_{jk})$. Notice that the Novikov variables are indexed in the second homology of the flag variety, not of the complete intersection.

Under the assumptions of Theorem \ref{thm:ci}, we have an isomorphism $\rH^2 (F,\mathbb{Q}) \cong \rH^2 (X,\mathbb{Q})$. Consequently, an isomorphism $\rH_2 (F,\mathbb{Z}) \cong \rH_2 (X,\mathbb{Z})$. Thus, Theorem \ref{thm:localizationCI} computes actual Gromov--Witten invariants and not only sums of such. Therefore, the matrix $A(q)$ defined above is the matrix of $c_1(TX) \star_{\mathrm{small}}$.

% First it computes the matrix of the Poincaré pairing. The entries $g_{i,j}$ are obtained using the Gromov--Witten invariant calculations.
%\[
%g_{i,j} = \int_{X} \sigma_i \cup \sigma_{j} = \left\langle \sigma_i , \sigma_j , 1 \right\rangle_{3,0}
%\]
%where $\sigma_i,\sigma_j$ are the Schubert cells coming from the ambient $F$. Then, it computes the matrix $M$ with entries given by
%\[
%M_{i,j} = \sum_{\beta}\left( \int_{\beta} c_1(TX) \right) \left\langle \sigma_i ,\sigma_j  \right\rangle_{2,\beta} \cdot q^{\beta}.
%\]
%Finally, it computes the operator $A(q) = g \cdot M$.

When the map $\iota_*:\rH_2 (X,\mathbb{Z}) \rightarrow \rH_2 (F,\mathbb{Z})$ fails to be an isomorphism, one should keep track of its behavior to understand the differences of the matrices $A(q)$ and $c_1 (X)\star_{\mathrm{small}}$.

\medskip
%Under the assumption of Theorem \ref{thm:ci} these formulas compute exactly the matrix of $c_1 (X) \star_{small}$, if not, one should keep track of the behavior of the map $\iota_* H_2(X,\mathbb{Z}) \rightarrow H_2 (F, \mathbb{Z})$ to be able to make the necessary changes on the software output, to obtain the matrix of $c_1 (X) \star_{small}$.

Lastly, we briefly describe how the software can be used.

\medskip
 
 The inputs required for the computation are the following: the algebra, which specifies the group $G$; the simple roots kept, which determine the parabolic subgroup $P$; and the expression of the bundle, which specifies the vector bundle.

To specify a flag variety $F=G/P$, the code requires two pieces of information. First, one must specify the type of the semisimple Lie algebra $\mathfrak g$ of $G$. This is given by its Dynkin diagram:
\[
A_n,\quad B_n,\quad C_n,\quad D_n,\quad E_6,\quad E_7,\quad E_8,\quad F_4,\quad G_2.
\]
Thus, for example, the input ``A2'' corresponds to the root system of type $A_2$, while ``G2'' corresponds to the root system of type $G_2$. For a semisimple algebra with more than one simple factor, the simple components are separated by ``x''. Thus, ``A2xA2'' represents the product $\SL_3(\mathbb{C}) \times \SL_3 (\mathbb{C})$.

Second, we must specify the parabolic subgroup $P$. Let
\[
\Pi={\alpha_1,\ldots,\alpha_n}
\]
denote the set of simple roots of $G$, with the numbering determined by the chosen Dynkin diagram convention. A standard parabolic subgroup $P$ is determined by a subset
\[
\Pi_P\subseteq \Pi.
\]
The code records the complementary set of simple roots, that is,
\[
\Pi\setminus\Pi_P={\alpha_{i_1},\ldots,\alpha_{i_k}},
\]
by specifying the corresponding indices $(i_1,\ldots,i_k)$ in the input of the kept simple roots. Thus, one enters
\[
i_1,\ldots,i_k
\]
in the ``kept simple roots'' field.

Lastly, we need to specify the vector bundle under consideration. We do so in the ``bundle K'' field. For sums of line bundles a reduced expression can be used. The line bundle
\[
\mathcal{O}(a_1,\ldots,a_k)
\]
is encoded by
\[
a_1,\ldots,a_k.
\]
Different summands are separated by a semicolon ``;''. Thus, the vector bundle
\[
\mathcal{O}(a_1,\ldots,a_k)\oplus
\mathcal{O}(b_1,\ldots,b_k)
\]
is entered as
\[
a_1,\ldots,a_k;b_1,\ldots,b_k.
\]

Alternatively, one can write more general vector bundles on the algebra $A_n$ and powers of it. In addition to line bundles, the software supports the tautological bundle $\mathcal{U}$ and the quotient bundle $\mathcal{Q}$ on Grassmannians as well as their pullbacks to the space of flags $\mathrm{Fl}( m_1,\ldots, m_k )$ or products of Grassmannians (in particular, products of projective spaces). As well, it supports operations with these bundles.\\
The table below summarizes the notation used by the software and illustrates how the corresponding bundle operations should be encoded.

\begin{table}[H]
\centering
\begin{tabular}{c|c}
\hline
\textbf{Bundle} & \textbf{Code notation} \\
\hline
$\mathcal{O}(a_1 ,\ldots, a_k)$
  & \texttt{O$(a_1,\ldots, a_k$)} \\[2mm]

$\mathcal{U}$ with regards to the simple root $\alpha_{node}$
  & \texttt{S(node)} \\[2mm]

$\mathcal{Q}$ with regards to the simple root $\alpha_{node}$
  & \texttt{Q(node)} \\[2mm]

Dual bundle
  & \texttt{dual(E)} \\[2mm]

Direct sum
  & \texttt{osum(E$_1$,E$_2$,\ldots)} \\[2mm]

Tensor product
  & \texttt{tensor(E$_1$,E$_2$)} \\[2mm]

Symmetric power
  & \texttt{sym(p,E)} \\[2mm]

Exterior power
  & \texttt{wedge(p,E)} \\[2mm]

 Quotients
  & \texttt{quot(E,F)} \\
\hline
\end{tabular}
\caption{Notation for homogeneous bundles used in the code.}
\label{tab:bundle-code}
\end{table}
Note that not all constructions will lead to vector bundles generated by global sections.

\bigskip

\bigskip

\end{document}